\documentclass[11pt,parskip=half]{scrartcl}
\usepackage[a4paper,hmargin={2.5cm,2.5cm},vmargin={2.5cm,2.5cm},heightrounded, marginparwidth=2.2cm, marginparsep=0.1cm]{geometry}

\usepackage[utf8]{inputenc}
\usepackage[T1]{fontenc}

\usepackage[leqno]{amsmath}
\usepackage{amssymb,amsthm}
\usepackage{tikz-cd}
\usepackage[cmtip]{xypic}

\theoremstyle{plain}
\newtheorem{theorem}{Theorem}[section]
\newtheorem{lemma}[theorem]{Lemma}
\newtheorem{proposition}[theorem]{Proposition}
\newtheorem{corollary}[theorem]{Corollary}

\theoremstyle{definition}
\newtheorem{definition}[theorem]{Definition}
\newtheorem{examples}[theorem]{Examples}
\newtheorem{example}[theorem]{Example}

\newtheorem{problem}[theorem]{Problem}
\theoremstyle{remark}
\newtheorem{remark}[theorem]{Remark}
\newtheorem*{remark*}{Remark}
\newtheorem{remarks}[theorem]{Remarks}
\newtheorem{facts}[theorem]{Facts}

\RequirePackage[shortlabels]{enumitem}
\counterwithin*{equation}{section}

\newcommand\triple{\rotatebox[origin=c]{270}{$\Rrightarrow$}}
\newlist{tfae}{enumerate}{1}%
\setlist[tfae,1]{label=(\roman*)}%

\newcommand{\catfont}[1]{\mathsf{#1}}

\newcommand{\SET}{\catfont{Set}}

\newcommand{\ORD}{\catfont{Ord}}

\newcommand{\MET}{\catfont{Met}}

\newcommand{\Cats}[1]{#1\text{-}\catfont{Cat}}
\newcommand{\CATs}[1]{#1\text{-}\catfont{CAT}}

\newcommand{\Setss}[1]{\catfont{Set}{/\!/}#1}
\newcommand{\Catss}[1]{\catfont{Cat}{/\!/}#1}
\newcommand{\CATss}[1]{\catfont{CAT}{/\!/}#1}
\newcommand{\nSetss}[1]{\catfont{Set}{||}#1}

\newcommand{\ncatfont}[1]{\mathbb{#1}}
\newcommand{\ncatA}{\ncatfont{A}}
\newcommand{\ncatX}{\ncatfont{X}}
\newcommand{\ncatY}{\ncatfont{Y}}
\newcommand{\ncatOne}{\ncatfont{E}}

\newcommand{\COCTS}[1]{#1\text{-}\catfont{COCTS}}
\newcommand{\Cocts}[1]{#1\text{-}\catfont{Cocts}}

\DeclareMathOperator{\Nat}{Nat}
\DeclareMathOperator{\ncolim}{ncolim}

\DeclareMathAlphabet{\mathpzc}{OT1}{pzc}{m}{it}
\DeclareMathOperator{\yoneda}{\mathpzc{y}}

\newcommand{\Psh}{\mathcal{P}}

\newcommand{\op}{\mathrm{op}}

\newcommand{\two}{\catfont{2}}
\newcommand{\V}{\mathcal{V}}

\newcommand{\kk}{\mathrm k}
\newcommand{\nv}{|\!|}

\RequirePackage{dsfont}
\newcommand{\NN}{\mathbb{N}}

\RequirePackage{mathtools}

\makeatletter

\def\slashedarrowfill@#1#2#3#4#5{%
  $\m@th\thickmuskip0mu\medmuskip\thickmuskip\thinmuskip\thickmuskip
   \relax#5#1\mkern-7mu%
   \cleaders\hbox{$#5\mkern-2mu#2\mkern-2mu$}\hfill
   \mathclap{#3}\mathclap{#2}%
   \cleaders\hbox{$#5\mkern-2mu#2\mkern-2mu$}\hfill
   \mkern-7mu#4$%
}

\newcommand*{\rightmodarrowfill@}{\slashedarrowfill@\relbar\relbar{\raisebox{0pc}{$\hspace{1pt}\circ$}}\rightarrow}
\newcommand*{\xmodto}[2][]{\ext@arrow 0055{\rightmodarrowfill@}{\;#1\;}{\;#2\;}}
\newcommand*{\lmodto}{\xmodto{\phantom{\to}}}

\makeatother

\usepackage[auth-lg]{authblk}

\title{Cauchy convergence in $\V$-normed categories}

\author[1]{Maria Manuel Clementino\thanks{Partially supported by {\em Centro de Matemática da Universidade de Coimbra} (CMUC), funded by the Portuguese Government through FCT/MCTES, DOI 10.54499/UIDB/00324/2020.}}

\author[2]{Dirk Hofmann\thanks{Partially supported by the Center for Research and Development in Mathematics and Applications (CIDMA) through the Portuguese Foundation for Science and Technology (FCT -- {\em Fundação para a Ciência e a Tecnologia}), DOI 10.54499/UIDB/04106/2020 and DOI 10.54499/UIDP/04106/2020.}}

\author[3]{Walter Tholen\thanks{Partially supported by the Natural Sciences and Engineering Council of Canada under the Discovery Grants Program, no. 501260.}}

\affil[1]{University of Coimbra, CMUC, Department of Mathematics, Portugal, \texttt{mmc@mat.uc.pt}}

\affil[2]{University of Aveiro, CIDMA, Department of Mathematics, Portugal, \texttt{dirk@ua.pt}}

\affil[3]{York University, Toronto, Canada, \texttt{tholen@yorku.ca}}

\usepackage[hypertexnames=false]{hyperref}

\hypersetup{
  colorlinks = true,
  citecolor= [rgb]{0,0.75,0}, 
  urlcolor=[rgb]{0.4,0,0.4}, 
  linkcolor=[rgb]{0.8,0,0} 
}

\begin{document}

\maketitle
\begin{abstract}{\small
Building on the notion of normed category as suggested by Lawvere, we introduce notions of Cauchy convergence and cocompleteness for such categories that differ from proposals in previous works. Key to our approach is to treat them consequentially as categories enriched in the monoidal-closed category of normed sets, {\em i.e.}, of sets which come with a norm function. Our notions lead to the anticipated outcomes for individual metric spaces or the additive groups of  normed vector spaces considered as small normed categories. But they quickly become more challenging when we consider large categories, such as the categories of all semi-normed or normed vector spaces and all linear maps as morphisms, not just because norms of vectors need to be allowed to have value $\infty$ in order to guarantee the existence of colimits of (sufficiently many) infinite sequences.
These categories, along with categories of generalized metric spaces,  are the key example categories discussed in detail in this paper.

Working with a general commutative quantale $\mathcal V$ as a value recipient for norms, rather than only with Lawvere's quantale $\mathcal R_+$ of the extended real half-line, we observe that the categorically atypical, and perhaps even irritating, structure gap between objects and morphisms in the example categories is already present in the underlying normed category of the enriching category of $\mathcal V$-normed sets. To show that the normed category and, in fact, all presheaf categories over it, are Cauchy cocomplete, we assume  the quantale $\mathcal V$ to satisfy a couple of light alternative extra properties which, however, are satisfied in all instances of interest to us.
Of utmost importance to the general theory is the fact that our notion of Cauchy convergence is subsumed by the notion of weighted colimit of enriched category theory. With this theory and, in particular, with results of Albert, Kelly and Schmitt, we are able to prove that all $\mathcal V$-normed categories have correct-size Cauchy cocompletions, for $\mathcal V$ satisfying our light alternative assumptions.

We also show that our notions are suitable to prove a Banach Fixed Point Theorem for contractive endofunctors of Cauchy cocomplete normed categories.}
\end{abstract}

{\small{\em Mathematics Subject Classification:} 18A35, 18D20; 18F75, 46M99, 54E35.

{\em Keywords:} quantale, (Lawvere) metric space, normed set, normed group, (semi-)normed vector space, normed category, normed functor, Cauchy sequence, normed colimit, weighted colimit, Cauchy cocomplete category, Cauchy cocompletion, Banach Fixed Point Theorem.}

{\footnotesize \tableofcontents}

\section{Introduction}\label{sec:some-title}

Considering the elements of a set $X$ as the objects of the indiscrete category $\mathrm iX$, all of whose hom-sets are singletons, we may regard a (generalized) metric $d:X\times X\to[0,\infty]$ as a function that assigns to the only morphism from $x$ to $y$ a ``norm'', $d(x,y)$. Then the point (in)equality and the triangle inequality
\begin{equation}\label{Lawveremetric}
0\geq d(x,x)\qquad\text{and}\qquad  d(x,y)+d(y,z)\geq d(x,z)
\end{equation}
tell us how this norm interacts with the structure of the category. For an arbitrary category $\ncatX$, rather than $\mathrm iX$, this leads to the notion of {\em norm} on $\mathbb X$, as a function that assigns
to every morphism $f:x\to y$ a value $|f|\in[0,\infty]$, such that 
\begin{equation}\label{R+norm}
0\geq |1_x|\qquad\text{and}\qquad|f|+|g|\geq|g\cdot f|
\end{equation}
 hold for all morphisms $g:y\to z$  in $\ncatX$. Furthermore, we may
increase the potential range of examples and applications by allowing norms to take values in an arbitrary (commutative and unital) {\em quantale}, {\em i.e.}, in a complete lattice $(\V,\leq)$ which, in addition, has a commutative monoid structure $(\V,\otimes,\kk)$, such that $\otimes$ distributes over arbitrary joins in each variable. We may therefore consider {\em $\V$-normed categories} $\ncatX$ where $(\V,\leq,\otimes,\kk)$ will, amongst others, take on the role of the {\em Lawvere quantale} $\mathcal R_+=([0,\infty],\geq,+,0)$, so that the conditions (\ref{R+norm}) read {\color{red} in $\V$} as
\begin{equation}\label{Vnorm}
\kk\leq |1_x|\qquad\text{and}\qquad
|f|\otimes|g|\leq|g\cdot f|\;.
\end{equation}

 Hence, we follow Lawvere's \cite{Lawvere73} original concept of normed category, as a category enriched in a certain symmetric monoidal-closed category, in this paper taken to be the category $\Setss{\V}$ of $\V$-{\em normed sets}. Its objects are mere sets equipped with a $\V$-valued function; morphisms are maps which keep or increase the $\V$-value of elements. The monoidal structure on $\Setss{\V}$ is such that, for the (always assumed to be small) hom-sets of a category $\ncatX$ to be $\V$-normed and to satisfy the conditions (\ref{Vnorm}) amounts precisely to making $\ncatX$ a category enriched in $\Setss{\V}$.

 Of course, a norm with values in the terminal quantale $\mathsf 1$ adds no structure to a category $\ncatX$, but already a norm valued in the {\em Boolean quantale} $\mathsf 2=(\{\top,\bot\},\Rightarrow,\wedge,\top)$ does. It is determined by 
 the class $\mathcal S$ of all morphisms $f$ with $|f|=\top$ or, equivalently, $\top\Rightarrow |f|$, satisfying the conditions
 \begin{equation}\label{2normasS}
1_x\in\mathcal S\qquad\text{and}\qquad f\in\mathcal S\;\wedge\;g\in\mathcal S\Longrightarrow g\cdot f\in\mathcal S\;.
\end{equation}
In other words, such $\mathsf 2$-valued normed categories are just categories $\ncatX$ that come equipped with a distinguished wide subcategory $\mathcal S$ ({\em i.e.,} a subcategory with the same class of objects as $\mathbb X$). While our principal interest concerns $\mathcal R_+$-normed categories and, to a lesser extent, $\mathsf 2$-normed categories, we occasional resort to other quantales, be it only to demonstrate that the general theory of $\V$-normed categories doesn't rely on special properties that the Boolean quantale and the Lawvere quantale may share, such as being integral (so that the $\otimes$-neutral element is the top element), or being based on a completely distributive lattice.
Readers interested in even greater generality, so that $\V$ may be an arbitrary symmetric monoidal-closed category $\V$ rather than just a quantale, may want to consult \cite{BettiGaluzzi75}.

For $\V=\mathcal R_+$ we normally suppress the reference to $\V$. Many of the large, and sufficiently interesting, normed categories have objects with some metric structure which, however, is hardly, or not at all, respected by the morphisms. But the metric structure of the objects may then be used to specify classes of well-behaved morphisms. For instance, let $\mathbb X$ be the (large) category $\mathsf{NVec}_{\infty}$ whose objects are all normed real vector spaces in the usual sense, except that we allow norms to assume the value $\infty$ (see further below for some justification), along with a natural adjustment of the real arithmetic for this value, and whose morphisms are {\em all} linear maps ({\em i.e.}, the linear $\infty$-Lipschitz maps). From a standard categorical perspective, forming this category appears to be  highly questionable since it makes two objects $X$ and $Y$ isomorphic as soon as they are algebraically isomorphic, regardless of their norms. (In fact, with a choice of a basis for every space granted, $\mathbb X$ becomes equivalent to the category of all real vector spaces and their linear maps.) Nevertheless, $\mathbb X$ has a {\em raison d\!'\,$\hat{e}$tre} when regarded as a {\em normed} category, as it may allow us to investigate morphisms of interest within the same category, such as the (uniformly) continuous maps, or the maps with a given Lipschitz value $\geq 1$. Indeed, for a linear map $f:X\to Y$, with $|\!|\cdot|\!|$ denoting the given norms of vectors in $X$ and $Y$, 
writing $\log^{\circ}\alpha:=\max\{0,\log \alpha\}$ when $\alpha >0$, under a natural extension of the real arithmetic to $\infty$ one simply considers
\begin{equation}\label{NVecnorm}
|f|=\sup_{x\in X}\,\log^{\circ}\frac{\nv fx\nv}{\nv x\nv}\;,
\end{equation}
so that $|f|$ becomes minimal in $[0,\infty]$ with respect to the natural order and the property that
$$e^{|f|}|\!|x|\!|\geq|\!|fx|\!|$$
holds for all $x\in X$. This makes $e^{|f|}$ the {\em Lipschitz value} $\mathrm L(f)$ of the map $f$ whenever $\mathrm L(f)\geq 1$, so that the condition $|f|<\infty$ characterizes $f$ as bounded (or, equivalently, as (uniformly) continuous), while $|f|=0$ describes it as non-expanding, or 1-Lipschitz.

Just as for metric spaces, the concept of {\em Cauchy convergence} should be fundamental in the study of normed categories. But what is it? And once defined, what does completeness mean? Do there exist completions, and are there protagonistic normed categories in this context, like the presheaf categories in the completion theory of ordinary categories?
In this paper we try to give answers to these questions and test them in examples. Taking seriously the enriched categorical perspective that is already present in \cite{Lawvere73, BettiGaluzzi75, Lawvere86}, our answers differ from those presented in other papers, such as \cite{Kubis17, Neeman20, Insall21}. 

For any quantale $\V$, seeing the study of $\V$-normed categories as embedded into enriched category theory, we must pay close attention to the enriching symmetric monoidal-closed category $\Setss{\V}$ of $\V$-normed sets and $\V$-normed maps $f:A\to B$ (satisfying $|a|\leq|f a|$ for all $a\in A$), with the tensor product making the cartesian product $\V$-normed in the obvious way. 
Despite its simplicity, $\Setss{\V}$ has as an unexpected feature: the internal hom of objects $A$ and $B$ is given by {\em all} mappings $A\to B$,  not just by the ($\Setss{\V}$)-morphisms $A\to B$. Hence, when one considers $\Setss{\V}$ as a $\V$-normed category, {\em i.e.} as a ($\Setss{\V}$)-enriched category, we obtain a category that has the same objects as $\Setss{\V}$, but whose morphisms are mere mappings, without any constraint vis-\`{a}-vis the norms of their domain or codomain. To avoid confusion, we use a different name for this category: $\nSetss{\V}$; its $\V$-norm gives a measure to which extent a mapping of $\V$-normed sets may (fail to) be $\V$-normed.
It contains the ordinary category $\Setss{\V}$ as a non-full subcategory. In Kelly's \cite{Kel82} notation for the underlying ordinary category of an enriched category, one has
$$\Setss{\V}=(\nSetss{\V})_{\circ}\,.$$
As the recipient category for the presheaves over any given $\V$-normed category $\ncatX$, the $\V$-normed category $\nSetss{\V}$ is key to the study of any kind of completions of $\ncatX$.

For this reason, in Sections 2-4 we take time to present the fundamentals of $\V$-normed category theory in detail. Alongside many examples of such small and large categories, in Theorem \ref{CatVproperties} we summarize the properties of the category $\Catss{\V}$ of all small $\V$-normed categories and $\V$-normed functors  most important to us: it is  {\em symmetric monoidal closed} \cite{Kel82, MacLane98}, as well as
 {\em topological} \cite{AHS90,HST14} over $\Setss{\V}$, which gives the recipe for the construction of  limits and colimits of $\V$-normed categories. That $\Catss{\V}$ is also
  {\em locally presentable} \cite{GU71, AR94} is shown in an appendix (Section 13). Other than $\nSetss{\V}$, we introduce  at the general $\V$-level the $\V$-normed categories $\V\text{-}\mathsf{Lip}$ and $\V\text{-}\mathsf{Dist}$, both having as their objects small $\V$-categories, {\em i.e.}, Lawvere metric spaces when $\V=\mathcal R_+$, whilst their morphisms are respectively arbitrary maps and $\V$-distributors. The former category facilitates the study of norms for categories of metric spaces and of normed vector spaces, and the norm of the latter category naturally leads to non-symmetrized Hausdorff metrics (as considered in \cite{Lawvere73} and studied in \cite{ACT10, Stubbe10}).

In Sections 5 and 6, introducing the key notions of {\em Cauchy sequence} and of {\em normed convergence} of a sequence in a $\V$-normed category, we tighten the corresponding definitions as proposed by Kubi\'{s} \cite{Kubis17} in the context of $\V=\mathcal R_+$, in such a way that,
unlike in Kubi\'{s}'s work, normed colimits become unique and conform with the enriched setting. Our categorical notion of {\em Cauchy cocompleteness}\footnote{The term, or its dual, is not to be confused with Cauchy completeness in the sense of {\em idempotent completeness}. To see how the self-dual notion of idempotent completeness fits with our term, see the appended Section 14.},  fits with classical completeness concepts for some key small normed categories:
\begin{itemize}
\item A classical metric space $X$ is complete if, and only if, the normed category $\mathrm iX$, with the only morphism in a hom-set  normed by the metric, is Cauchy cocomplete (Corollary \ref{CauchycocoforiX});
\item a classical normed vector space is Banach if, and only if, its additive group, considered as a one-object category with the given vector norm, is Cauchy cocomplete (Theorem \ref{BanachasCauchycoco}).
\end{itemize} 

We consider the key examples of large normed categories in Sections 7 and 8 and show:
\begin{itemize}
\item The category $\MET_{\infty}$ of all Lawvere metric spaces (with the generalized metric satisfying just (\ref{Lawveremetric})) and arbitrary maps as morphisms is Cauchy cocomplete (Corollary \ref{MetCauchycocomplete});
\item also the category $\mathsf{SNVec}_{\infty}$ of generalized semi-normed vector spaces (non-zero vectors may have norm $0$ or $\infty$) with arbitrary linear operators normed by (\ref{NVecnorm}) (Theorem \ref{seminormedtheorem}) is Cauchy cocomplete, but its full subcategory $\mathsf{NVec}_{\infty}$ of generalized normed vector spaces (norm $\infty$ is still permitted for non-zero vectors, but not norm $0$) is not (Proposition \ref{NVecCauchycocomplete}). 
\end{itemize}
For the $\MET_{\infty}$ result, employing the methods used in Flagg's pioneering work \cite{Fla92, Flagg97} and, more recently, in \cite{HN20}, we establish more generally the Cauchy cocompleteness of the $\V$-normed category
 $\V\text{-}\mathsf{Lip}$, under additional conditions on the quantale $\V$ (Theorem  \ref{VLiptheorem}). Expanding on previous work (see in particular \cite{Kubis17, Perrone21}), 
  we are then led naturally to proving the claims regarding first $\mathsf{SNVec}_{\infty}$  and then $\mathsf{NVec}_{\infty}$.
There are good reasons for allowing infinite vector norms, as well as zero norms for non-zero vectors. They arise already in elementary contexts when studying infinite sequences of operators, no matter how good the initial data may be.  Indeed, regarding non-separation,  consider the sequence
\begin{equation}\label{sequenceofreals}
\xymatrix{\mathbb R=\mathbb R_1\ar[r] & \mathbb R_{\frac{1}{2}}\ar[r] & \mathbb R_{\frac{1}{3}}\ar[r] & ...\ar[r] & \mathrm{colim}_n\;\mathbb R_{\frac{1}{n}}\,,\\
}\end{equation}
where $\mathbb{R}_c$ is the $1$-dimensional vector space of real numbers normed by $\nv 1\nv=c$ with a constant $c>0$. Its connecting identity maps are strictly contractive, and it has a rather natural normed colimit (in our sense) in the category $\mathsf{SNVec}_{\infty}$, namely $\mathbb R_0$ where we now permit $c=0$, but the sequence has no  normed colimit in $\mathsf{NVec}_{\infty}$ (Proposition \ref{NVecnotCauchycocomplete}). If we turn around the above sequence while inverting also the norms (to still have ``nice'' contractive operators), we obtain the {\em inverse sequence}
$$\xymatrix{\mathbb R=\mathbb R_1 & \mathbb R_{2}\ar[l] & \mathbb R_{3}\ar[l] & ...\ar[l] & \mathrm{lim}_n\;\mathbb R_{n}\ar[l]\,.\\
}$$
Its {\em normed limit}  (dual to normed colimit) is $\mathbb R_{\infty}$, {\em i.e.} now $c=\infty$, in both  $\mathsf{SNVec}_{\infty}$ and $\mathsf{NVec}_{\infty}$.

In Section 9 we present the key element of our study of $\V$-normed categories, understood to be enriched in the monoidal-closed category $\Setss{\V}$, and prove that its presheaf categories are Cauchy cocomplete, under a couple of alternative mild conditions on the quantale $\V$. (Theorem \ref{d:thm:1}). As indicated above, these are the functor categories $[\ncatX,\nSetss{\V}]$ of $\V$-normed functors of a small $\V$-normed category $\ncatX$ into the $\V$-normed category $\nSetss{\V}$ associated with $\Setss{\V}$. As shown in the appended Section 15, the two alternative conditions on $\V$ are logically independent, but we are actually not aware of any quantale for which $\nSetss{\V}$ fails to be Cauchy cocomplete.

Then, in Sections 10 and 11,
we exhibit our normed colimits as {\em weighted} colimits in the sense of enriched category theory, and invoke the machinery developed by Albert, Kelly and Schmitt \cite{AK88, KS05} to show the existence of a Cauchy cocompletion of $\ncatX$ belonging to the same universe as the given $\V$-normed category $\ncatX$. Unlike these results, it is easy to see that our notions lead to known concepts and outcomes, for instance when applied to individual (Lawvere) metric spaces seen as small normed categories; see in particular \cite{BBR98, Vickers05, HofmannTholen10, HofmannReis13}. We leave to future work the question whether the methods used in these and other papers may be generalized to produce a more direct construction of the Cauchy cocompletion of a normed category than the one presented here.

Finally, following a claim by Kubi\'{s} \cite{Kubis17}, in Section 12 we present an easily established {\em Banach Fixed Point Theorem} for a contractive endofunctor of a Cauchy cocomplete normed category which replicates the classical theorem in the case of a complete metric space, considered as a small normed category.

Normed categories are often called {\em weighted} (see \cite{Grandis07, BSS17, Perrone21}). We avoid this change of the original terminology, mostly to be able to distinguish between our notion of normed colimit and the established notion of weighted colimit of enriched category theory when we prove (in Section 10) the non-trivial fact that the former notion is subsumed by the latter. Also, we refrain from imposing any further {\em a priori} conditions on the notion (as done in \cite{Kubis17, Insall21}) but discuss these as special add-on properties of Lawvere's fundamental notion. For the treatment of norms in the special environment  of triangulated categories, we refer to \cite{Neeman20}.

 {\em Acknowledgements:} We are indebted to the anonymous referee for many valuable suggestions which led us to a substantial revision and improvement of an earlier version of this paper. We thank Javier Guti\'{e}rrez-Garc\'{i}a for his assistance in finding a quantale that helps  distinguishing the alternative conditions used in Section 9; see Section 15 for details. We also thank Paolo Perrone for providing access to the paper \cite{BettiGaluzzi75}, as well as Marino Gran who made possible our joint stay at Louvain-la-Neuve prior to CT 2023 in June 2023. Following the third-named author's talks given in April/May of 2022, joint work on this paper was initiated at that occasion, with some key results presented by the second-named author in a talk given in October 2023.

\section{$\V$-normed sets}
Throughout this paper $\V=(\V,\leq,\otimes,\kk)$ is a unital and commutative {\em quantale}, that is: $(\V,\leq)$ is a complete lattice and $(\V,\otimes,\kk)$ is a commutative monoid with neutral element $\kk$, such that, for all $v\in\V$, the map $-\otimes v:\V\to \V$ preserves arbitrary suprema:
$(\bigvee_{i\in I}u_i)\otimes v=\bigvee_{i\in I}(u_i\otimes v);$
in particular, $\bot\otimes v=\bot$ for the bottom element $\bot$ of $\V$.  Hence, $-\otimes v$  has a right adjoint, $[v,-]:\V\to\V$, for every $v\in \V$, making $\V$ a
 (small and thin) symmetric monoidal-closed category, with its internal hom characterized by
$$u\leq[v,w]\iff u\otimes v\leq w$$
for all $u,v,w\in \V$. The standard quantales considered in this paper are the {\em Boolean quantale}, $\two=\{\bot,\top\}$ with $\otimes=\wedge$ and $\kk=\top$, and the {\em Lawvere quantale}, $\mathcal R_+=([0,\infty],\geq,+,0)$, ordered by the natural $\geq$-order of the extended real half-line. In $\mathcal R_+$, the internal hom is computed as
$[v,w]=\max\{0,w-v\},\; [v,\infty]=\infty,\;[\infty,w]=[\infty,\infty]=0$ for all $v,w<\infty$, and in $\mathsf 2$ it is given by the implication: $[v,w]=(v\Rightarrow w)$.

\begin{definition}\label{normedsetdef}
A $\V$-{\em normed set} is a set $A$ that comes with a function $|\text{-}|_A:A\to\V$, and a $\V$-{\em normed map} $(A,|\text{-}|_A)\to(B,|\text{-}|_B)$ is a mapping $f:A\to B$ satisfying $|a|_A\leq|fa|_B$ for all $a\in A$:
\begin{displaymath}
\xymatrix{A\ar[rr]^f\ar[rd]_{|\text{-}|_A}^{\quad\leq} && B\ar[ld]^{|\text{-}|_B}\\
& \V & \\
}
\end{displaymath}
Henceforth, we usually drop the subscripts.
This defines the category $\Setss{\V}\;.$
\end{definition}

This category is simply the formal coproduct completion $\mathrm{Fam}\V$ of the category $(\V,\leq)$. Applying the
$\mathrm{Fam}$-construction to the unique functor $\V\to\mathsf 1$ of $\V$ to the terminal quantale  one obtains the forgetful functor \(\Setss{\V}\to\SET\), which is {\em topological} \cite{AHS90}; that is: given a family of any size of mappings $f_i:A\to B_i\;(i\in I)$ with a fixed set $A$ and all $B_i$ $\V$-normed, then there is an ``initial'' $\V$-norm on $A$, namely
\begin{equation}
|a|=\bigwedge_{i\in I}|f_ia|.
\end{equation}
Equivalently: given any family of mappings $g_i:A_i\to B\;(i\in I)$ from $\V$-normed sets $A_i$ to a given set $B$, then there is a ``final'' $\V$-norm on $B$ that is described by
\begin{equation}\label{finalSetVstructure}
|b|=\bigvee_{i\in I}\bigvee_{a\in g_i^{-1}b}|a|.
\end{equation}

Consequently, \(\Setss{\V}\) is complete and cocomplete. Moreover, the forgetful functor has a left adjoint, putting on every set the {\em discrete $\V$-norm}  with constant value $\bot$, as well as a right adjoint, putting on every set the {\em indiscrete $\V$-norm} with constant value $\top$. In particular, \(\Setss{\V}\to\SET\) is represented by the discrete singleton $\V$-normed set $\mathrm E_{\bot}$, {\em i.e.}, by $\{*\}$ with $|*|=\bot$.

More importantly, one has:

\begin{proposition}\label{SetV}
The category \(\Setss{\V}\) is symmetric monoidal closed.
\end{proposition}

\begin{proof}
For $\V$-normed sets $A$ and $B$, their tensor product $A\otimes B$ is carried by the cartesian product $A\times B$, normed by $|(a,b)|=|a|\otimes|b|$ in $\V$, and the tensor-neutral set $\mathrm E_{\mathrm k}$ is the set $\{*\}$ normed by $|{*}|=\kk$.

To determine the internal hom-object, denoted by $[A,B]$, we first observe that its elements are equivalently described by the $\V$-normed maps $\mathrm E_{\bot}\to[A,B]$, which must correspond to the $\V$-normed maps $\mathrm E_{\bot}\otimes A\to B$. But these correspond precisely to arbitrary $\SET$-maps $A\to B$, since $\mathrm E_{\bot}\otimes A$ puts just the discrete structure on the set $A$. Consequently, $[A,B]$ has carrier set $\SET(A,B)$, {\em i.e.}, the set of {\em  all} mappings $\varphi: A\to B$, with their norm  defined by
\begin{equation}\label{inthomnorm}
 |\varphi|=\bigwedge_{a\in A}[|a|,|\varphi a|].
 \end{equation}
 This turns out to be the largest structure (in the order induced pointwise by $\V$) making the evaluation map $[A,B]\otimes A\to B$ $\V$-normed; {\em i.e.}, $|\varphi|$ is maximal with the property
$$|\varphi|\otimes |a|\leq|\varphi a|$$ for all $a\in A$.
\end{proof}

\begin{remarks}\label{firstremarks}
(1) We note that, for $\varphi \in[A,B]$, one has $\kk\leq|\varphi|$ precisely when $|a|\leq|\varphi a|$ for all $a\in A$, that is, when $\varphi: A\to B$ is a $\V$-normed map. Hence, $|\varphi|$ is to be seen as the extent to which the arbitrary map $\varphi$ may (fail to) be a morphism of \(\Setss{\V}\).

(2) When we consider the lattice $\V$ as a small thin category, the functor $\mathsf 1\to \SET$ of the terminal category $\mathsf 1$ pointing to the terminal object $\{*\}$ of $\SET$ ``lifts'' to the functor $\mathrm i: \V\to\ \Setss{\V}$, which assigns to every $v\in\V$ the set $\mathrm E_v=\{*\}$, normed by $|\!*\!|=v$. It has a left adjoint, $\mathrm s$, which assigns to every object $A$ its ``sum'', or ``supremum'' $\mathrm sA=\bigvee_{a\in A}|a|$, also regarded as its ``optimal value'' \cite{Perrone21}. In the commutative diagram
\begin{equation}\label{VinSetV}
\xymatrix{\V\ar[d]\ar[rr]^{\mathrm i} &&  \Setss{\V}\ar[d]^{\mathrm{forget}}\ar@/^1.0pc/[ll]^{\mathrm s}_{\top}\\
\mathsf 1\ar[rr] && \SET\ar@/^1.0pc/[ll]_{\top}\\
}
\end{equation}
all arrows are monoidal homomorphisms; they, except possibly $\mathrm s$, preserve also the internal homs.

(3) As a left adjoint, the functor $\mathrm s$ preserves all colimits, and it also preserves products if (and only if) the lattice $\V$ is completely distributive.

(4) Other than the forgetful functor $\Setss{\V}\to\SET$ as in (\ref{VinSetV}), one may, for every $v\in \V$, consider more generally the functor $P_v:\Setss{\V}\to\SET$ which assigns to a $\V$-normed set $A$ the set $\{a\in A\mid v\leq|a|\}$. It has a left adjoint which puts the $\V$-norm with constant value $v$ onto every set, and it is represented by the $\V$-normed set $\mathrm E_v$ as defined in (2). The set of objects
$\{\mathrm E_v\mid v\in\V\}$
distinguishes itself as being a {\em strong generator} of the category $\Setss{\V}$. Indeed, for any $\V$-normed set $B$, the family of all morphisms $\mathrm E_v\to B$ with some $v\in \V$ is not only jointly epic, but in fact strongly so, since $B$ carries the final structure (as described by (\ref{finalSetVstructure})) with respect to this family. The importance of the strong generator lies in the fact that it makes the cocomplete category actually {\em locally presentable} \cite{GU71}: see
 Section 13.
\end{remarks}

\section{$\V$-normed categories}
\begin{definition}
A $\V$-{\em normed category} $\mathbb X$ is a (\(\Setss{\V}\))-enriched category. This means that $\mathbb X$ is  an ordinary category with (small) $\V$-normed hom-sets such that, for all objects $x,y,z$, the maps
$$\mathrm E_{\kk}\to \mathbb X(x,x),\;*\mapsto 1_x,\qquad\text{and}\qquad
\mathbb X(x,y)\otimes\mathbb X(y,z)\to \mathbb X(x,z),\; (f,g)\mapsto g\cdot f,$$
are $\V$-normed; equivalently, all morphisms $f:x\to y$ and $g:y\to z$  satisfy the conditions (\ref{Vnorm}).
A functor $F:\mathbb X\to \mathbb Y$ is $\V$-{\em normed} if it makes its hom maps $\mathbb X(x,y)\to\mathbb Y(Fx,Fy)$ $\V$-normed; that is, if
$|f|\leq |Ff|$
holds for all morphisms $f$ in $\mathbb X$. A natural transformation $\alpha:F\Longrightarrow G$ is $\V$-{\em normed} if $\kk\leq|\alpha_x|$ for all objects $x$ in $\ncatX$. We use respectively the abbreviations
\begin{displaymath}
  \Catss{\V}=\Cats{(\Setss{\V})}
  \quad\text{and}\quad
 \CATss{\V}=\CATs{(\Setss{\V})}
\end{displaymath}
for the emerging 2-category of all small $\V$-normed categories with their $\V$-normed functors and $\V$-normed natural transformations, and for its (higher-universe) counterpart of all $\V$-normed categories. In case $\V=\mathcal R_+$ we normally suppress the prefix $\V$.
\end{definition}

\begin{facts}\label{secondremarks}
  (1) Considering the monoid $(\V,\otimes,\kk)$ as a one-object 2-category with its 2-cells given by the order of $\V$, we may describe a $\V$-normed category $\mathbb X$  equivalently as a 2-category with only identical 2-cells, equipped with a lax functor $|\text{-}|:\mathbb X\to\V$. A $\V$-normed functor $F:\mathbb X\to\mathbb Y$ is then a (lax but, in fact, necessarily strict) 2-functor producing the lax-commutative diagram
  \begin{center}
    $\xymatrix{\mathbb X\ar[rr]^F\ar[rd]_{|\text{-}|_{\mathbb X}}^{\quad\leq} && \mathbb Y\ar[ld]^{|\text{-}|_{\mathbb Y}}\\
      & \V & \\
    }$
  \end{center}

(2) The (monoidal) functors of diagram (\ref{VinSetV}) 
induce the diagram
\begin{equation}\label{VCatinCatV}
\xymatrix{
\V\text{-}\mathsf{Cat}\ar[d]_{\mathrm{forget}}\ar[rr]^{\mathrm i} &&  \Catss{\V}\ar[d]^{\mathrm{forget}}\ar@/^1.0pc/[ll]^{\mathrm s}_{\top}\\
\SET\ar[rr]^{\mathrm{indiscrete}} && \mathsf{Cat}\ar@/^1.0pc/[ll]_{\top}^{\mathrm{ob}}\\
}
\end{equation}
of change-of-base functors. Here an object of $\V\text{-}\mathsf{Cat}$ is (as in \cite{Lawvere73} and \cite{HST14}) a set $X$ which, for all $x,y\in X$, comes with a value $X(x,y)\in \V$, satisfying the laws
\begin{equation}\label{VCatconditions}
\kk\leq X(x,x)\qquad\text{and}\qquad X(x,y)\otimes X(y,z)\leq X(x,z).
\end{equation}
The functor $\mathrm i$ describes the $\V$-category $X$ equivalently as an indiscrete category $\mathbb X=\mathrm iX$ with $\mathrm{ob}\mathbb X=X$, putting the $\V$-norm $|x\to y|=X(x,y)$ on the only morphism in $\mathbb X(x,y)$. The functor $\mathrm s$ takes an arbitrary small $\V$-normed category $\mathbb X$ to the $\V$-category $\mathrm s\mathbb X=\mathrm{ob}\mathbb X$ with
\begin{equation}
(\mathrm s\mathbb X)(x,y)=\bigvee\{|f|\;|\;f\in\mathbb X(x,y)\}.
\end{equation}

(3) The norm-forgetting functor $\Catss{\V}\to\mathsf{Cat}$ must be  carefully distinguished from the functor $$(-)_{\circ}:\Catss{\V}\to \mathsf{Cat}$$ which sends a small $\V$-normed category $\mathbb X$ to the category $\mathbb X_{\circ},$ defined (as in enriched category theory \cite{Kel82}) to have the same objects as $\mathbb X$, but the morphisms of which are only those morphisms $f:x\to y$ in $\mathbb X$ with $\kk\leq |f|$ (since these are equivalently described by the $(\Setss{\V})$-morphisms $\mathrm E_{\kk}\to\mathbb X(x,y))$. Extending the terminology of \cite{Kubis17} from $\mathcal R_+$ to arbitrary $\V$, we call the morphisms of $\mathbb X_{\circ}$ the $\kk$-{\em morphisms} of $\mathbb X$, and we say that
the (ordinary and generally non-full) subcategory $\mathbb X_{\circ}$ of $\mathbb X$ is the {\em category of $\kk$-morphisms in $\mathbb X$}. An isomorphism $f$ in $\mathbb X_{\circ}$ is called a {\em $\kk$-isomorphism} of $\mathbb X$; {\em i.e.}, $f$ is an isomorphism in $\mathbb X$ such that both, $f$ and $f^{-1}$, are $\kk$-morphisms.

{\em Caution:} An isomorphism in the ordinary category $\mathbb X$ may not belong to $\mathbb X_{\circ}$, and even if it does, it may not be an isomorphism in $\mathbb X_{\circ}$: for a (non-symmetric) two-point metric space $X=\{a,b\}$ with $X(a,b)=1$ and all other distances $0$, just consider $\mathbb X=\mathrm iX$, as formed in Examples \ref{exsofnormedcats}(2).

(4) Being of the form $\mathcal W\text{-}\mathsf{Cat}$ for some symmetric monoidal-closed category $\mathcal W$, all four categories of diagram (\ref{VCatinCatV}) are again symmetric monoidal closed (if not cartesian closed). In particular, recall that the tensor product of $X$ and $Y$ in $\V\text{-}\mathsf{Cat}$ is carried by the cartesian product and structured by
\begin{equation}
(X\otimes Y)((x,y),(x',y'))=X(x,x')\otimes Y(y,y'),
\end{equation}
and that their internal hom, $[X,Y]$, is carried by the hom-set $\V\text{-}\mathsf{Cat}(X,Y)$ and  structured by
\begin{equation}
[X,Y](f,g)=\bigwedge_{x\in X}Y(fx,gx).
\end{equation}
Of course, $\mathrm E=\{*\}$ with $\mathrm E(*,*)=\kk$ is the monoidal unit in $\V\text{-}\mathsf{Cat}$ (see \cite{HST14} for details).
\end{facts}

We define the {\em tensor product} $\mathbb X\otimes\mathbb Y$ of the $\V$-normed categories $\mathbb X$ and $\mathbb Y$ to be carried by the ordinary category $\mathbb X\times \mathbb Y$, structured by
\begin{equation}
|(f,f')|=|f|\otimes|f'|.
\end{equation}
One then routinely shows that their {\em internal hom} $[\mathbb X,\mathbb Y]$ is given by the $\V$-normed functors $\mathbb X\to \mathbb Y$ and {\em all} natural transformations between them, normed by
\begin{equation}\label{nattransnorm}
|\alpha|=\bigwedge_{x\in \mathrm{ob}\mathbb X}|\alpha_x|.
\end{equation}
The terminal category $\mathbb E$ in $\mathsf{Cat}$ with $\mathrm{ob}\mathbb E=\mathrm E=\{\ast\}$ becomes the {\em monoidal unit} when one puts $|1_*|=\kk$ (but note that it is terminal in  $\Catss{\V}$ only if $\kk=\top$). Clearly, the functors $\mathrm i$ and $\mathrm s$ preserve the monoidal structure, and $\mathrm i$ preserves even the closed structure.

We also observe that the forgetful functor $\Catss{\V}\to\mathsf{Cat}$ is, like $\V\text{-}\mathsf{Cat}\to\SET$,  topological; that is: for any  (arbitrarily large) family $F_i:\mathbb X\to\mathbb Y_i\;(i\in I)$ of functors of a fixed category $\mathbb X$ to $\V$-normed categories, there is the  ``initial'' $\V$-normed structure on $\mathbb X$, given by
\begin{equation}\label{initialnorm}
|f|=\bigwedge_{i\in I}|F_if|.
\end{equation}

Let us summarize the main points of these observations, as follows:

\begin{theorem}\label{CatVproperties}
The  2-category  $\Catss{\V}$ of (small) $\V$-normed categories, $\V$-normed functors and $\V$-normed natural transformations is symmetric monoidal closed and topological over $\mathsf{Cat}$. In particular, $ \Catss{\V}$ is complete and cocomplete, and the forgetful functor $ \Catss{\V}\to\mathsf{Cat}$ has both, a right and a left adjoint. 
The restriction to small $\V$-normed categories whose carrier is indiscrete reproduces the corresponding statements for the category $\V\text{-}\mathsf{Cat}$ and its forgetful  functor to $\SET$.
\end{theorem}

 Here is a first list of easy examples; others follow in Sections 4--9.

\begin{examples}\label{exsofnormedcats}{($\V=\mathsf 1,\two, \mathcal R_+$)}
(1) A $\mathsf 1$-normed category (for the terminal quantale $\mathsf 1$) is just an ordinary category, and for $\V=\mathsf 1$  diagram (\ref{VCatinCatV}) flattens to
$\mathrm s\dashv\mathrm i:\Cats{\mathrm 1}=\SET\longrightarrow\mathsf{Cat}=\Catss{\mathsf 1}.$

 For the Boolean quantale $\V=\two=(\{\top,\bot\},\Rightarrow,\wedge,\top) $, diagram (\ref{VCatinCatV}) takes the form
\begin{center}
$\xymatrix{
\ORD\ar[d]_{\mathrm{forget}}\ar[rr]^{\mathrm i} &&  \Catss{\two}\ar[d]^{\mathrm{forget}}\ar@/^1.0pc/[ll]^{\mathrm s}_{\top}\\
\SET\ar[rr]^{\mathrm{indiscrete}} && \mathsf{Cat}\ar@/^1.0pc/[ll]_{\top}^{\mathrm{ob}}\\
}$
\end{center}
Here $\ORD=\Cats{\mathsf 2}$ is the category of preordered sets and monotone maps. Objects in $\Catss{\two}$ may be described as small categories $\mathbb X$ which come with a distinguished class  $\mathcal S$ of morphisms that is closed under composition and contains all identity morphisms; see (\ref{2normasS}). Necessarily then, as a category, $\mathcal S=\mathbb X_{\circ}$ as defined in Facts \ref{secondremarks}(3). Morphisms in $\Catss{\two}$ are functors preserving the distinguished morphisms.

(2) For the Lawvere quantale $\V=\mathcal R_+=([0,\infty],\geq,+,0)$,  in a normed category $\mathbb X$, we normally write the norm conditions (\ref{R+norm}) with the natural $\leq$: 
$$\quad |1_x|=0\qquad\text{and}\qquad|g\cdot f|\leq |f|+|g|$$
for all $f:x\to y$ and $g:y\to z$. A normed functor $F:\mathbb X\to\mathbb Y$ must be non-expanding: $|Ff|\leq|f|$ for all morphisms $f$ in $\mathbb X$.
Every Lawvere metric space $X=(X,d)$ (defined to satisfy just (\ref{Lawveremetric})) gives equivalently, via Facts \ref{secondremarks}(2) and (\ref{VCatconditions}), a small indiscrete normed category when we write $X(x,y)$ for $d(x,y)$; likewise for non-expanding maps.This describes the full reflective embedding
$$\mathrm i: \mathsf{Met}_1:=\Cats{\mathcal R_+}\longrightarrow\mathsf{NCat}_1:=\Catss{\mathcal R_+},$$
 with the subscript $1$ indicating the Lipschitz constant defining the morphisms. Its left adjoint $\mathrm s$ provides the set $X$ of all objects of a small normed category $\mathbb X$ with the (Lawvere) metric\footnote{Our use of $\inf$ and $\sup$  refers to the standard order of the reals; they become $\bigvee$ and $\bigwedge$ in the quantale ${\mathcal R}_+$.}
$$X(x,y)=\inf\{|f|\mid f\in\mathbb X(x,y)\}.$$
(3) (Lawvere \cite{Lawvere73}) The subsets of a (Lawvere) metric space $X$ are the objects of the small normed category $\mathbb HX$ whose morphisms $\varphi:A\to B$ are arbitrary $\SET$-maps, normed (as in (\ref{inthomnorm}))) by
\begin{equation}\label{HXnorm}
|\varphi|=\sup_{x\in A}X(x,\varphi x).
\end{equation}
The reflector $\mathrm s\vdash\mathrm i$ provides the powerset of $X$ with the non-symmetrized Hausdorff metric
\begin{equation}\label{Hausdorffmetric}
d(A,B)=\inf_{\varphi:A\to B}\sup_{x\in A}X(x,\varphi x)=\sup_{x\in A}\inf_{y\in B}X(x,y),
\end{equation}
where the validation of the second equality (presenting the metric in its more usual form) requires the Axiom of Choice.
 Whereas the assignment $X\longmapsto \mathbb HX$ is not functorial, with Choice the assignment $X\longmapsto\mathrm s(\mathbb{H}X)$ does extend to a functor $\MET_1\to \MET_1$, as we show more generally in Corollary \ref{Hausdorffnormcor}.

(4)  (See also \cite{Kubis17, Perrone21}.) Here is a norm that measures the degree to which an arbitrary mapping between metric spaces fails to be 1-Lipschitz ({\em i.e.}, fails to be non-expanding). Just form the (somewhat strange) category $\mathsf{Met}_{\infty}$ whose objects are Lawvere metric spaces, and whose morphisms $\varphi:X\to Y$ are mere $\SET$-maps, normed by
\begin{equation}\label{Metnormformula}
|\varphi|=\sup_{x,x'\in X}\log^{\circ}\frac{ Y(\varphi x,\varphi x')}{X(x,x')},
\end{equation}
where we have used the abbreviation $\log^{\circ}\alpha=\max\{0,\log\alpha\}$ for $\alpha\in[0,\infty]$ and extended the real arithmetic to $[0,\infty]$, the details of which are given in  Examples \ref{realswithmult}(2) and before Corollary \ref{MetCauchycocomplete}.
If $X$ is a metric space in the classical sense, then this extension may be largely avoided since
$|\varphi|=\log^{\circ}\mathrm L(\varphi),$
where
\begin{equation}\label{Lipschitzvalue}
\mathrm L(\varphi)=\sup\,\{\frac{ Y(\varphi x,\varphi x')}{X(x,x')}\mid x,x'\in X,\; X(x,x')\neq 0\}
\end{equation}
 is the {\em Lipschitz value} of $\varphi$ in $[0,\infty]$.
Since the $0$-morphisms in
the normed category $\mathsf{Met}_{\infty}$, {\em i.e.}, the morphisms $\varphi$ with $|\varphi|=0$, are precisely the $1$-Lipschitz maps, we have
$(\mathsf{Met}_{\infty})_{\circ}=\mathsf{Met_1}\;.$

If $X$ and $Y$ are the underlying metric spaces of normed vector spaces and $\varphi$ is linear, then (with $\nv\text{-}\nv$ denoting the given norms of the vector spaces), (\ref{Metnormformula}) reads equivalently as (\ref{NVecnorm}):
$|\varphi|=\sup_{x\neq 0}\,\log^{\circ}(\frac{\nv \varphi x\nv}{\nv x\nv})\;.$
 Hence, for $\varphi$ 1-Lipschitz, $e^{|\varphi|}=\nv\varphi\nv$ is the usual operator norm of $\varphi$.
 \end{examples}

 The underlying lattices of the quantales $\V=\mathsf 1,\two, \mathcal R_+$
 are all linearly ordered and, in particular, {\em completely distributive}. These quantales are also {\em integral, i.e.}, the quantalic unit is the top element. Examples of some interest with quantales not enjoying at least one of these properties follow.

 \begin{examples}\label{otherexamples}{(Other quantales)}
(1) For any commutative monoid $(M,+,0)$ one has the {\em free} quantale $(\mathcal PM,\subseteq, +,
\{0\})$ over the monoid $M$, with the powerset of $M$ structured by $A+B=\{a+b\mid a\in A,\, b\in B\}$ for all $A,B\subseteq M$. The lattice $\mathcal PM$ is still completely distributive, but the quantale, unless $M$ is trivial,  drastically fails to be {\em integral} since the tensor-neutral element $\{0\}$ is actually an atom.  A $\mathcal PM$-normed category $\mathbb X$ may be described as a category $\ncatX$ equipped with a family $(\mathcal S_{\alpha})_{\alpha\in M}$ of classes of morphisms satisfying $\mathrm{Id}(\ncatX)\subseteq \mathcal S_0$ and $\mathcal S_{\beta}\cdot\mathcal S_{\alpha}\subseteq\mathcal S_{\alpha+\beta}$ for all $\alpha,\beta\in M$. Only the class $\mathcal S_0$ is secured to be a subcategory of $\ncatX$ and, in fact, takes on the role of $\ncatX_{\circ}$ (see Facts \ref{secondremarks}(3)). 
A $\mathcal PM$-normed functor must preserve the distinguished classes at each level $\alpha\in M$.
We note that the trivial monoid $\{0\}$ returns the case $\V=\mathsf 2$ of Examples \ref{exsofnormedcats}(1).

(2) Let the commutative monoid $(M,+,0)$ be cancellative (so that $+$ is injective in each variable) and be ordered discretely. Then, forming its MacNeille completion $M_{\bot}^{\top}$, by adding to it a bottom and a top element, one obtains a non-integral quantale by letting the tensor coincide with $+$ within $M$ and putting $\top\otimes \alpha =\top$ for all $\alpha\neq\bot$ (see \cite{GH24}). If $M$ has at least 3 elements, the lattice $M_{\bot}^{\top}$ fails to be distributive. The structure of an $M_{\bot}^{\top}$-normed category $\ncatX$ whose norm function doesn't attain the values $\bot$ or $\top$ is given by a functor $\ncatX\to M$ to the one-object category $M$. Via its fibres, it may, as in (1), be described equivalently by a family $(\mathcal S_{\alpha})_{\alpha\in M}$ satisfying the {\em additional} condition that its non-empty members form a partition of the entire morphism class of $\ncatX$.
For example, for a directed graph $G=(E,V)$ that comes with an arbitrary (weight) function $w:E\to \mathbb N$ (of the edges to the additive non-negative integers), the free category of finite paths in $G$ carries a unique ${\mathbb N}_\bot^\top$-norm that coincides with $w$ on paths of length $1$.

(3) For every topological space $X$ one has the complete Boolean algebra $\mathsf{RO}(X)$ of regular open sets in $X$, giving the integral quantale $(\mathsf{RO}(X),\subseteq,\cap,X)$. For $X$ Hausdorff without isolated points, $\mathsf{RO}(X)$ is atomless, and if $X$ is even metrizable and separable, one has $\mathsf{RO}(X)\cong \mathsf{RO}(\mathbb{R})$, which is known to fail complete distributivity (\cite{Halmos63}). Following the slogan ``open sets are semi-decidable properties'' (\cite{Spreen98}), for certain spaces $X$ one may consider a $\mathsf{RO}(X)$-valued norm on a category as providing every morphism with a range of ``programs'' executing the morphism. A suitable such space may be given by the set $P$ of all {\em finite partial} functions $p:\mathbb N\to \mathbb N$, ordered (as relations on the set $\mathbb N$) by reverse inclusion and topologized by its down-closed sets as opens. Again the lattice $\mathsf{RO}(P)$ fails to be completely distributive; see \cite[p. 214]{Koppelberg89}. 
\end{examples}

\section{The $\V$-normed categories $\nSetss{\V},\;\V\text{-}\mathsf{Lip},$ and $\V\text{-}\mathsf{Dist}$}
In this section, for every quantale $\V$, we present three large $\V$-normed categories. For the first one, we note that 
every symmetric monoidal-closed category $\mathcal W$ becomes a $\mathcal W$-enriched category with the same objects, {\em qua} its internal hom; see, for example, \cite{Borceux94}. Exploiting this fact for $\mathcal W=\Setss{\V}$ we obtain a $\V$-normed category whose objects are $\V$-normed sets, but whose morphisms $A\to B$ are given by the internal hom-objects $[A,B]$ of $\Setss{\V}$, {\em i.e.,} by {\em all} $\SET$-maps $A\to B$. As emphasized already in the Introduction, the emerging normed category must be carefully distinguished from the category $\Setss{\V}$. As it plays an important role in what follows,  it deserves a separate notation,
$ \nSetss{\V},$
not to be confused with its (generally non-full) subcategory $\Setss{\V}$. For clarity, with the proof of Proposition \ref{SetV} and the norm (\ref{inthomnorm}), we summarize these points, as follows.

\begin{proposition}
  The category $\nSetss{\V}$ of $\V$-normed sets with arbitrary mappings as morphisms becomes a $\V$-normed category with
   $ |A\xrightarrow{\varphi}B|
    =\bigwedge_{a\in A}[|a|,|\varphi a|]$.
  In the notation and terminology of {\em Facts \ref{secondremarks}(3)}, the ordinary category $\Setss{\V}$ is precisely the category of $\kk$-morphisms of the $\V$-normed category $\nSetss{\V}$; that is:
  $  (\nSetss{\V})_{\circ}=\Setss{\V}.$
\end{proposition}

\begin{remarks}\label{VisVnormed}
(1)  As an ordinary category, $\nSetss{\V}$ is equivalent to the category $\SET$. The separate notation is justifiable only because $\nSetss{\V}$ is regarded as a $\V$-normed category.

(2) The monoid $(\V,\otimes,\kk)$ may be regarded as a one-object $\V$-normed category, with an identical norm function. As the monoid acts on itself, we obtain a functor
$$\lambda: \V\longrightarrow \nSetss{\V},\quad u\longmapsto (\lambda_u:\V\to\V, \,v\mapsto u\otimes v),$$
where $\V$, as a domain and codomain of the (left-)translation $\lambda_u$, is regarded just as an identically $\V$-normed set. The functor $\lambda$ is $\V$-normed, actually norm-preserving:
$ |\lambda_u|=\bigwedge_{v\in\V}[v,u\otimes v]=u.$
\end{remarks}

 In generalization of Example \ref{exsofnormedcats}(4), next we consider another category in which morphisms are not required to respect the structure of the objects: the objects of  the category $\V\text{-}\mathsf{Lip}$ are small $\V$-categories (see Facts \ref{secondremarks}(2)), with arbitrary maps as morphisms (so that, as an ordinary category, $\V$-$\mathsf{Lip}$ is actually equivalent to $\SET$ again, as in Remarks \ref{VisVnormed}(1)). Their $\V$-norm measures to which extent they may fail to be $\V$-functors, as follows.
 \begin{proposition}\label{Lipschitznorm}
When we define the {\em Lipschitz $\V$-norm}  of a mapping $\varphi: X\to  Y$ between small $\V$-categories by letting $|\varphi|$ be maximal with $|\varphi|\otimes X(x,x')\leq Y(\varphi x\varphi x')$ for all $x,x'\in X$, so that:
   \begin{equation}\label{VLipnorm}
  |\varphi|= \bigwedge_{x,x'\in X}[X(x,x'),Y(\varphi x,\varphi x')],
  \end{equation}
then we obtain the $\V$-normed category $\V\text{-}\mathsf{Lip}$
  with $(\V\text{-}\mathsf{Lip})_{\circ}=\Cats{\V}$.
 
 Furthermore, the forgetful functor
 $\V\text{-}\mathsf{Lip}\longrightarrow \nSetss{\V},\; X\longmapsto X\times X,$
defined to remember just that every $\V$-category $X$ comes with a function $X\times X\to\V$, is not only $\V$-normed but actually norm-preserving. Restricting it to $\kk$-morphisms gives a faithful functor $\Cats{\V}\longrightarrow\Setss{\V}$.
\end{proposition}

\begin{proof}
For arbitrary maps $\varphi :X\to Y$ and $\psi: Y\to Z$ of $\V$-categories $X,Y$ and $Z$, utilizing the fact that $\V$ with its internal hom $[\text{-,-}]$ is a $\V$-category, we obtain
\begin{align*}
|\varphi|\otimes|\psi| &= (\bigwedge_{x,x'\in X} [X(x,x'),Y(\varphi x,\varphi x')])\otimes(\bigwedge_{y,y'\in Y}[Y(y,y'),Z(\psi y,\psi y')])\\
&\leq\bigwedge_{x,x'\in X}[X(x,x'),Y(\varphi x,\varphi x')]\otimes [Y(\varphi x,\varphi x'),Z(\psi\varphi x,\psi\varphi x')] \\
& \leq\bigwedge_{x,x'\in X}[X(x,x'),Z(\psi\varphi x,\psi\varphi x')]=|\psi\cdot\varphi|\;.
\end{align*}
Since trivially $\kk\leq| \mathrm{id}_X|$, this makes $\V\text{-}\mathsf{Lip}$ $\V$-normed. The other statements are even easier to verify.
\end{proof}

We will apply the Proposition in Section 7 in order to obtain results for categories of generalized metric spaces.

There is another well-known way of weakening the notion of $\V$-functor. Recall that a {\em $\V$-distributor} $\rho:X\lmodto Y$ (also {\em $\V$-(bi)module} or {\em -profunctor}) of $\V$-categories $X$ and $Y$ is given by a $\V$-functor $\rho:X^{\op}\otimes Y\to\V$, {\em i.e.}, by a function $\rho$ satisfying
$$X(x',x)\otimes\rho(x,y)\otimes Y(y,y')\leq \rho(x',y')$$
for all $x,x'\in X$ and $y,y'\in Y$. Its composite with $\sigma:Y\lmodto Z$ is defined by
$$(\sigma\cdot\rho)(x,z)=\bigvee_{y\in Y}\rho(x,y)\otimes\sigma(y,z).$$
With the identity $\V$-distributor $1_X^*$ on $X$ given by the structure of $X$, one obtains the category $\V\text{-}\mathsf{Dist}$, together with the identity-on-objects functors
$$\xymatrix{\Cats{\V}\ar[rr]^{-_*} && \V\text{-}\mathsf{Dist} && (\Cats{\V})^{\op}\ar[ll]_{-^*}\;,\\}$$
defined by $f_*(x,y)=Y(fx,y)$ and $f^*(y,x)=Y(y,fx)$ for every $\V$-functor $f:X\to Y$ and all $x\in X, y\in Y$. With the order of $\V$-distributors induced pointwise by the order of $\V$, we can regard $\V\text{-}\mathsf{Dist}$ as a 2-category, with 2-cells given by order. One then has $f_*\dashv f^*$, {\em i.e.}, $1_X^*\leq f^*\cdot f_*$ and $f_*\cdot f^*\leq 1_Y^*$.

\begin{proposition}\label{Hausdorffnorm}
Setting the {\em Hausdorff norm} of a $\V$-distributor $\rho:X\lmodto Y$ of $\V$-categories as
\begin{equation}\label{VDistnorm}
|\rho|=\bigwedge_{x\in X}\bigvee_{y\in Y}\rho(x,y)
\end{equation}
one makes $\V\text{-}\mathsf{Dist}$ a $\V$-normed category such that every $\V$-functor $f$, represented as a $\V$-distributor $f_*$ or $f^*$, becomes a $\kk$-morphism. The function $|\text{-}|$ is monotone, thus making the functor
$|\text{-}|:\V\text{-}\mathsf{Dist}\longrightarrow\V$
of Facts {\em \ref{secondremarks}(1)} a lax 2-functor.
\end{proposition}
\begin{proof}
Given $\rho$ and $\sigma:Y\lmodto Z$ one has
\begin{align*}
|\rho|\otimes|\sigma| &=(\bigwedge_{x\in X}\bigvee_{y\in Y} \rho(x,y))\otimes (\bigwedge_{y'\in Y}\bigvee_{z\in Z}\sigma(y',z))     \\
  &\leq \bigwedge_x\bigvee_y(\rho(x,y)\otimes \bigvee_z\sigma(y,z))\\
  &=\bigwedge_x\bigvee_z\bigvee_y\rho(x,y)\otimes\sigma(y,z)\\
  &=\bigwedge_x\bigvee_z(\sigma\cdot\rho)(x,z)
  =|\sigma\cdot\rho|\;.
\end{align*}
Since one also has $\kk\leq \bigwedge_xX(x,x)=|1_X^*|$, this proves the principal assertion of the Proposition. The additional claim may also be verified easily.
\end{proof}

\begin{remark}\label{LawveresHausdorff}
Alternatively, for a $\V$-distributor $\rho:X\lmodto Y$ one may set (see \cite{Lawvere73} and (\ref{Hausdorffmetric})) $$|\!|\rho|\!|=\bigvee_{\varphi:X\to Y}\bigwedge_{x\in X} \rho(x,\varphi x)$$
to make the category $\V\text{-}\mathsf{Dist}$  $\V$-normed; here the join runs over {\em all} mappings $\varphi:X\to Y$. The choice-free proof of this claim proceeds similarly to the proof for $|\rho|$. But with the Axiom of Choice, if the complete lattice $\V$ is completely distributive, one has in fact $|\!|\rho|\!|=|\rho|$ for all $\rho$. 
\end{remark}

For every $\V$-distributor $\rho:X\lmodto Y$ and all subsets $A\subseteq X,\;B\subseteq Y$, denoting their inclusion maps to $X$ and $Y$ by $\mathrm i_A$ and $\mathrm i_B$, respectively, we define
\begin{equation}\label{Hausdorfffunctor}
(\mathcal H\rho)(A,B):=|\mathrm i_B^*\cdot \rho\cdot(\mathrm i_A)_*|=\bigwedge_{x\in A}\bigvee_{y\in B}\rho(x,y)
\end{equation}
and use the abbreviation $\mathcal HX=\mathcal H1_X^*$. Applying the norm rules of Proposition \ref{Hausdorffnorm} we now show how one easily concludes  (some essential parts of)  \cite[Theorem 3]{ACT10} on the Hausdorff monad on $\V\text{-}\mathsf{Cat}$ (identified in \cite{Stubbe10} as describing its Eilenberg-Moore algebras as the conically cocomplete $\V$-categories), and on the lax extension of that monad to $\V\text{-}\mathsf{Dist}$:

\begin{corollary}\label{Hausdorffnormcor}
The function $\mathcal HX$ makes the powerset of every $\V$-category $X$ a $\V$-category, denoted again by $\mathcal HX$, such that $\mathcal H\rho:\mathcal HX\lmodto\mathcal HY$ becomes a $\V$-distributor for every $\V$-distributor $\rho:X\lmodto Y$. This defines a $\V$-normed lax 2-functor $\mathcal H$, so that $|\rho|\leq|\mathcal H\rho|$,  and it restricts to a (strict) endofunctor of $\Cats{\V}$ that lifts the powerset functor of $\SET$, so that one has the commutative diagram
\begin{equation}\label{VCatvsVDist}
\xymatrix{\Cats{\V}\ar[d]_{\mathcal H}\ar[rr]^{-_*} && \V\text{-}\mathsf{Dist}\ar[d]^{\mathcal H} && (\Cats{\V})^{\op}\ar[d]^{{\mathcal H}^{\mathrm{op}}}\ar[ll]_{-^*}\\
\Cats{\V}\ar[rr]^{-_*} && \V\text{-}\mathsf{Dist} && (\Cats{\V})^{\op}\ar[ll]_{-^*}\\
}\end{equation}
\end{corollary}

\begin{proof}
For $\V$-distributors $\rho:X\lmodto Y,\;\sigma:Y\lmodto Z$ and all subsets $A\subseteq Y$ and $C\subseteq Z$, we have
\begin{align*}
(\mathcal H\sigma\cdot\mathcal H\rho)(A,C)&=\bigvee_{B\subseteq Y}\mathcal H\rho(A,B)\otimes\mathcal H\sigma(B,C)\\
&=\bigvee_{B\subseteq Y}|\mathrm i_C^*\cdot\sigma\cdot(\mathrm i_B)_*|\otimes|\mathrm i_B^*\cdot\rho\cdot (\mathrm i_A)_*|  \\
&\leq \bigvee_{B\subseteq Y}|\mathrm i_C^*\cdot\sigma\cdot(\mathrm i_B)_*\otimes\mathrm i_B^*\cdot\rho\cdot (\mathrm i_A)_*|  \\
&\leq |\mathrm i_C^*\cdot \sigma\cdot1_Y^*\cdot\rho\cdot(\mathrm i_A)_*|\\&=\mathcal H(\sigma\cdot\rho)(A,C)
\end{align*}
and $\kk\leq |1_A^*|\leq|\mathrm i_A^*\cdot(\mathrm i_A)_*|=\mathcal HX(A,A).$ For $\rho=\sigma=1_X^*$, this shows that $\mathcal HX$ is a $\V$-category. Choosing alternately only one of the $\V$-distributors to be identical will show that $\mathcal H\rho$ is a $\V$-distributor, while the general case confirms that $\mathcal H$ is a lax 2-functor of $\V\text{-}\mathsf{Dist}$. It is $\V$-normed since
$$\mathcal H\rho(A,B)=|\mathrm i_B^*\cdot\rho\cdot(\mathrm i_A)_*|\geq|\mathrm i_B^*|\otimes|\rho|\otimes|(\mathrm i_A)_*|  \geq\kk\otimes|\rho|\otimes\kk=|\rho|\;. $$
For a $\V$-functor $f:X\to Y$ also  
$\mathcal Hf:\mathcal HX\to \mathcal H Y$ with $(\mathcal Hf)(A)=f(A)$ is a $\V$-functor since
$$\mathcal H X(A,A')=|(\mathrm i_{A'})^*\cdot(\mathrm i_A)_*|\leq|(\mathrm i_{A'})^*\cdot f^*\cdot f_*\cdot (\mathrm i_A)_*|=|(\mathrm i_{f(A')})^*\cdot(\mathrm i_{f(A)})_*|=\mathcal HY(f(A'),f(A))\;.$$
Finally, the left part of (\ref{VCatvsVDist}) commutes since
$$( \mathcal Hf_*)(A,B)=|\mathrm i_B^*\cdot f_*\cdot (\mathrm i_A)_*|=|\mathrm i_B^*\cdot(\mathrm i_{f(A)})_*|=\mathcal HY(f(A),B)=(\mathcal Hf)_*(A,B)\;, $$
and the commutativity of the right part follows by duality.
\end{proof}

\begin{remarks}
(1) For $\V=\mathcal R_+$, with Choice and $\mathbb HX$ as in Examples \ref{exsofnormedcats}(3), one has $\mathcal HX=\mathrm s(\mathbb HX)$.

(2) For further investigations on the functor $\mathcal H$ and various restrictions thereof we refer the reader to \cite{HN20}. The question to which extent completeness properties of the object $X$ get transferred to $\mathcal HX$, without the symmetrization of the structure and/or some restriction on the subsets of $X$ to be considered, such as the (in some sense) compact subsets, remains open.
\end{remarks}

\section{Normed convergence and symmetry}

In order to introduce the concept of normed convergence in a $\V$-normed category, we find it useful to remind ourselves how sequential colimits are formed in $\Setss{\V}$. The following easily checked statement is an immediate consequence of $\Setss{\V}$ being topological over $\SET$;  see (\ref{finalSetVstructure}).
\begin{proposition}\label{finalinSetV}
  The colimit of a sequence \(A_{0}\to A_{1}\to A_{2} \dots\)
  in \(\Setss{\V}\) is formed by providing the colimit $A$ of the underlying sequence in $\SET$ with the least norm that makes the colimit cocone $\xymatrix{(A_N\ar[r]^{\kappa_N\;\;} & A)_{N\in\mathbb N}}$ live in $\SET/\!/\mathcal V$; that is, for all $c\in A$ one has
  $$|c|=\bigvee\{|a|\mid a\in \bigcup_{N\in\mathbb N}\kappa_N^{-1}c\}\,.$$ 
\end{proposition}

We note that, clearly, for every $N_0\in\mathbb N$ one has $|c| = \bigvee_{n\geq N_0}\;\bigvee_{a\in\kappa_n^{-1}c}|a|$, so that we can also write $|c|=\bigwedge_{N\in\mathbb N}\;\bigvee_{n\geq N}\;\bigvee_{a\in\kappa_n^{-1}c}|a|$. In this form the formula will reappear in Section 9 where we investigate normed colimits of sequences in $\SET||\V$, not necessarily in  $\SET/\!/\mathcal V$.
\begin{definition}\label{defnormedcolimit}
Let $s: \mathbb N\to \mathbb X$ be a sequence\footnote{Here the ordered set $\mathbb N$ is treated as a category, discretely $\V$-normed with constant value $\bot$ for all non-identical arrows, so that the sequence $s$ becomes a $\V$-normed functor $\mathbb N\to\mathbb X$.} in a $\mathcal V$-normed category $\mathbb X$, written as
$$ s=\xymatrix{(x_m\ar[r]^-{s_{m,n}}& x_n)_{m\leq n\in\mathbb N}}\;.$$

An object $x$ together with a cocone $\gamma=\xymatrix{(x_n\ar[r]^-{\gamma_n}&x)_{n\in \mathbb N}}$ 
is a {\em normed colimit} of $s$ in $\mathbb X$ if
\begin{itemize}
\item[(C1)]  $\gamma:s\to \Delta x$ is a colimit cocone in the ordinary category $\mathbb X$, and
\item[(C2)] for all objects $y$ in $\mathbb X$, the canonical $\SET$-bijections\footnote{Note that a colimit $x$ of $s$ in the ordinary category $\mathbb X$ is also a colimit of every restricted sequence $s_{|N}$.}
$$\xymatrix{\mathrm{Nat}(s_{|N},\Delta y)\ar[rr]^-{\kappa_N} && \mathbb X(x,y), & (f\cdot\gamma_n)_{n\geq N} && f\ar@{|->}[ll]_-{\kappa_N^{-1}} & ({N\in\mathbb N})\\}$$
form a colimit cocone in $\SET/\!/\mathcal V$, where $s_{|N}$ is the restriction of $s$ to $\uparrow\! N=\{N, N+1, \dots\}$ and $\mathrm{Nat}(s_{|N},\Delta y)=[\uparrow\! N,\mathbb X](s_{|N},\Delta y)\;(N\in\mathbb N)$ is considered as a sequence in $\Setss{\V}$, with all connecting maps given by restriction.
 \end{itemize}
\end{definition}

It is important to note that the cocone $\gamma$ may generally not be replaced by an isomorphic replica in the ordinary category $\mathbb X$. In fact, it may happen that all cocones over a given sequence $s$ satisfy (C1), but only one of them satisfies also (C2); see Example \ref{positiverationals}.

Keeping the notation of this definition, let us first analyze the meaning of (C2):
\begin{proposition}\label{C2explained}
Condition {\em (C2)} says equivalently that, for all morphisms $f$ in $\mathbb X$ with domain $x$, one must have
\begin{equation}\label{ConditionC2}
  |f|=\bigvee_{N \in\mathbb N}\bigwedge_{n\geq N}|f\cdot\gamma_n|\,.
\end{equation}
The $\leq$-part of this equality is satisfied if, and only if,
$
  \mathrm k\leq\bigvee_{N \in\mathbb N}\bigwedge_{n\geq N}|\gamma_n|\,.
$
\end{proposition}

\begin{proof}
Trivially,  with the given cocone $\gamma$ satisfying (C1), the natural $\SET$-bijections $\kappa_N\; (N\in \mathbb N)$ form a colimit cocone in $\SET$. For them to form a colimit cocone in $\Setss{\V}$ equivalently means by Proposition \ref{finalinSetV} that
$$|f|=\bigvee\{|\beta|\mid \beta=\Delta f\cdot \gamma_{|N}:s_{|N}\to \Delta y,\;N\in\mathbb N\}$$
holds for all $f:x\to y$ in $\mathbb X$. This, by the norm formula (\ref{nattransnorm}) for natural transformations (Proposition \ref{CatVproperties}), amounts to the claimed formula for $|f|$.
The second statement of Proposition \ref{C2explained} follows from the following more general lemma.
\end{proof}

\begin{lemma}\label{kcocone}
For any cocone $\alpha:s\to \Delta x$ over a sequence $s$ in a $\V$-normed category $\mathbb X$, the following are equivalent:
\begin{itemize}
\item[{\em (i)}] $\kk\leq \bigvee_{N \in\mathbb N}\bigwedge_{n\geq N}|\alpha_n|\,;$
\item[{\em (ii)}] $|1_x|\leq \bigvee_{N \in\mathbb N}\bigwedge_{n\geq N}|\alpha_n|\,;$
\item[{\em (iii)}] $|f|\leq \bigvee_{N \in\mathbb N}\bigwedge_{n\geq N}|f\cdot\alpha_n|$, for every morphism $f:x\to y$ in $\mathbb X$\,.
\end{itemize}
\end{lemma}
\begin{proof}
Trivially, one has (iii)$\Longrightarrow$(ii)$\Longrightarrow$(i). For (i)$\Longrightarrow$(iii), we note
$$|f|=|f|\otimes\kk\leq|f|\otimes\bigvee_{N \in\mathbb N}\bigwedge_{n\geq N}|\alpha_n|\leq\bigvee_{N \in\mathbb N}\bigwedge_{n\geq N}|f|\otimes|\alpha_n|\leq \bigvee_{N \in\mathbb N}\bigwedge_{n\geq N}|f\cdot\alpha_n|.$$
\end{proof}
Next we confirm that only any infinite tail of a sequence matters for normed convergence:
\begin{corollary}\label{tailssuffice}
Let $\gamma:s\to\Delta x$ be a cocone over the sequence $s$ in the $\V$-normed category $\ncatX$, and let $K\in\mathbb N$. Then $\gamma$ makes $x$ a normed colimit of $s$ if, and only if, the truncated cocone $\gamma_{|K}=(\gamma_n)_{n\geq K}$ makes $x$ a normed colimit of the truncated sequence $s_{|K}$.
\end{corollary}

\begin{proof}
The truncation at $K$ gives, for every object $y$,  a bijection $\mathrm{Nat}(s,\Delta y)\longrightarrow\mathrm{Nat}(s_{|K},\Delta y)$, whence $\gamma$ is an ordinary colimit cocone precisely when $\gamma_{|K}$ is one. Furthermore, given $f:x\to y$ in $\ncatX$, with $\Phi(N):=\bigwedge_{n\geq N}|f\cdot \gamma_n|$ one has $\bigvee_{N\in\mathbb N}\Phi(N)=\bigvee_{N\geq K}\Phi(N)$, where $\leq$ holds since $\Phi$ is monotonely increasing, and where $\geq$ holds trivially. Consequently, (C1) and (C2) hold for $\gamma$ if, and only if, they hold for $\gamma_{|K}$.
\end{proof}

Extending the terminology used for morphisms in Facts \ref{secondremarks}(3), we call a cocone  $\alpha: s\to\Delta x$ over a sequence $s=(x_n)_{n\in\mathbb N}$ in a normed category $\mathbb X$  a $\kk$-{\em cocone} if it satisfies condition (i) of Lemma \ref{kcocone}. We conclude from Proposition \ref{C2explained}:

\begin{corollary}\label{Condition2ab}
An object $x$ with a cocone $\gamma:s\to \Delta x$ is a normed colimit of a sequence $s$ in a $\V$-normed category $\mathbb X$ if, and only if, $\gamma$ is a colimit cocone in the ordinary category $\mathbb X$ such that
\begin{itemize}
\item[{\em(C2a)}]  $\kk\leq\bigvee_{N \in\mathbb N}\bigwedge_{n\geq N}|\gamma_n|$,  i.e., $\gamma$ is a $\kk$-cocone;
\item[{\em (C2b)}]
$|f|\geq \bigvee_{N\in\mathbb N} \bigwedge_{n\geq N}|f\cdot\gamma_n|$, for every morphism $f:x\to y$ in $\mathbb X$.
\end{itemize}
\end{corollary}

\begin{corollary}\label{ncolimunique}
A normed colimit of a sequence in a $\V$-normed category $\mathbb X$ is uniquely determined up to a $\kk$-isomorphism,  i.e., up to an isomorphism in the category $\mathbb X_{\circ}$
  of $\kk$-morphisms of $\ncatX$.
\end{corollary}
\begin{proof}
If $\gamma:s\to \Delta x$ and $\delta:s\to\Delta y$ are both colimit cocones representing $x$ and $y$ as normed colimits of $s$, respectively, then the canonical morphism $f:x\to y$ is not only an isomorphism in $\mathbb X$, but also satisfies
$$|f|=\bigvee_{N \in\mathbb N}\bigwedge_{n\geq N}|f\cdot\gamma_n|=\bigvee_{N \in\mathbb N}\bigwedge_{n\geq N}|\delta_n|\geq\kk\;,$$
and likewise $|f^{-1}|\geq\kk$. Hence, $f$ is an isomorphism in  $\mathbb X_{\circ}$.
\end{proof}

Here is a sufficient, but not necessary, condition on the $\V$-normed category $\mathbb X$ (which will be discussed further in Facts \ref{thirdremarks}) to make condition (C2b) of Corollary \ref{Condition2ab} redundant, as follows.

\begin{proposition}\label{ncolimunderconditionS}
Let $\mathbb X$ be a $\V$-normed category satisfying the condition
\begin{itemize}
\item[{\em(S)}] \qquad$|f\cdot h|\otimes |h|\leq|f|$\quad for all composable morphisms $h$ and $f$.
\end{itemize}
Then an object $x$  with a cocone $\gamma:s\to\Delta x$ is a normed colimit of a sequence $s$ in $\mathbb X$ if, and only if, {\em (C1)} $\gamma$ is a colimit cocone in the ordinary category $\mathbb X$, and {\em (C2a)} $\gamma$  is a $\kk$-cocone.
\end{proposition}

\begin{proof}
First we note that, in (S), the morphism $h$ may be replaced equivalently by any cocone $\alpha: D\to \Delta x$, for some diagram $D:\mathbb I\to\mathbb X$ with $\mathbb I\neq\emptyset$, so that (S) then reads as
$|\Delta f\cdot \alpha|\otimes |\alpha|\leq |f|.$
Indeed, for all $i\in\mathbb I$, using (S) and $\mathbb I\neq\emptyset$, one has
$$|\Delta f\cdot \alpha|\otimes |\alpha|=\bigwedge_{i\in\mathbb I}|f\cdot\alpha_i|\otimes\bigwedge_{i\in\mathbb I}|\alpha_i|\leq\bigwedge_{i\in\mathbb I}|f\cdot\alpha_i|\otimes|\alpha_i|\leq\bigwedge_{i\in I}|f|=|f|\;.$$
Having (C2a) and this extended version of (S), we can now show (C2b) of Corollary \ref{Condition2ab}, as follows, utilizing also the fact that the occurring joins are directed:
\begin{align*}
\bigvee_{N\in\mathbb N}\bigwedge_{n\geq N}|f\cdot\gamma_n| &\leq(\bigvee_{N\in\mathbb N}\bigwedge_{n\geq N}|f\cdot \gamma_n|)\otimes(\bigvee_{N\in\mathbb N}\bigwedge_{n\geq N}|\gamma_n|) \\
&\leq \bigvee_{N\in\mathbb N}(\bigwedge_{n\geq N}|f\cdot \gamma_n|\otimes\bigwedge_{n\geq N}|\gamma_n|)\\
&\leq \bigvee_{N\in\mathbb N}\bigwedge_{n\geq N}|f\cdot \gamma_n|\otimes|\gamma_n|\;\leq\;|f|\;. 
\end{align*}
\end{proof}

\begin{facts}\label{thirdremarks}
(1) Condition (S) is a (strong) {\em symmetry} condition on the normed category $\mathbb X$. Indeed, if $\mathbb X=\mathrm iX$ is given by a $\V$-category $X$ as in Facts \ref{secondremarks}(2), then (S) means equivalently that $X$ is {\em symmetric}, {\em i.e.}, that
$X(x,y)=X(y,x)$
holds for all $x,y\in X$. We call an arbitrary $\V$-normed category $\mathbb X$ satisfying (S) {\em forward symmetric}. The condition generally fails in $\nSetss{\V}$, even for $\V=\mathcal R_+=([0,\infty],\geq,+,0)$. Indeed, considering $\mathbb N$ as an identically $\mathcal R_+$-normed set, then for the endomaps $f$ and $h$ which keep $0$ fixed while $hn=n-1$ and $fn=n+1$ for all $n>0$, one has $|f|=1$ but $|f\cdot h|+|h|=0$.

(2) The dualization of (S) reads as
\begin{itemize}
\item[{(S$^{\mathrm{op}}$)}] \qquad$| g\cdot f|\otimes |g|\leq|f|$\quad for all composable morphisms $f$ and $g$;
\end{itemize}
we call $\mathbb X$ {\em backwards symmetric} in this case. Indeed
for $\mathbb X=\mathrm iX$ as in (1), condition (S$^{\mathrm{op}}$) again amounts to the $\V$-category $X$ being symmetric, and again, it generally fails in $\nSetss{\V}$. However, for arbitrary $\mathbb X$, conditions (S) and (S$^{\mathrm{op}}$) are far from being equivalent (as already the example in (3) shows). But, as noted for $\V=\mathcal R_+$  in \cite[Lemma 2.2] {Kubis17}, each of the two conditions implies the inverse of an isomorphism $f$ in the ordinary category $\mathbb X$ to have the same norm as $f$; for example, with (S$^{\mathrm{op}}$) one has
$$|f|\geq|f^{-1}\cdot f|\otimes|f^{-1}|\geq \kk\otimes |f^{-1}|=|f^{-1}|,$$
and likewise for $|f^{-1}|\geq|f|$. In particular, {\em if $\mathrm{(S)}$ or $\mathrm{(S^{\mathrm{op}})}$ holds, a morphism in $\mathbb X_{\circ}$ that is an isomorphism in the ordinary category $\mathbb X$ is also an isomorphism in $\mathbb X_{\circ}$.}

(3) For $\V=\mathcal R_+$, in addition to our conditions on a normed category, Kubi\'{s} \cite{Kubis17} includes condition (S$^{\mathrm{op}}$) as part of his definition of normed category, and then defines the normed convergence of $s$ to $x$ by requiring only conditions (C1) and (C2a), in their $\mathcal R_+$-versions.  This, however, does not make the colimit unique up to a $0$-isomorphism (here $0=\kk$).

Indeed, the following simple witness appears already in \cite{Kubis17}. Consider the category given by the preordered set $\mathbb N\cup\{a,b\}$ obtained from the natural numbers by adding new distinct elements $a,b$, and extend the natural order by $n\leq a\leq b$ and $n\leq b\leq a$ for all $n\in \mathbb N$; it gets (Kubi\'{s}-)normed when we put  $|x\to y|=0$ whenever $x\notin \{a,b\}$, and $|a\to b|=|b\to a|=\infty$. Hence, $a$ and $b$ are ordinary colimits of the sequence $(n)$, both satisfying (C1) and (C2a), and the ambient category satisfies (S$^{\mathrm{op}}$). However, (S) is violated -- not even (C2b) holds, which is why $a$ and $b$ fail to be $0$-isomorphic.
(If one modifies this example by declaring the norms of morphisms $n\to b$ to be $1$, rather than $0$, one still has a normed category satisfying (S$^{\mathrm{op}}$), but now the ordinary colimit $b$ no longer satisfies (C2a) whilst $a$ still does.)

(4) A {\em $\V$-normed monoid} is simply a monoid $(A,\cdot, 1)$ which, considered as a one-object category, is $\V$-normed; that is: $A$ comes with a function $|\text{-}|:A\to\V$ satisfying $\kk\leq|1|$ and $|a|\otimes|b|\leq|ab|$ for all $a,b\in A$. In case $\V=\mathcal R_+$, such normed monoids are often called semi-normed \cite{BO10}, but here we will omit the prefix. Every {\em left-invariant} Lawvere metric on a monoid $A$ makes $A$ a normed monoid \cite{Kubis17}. In fact, even for general $\V$, if a monoid $A$ carries a $\V$-category structure such that, for all $a,b, c\in a$, one has $A(ca,cb)=A(a,b)$, then
$|a|:=A(1,a)$
makes $A$ a $\V$-normed monoid. Indeed, trivially one has $\kk\leq|1|$ and
$$|a|\otimes |b|=X(1,a)\otimes X(1,b)=X(1,a)\otimes X(a,ab)\leq X(1,ab)=|ab|\;.$$
Conversely, if the $\V$-normed monoid $A$ is, algebraically, a group, then the norm makes $A$ a left-invariant $\V$-category, via
$$A(a,b):=|a^{-1}b|\;,$$
and this actually results into a one-one correspondence between $\V$-norms and left invariant $\V$-category structures on the given group $A$.

(5) Further to the case that the $\V$-normed monoid $(A,\cdot,1)$ considered in (4) is actually a group, let us call $A$ a {\em $\V$-normed group} if the additional condition
$|a^{-1}| = |a|$ holds for all $a\in A$. (For $\V=\mathcal R_+$, this gives the standard notion of {\em normed group} as used in \cite{BO10}.) The induced $\V$-category structure of a $\V$-normed group $A$ is symmetric, so that $A(a,b)=A(b,a)$ holds for all $a,b\in A$. Conversely, if the induced $\V$-category structure of $A$ is symmetric, then, as a one-object $\V$-normed category, $A$ is forward symmetric, {\em i.e.},
\begin{itemize}
\item[(S)]\qquad \qquad $|ab|\otimes |b|=A(1,ab)\otimes A(1,b)=A(1,ab)\otimes A(ab, a)\leq A(1,a)=|a|$
\end{itemize}
holds. Moreover, as follows already from (2), condition (S) implies $|a^{-1}|=|a|$ for all $a\in A$ and thus makes $A$ a $\V$-normed group.  Consequently, for a $\V$-normed monoid $A$ that, algebraically, is a group, the following conditions are equivalent:
\begin{itemize}
\item[(i)] $A$ is a $\V$-normed group;
\item[(ii)] the induced $\V$-category structure of $A$ is symmetric;
\item[(iii)] the one-object $\V$-normed category $A$ is forward symmetric.
\end{itemize}

{\em Caution:} A group $A$ that is a one-object $\V$-normed category, must not be confused with the (generally) multi-object $\V$-normed category $\mathrm iA$ as considered in (1). In the latter category, conditions (S) and (S$^{\mathrm{op}}$) are equivalent, unlike in the former category, unless $A$ is Abelian.

(6) For $\V=\mathsf 2$ where, as in Examples \ref{exsofnormedcats}(1), the norm of a $\V$-normed category $\ncatX$ is given by a morphism class $\mathcal S$ satisfying (\ref{2normasS}),  conditions (S) and (S$^{\mathrm{op}}$) amount to the cancellation conditions
$$
(\mbox S)\quad\, f\cdot h\in\mathcal S\;\&\; h\in\mathcal S\Longrightarrow f\in\mathcal S\qquad\qquad
(\mbox S^{\mathrm{op}})\quad g\cdot f\in\mathcal S\;\&\; g\in\mathcal S\Longrightarrow f\in\mathcal S\;.$$
They typically hold for classes of epimorphisms and classes of monomorphisms in $\ncatX$, respectively.

\end{facts}

\section{Cauchy cocompleteness}

We now extend (the key) Definition \ref{defnormedcolimit} in the expected way:

\begin{definition}\label{defCauchycoc}
For a $\mathcal V$-normed category $\mathbb X$, we say that
\begin{itemize}
\item a sequence $ s=\xymatrix{(x_m\ar[r]^{s_{m,n}\quad}& x_n)_{m\leq n\in\mathbb N}}$ in $\mathbb X$ is {\em Cauchy} if
$\mathrm k\leq \bigvee_{N\in\mathbb N}\;\bigwedge_{n\geq m\geq N}|s_{m,n}|,$ and
\item $\mathbb X$ is {\em Cauchy (norm-)cocomplete} if every Cauchy sequence in $\mathbb X$ has a normed colimit in $\mathbb X$.
\end{itemize}
\end{definition}

\begin{remark}\label{dualization}
The notions of Definitions \ref{defnormedcolimit} and \ref{defCauchycoc} dualize in an obvious way. Indeed, for a $\V$-normed category $\mathbb X$, the dual category $\mathbb X^{\mathrm{op}}$ of the ordinary category $\mathbb X$ becomes $\V$-normed when giving every morphism the same norm as in $\mathbb X$. Now, having an {\em inverse} sequence $s:\mathbb N^{\mathrm{op}}\to \mathbb X$, given by morphisms $s_{m,n}:x_n\to x_m$ in $\mathbb X$ for all $m\leq n\in\mathbb N$, the inverse sequence is said to be {\em Cauchy} in $\mathbb X$ if the sequence $s^{\mathrm{op}}:\mathbb N\to\mathbb X^{\mathrm{op}}$ is Cauchy in $\mathbb X^{\mathrm{op}}$. Furthermore, an object $x$ with a cone $\lambda:\Delta x\to s$ is a {\em $\V$-normed limit} of $s$ in $\mathbb X$ if $x$ with $\lambda^{\op}:s^{\op}\to \Delta x$ is a normed colimit of $s^{\mathrm{op}}$. This means that $\lambda$ is a limit cone in the ordinary category $\mathbb X$ such that
$|f|=\bigvee_{N\in\mathbb N}\bigwedge_{n\geq N}|\lambda_n\cdot f|,$
for all morphisms $f:y\to x$ in $\mathbb X$.
\end{remark}

In what follows we collect some initial facts and illustrate the meaning of Cauchy cocompleteness for the key quantales. For a brief comparison of this notion with the common notion of {\em idempotent completeness} for {\em ordinary} categories (and $\V$-enriched categories \cite{BorceuxDejean86}), see  Section 14.

\begin{facts}\label{firstCauchyexamples}

(1) For $\V=\mathcal R_+$, the condition for a sequence $s$ to be Cauchy in the $\mathcal R_+$-normed category $\mathbb X$ reads as
$\inf_{N\in\mathbb N}\sup_{n\geq m\geq N}|s_{m,n}|=0,$
and for the ordinary colimit $x$ with colimit cocone $\gamma$ in $\mathbb X$ to be a normed colimit means that
\begin{equation}\label{C2aforR}
\inf_{N\in\mathbb N}\sup_{n\geq N}|\gamma_n|=0\qquad\mbox{and}\qquad |f|\leq\inf_{N\in\mathbb N}\sup_{n\geq N}|f\cdot\gamma_n|
\end{equation}
 for every morphism $f:x\to y$ in $\mathbb X$  (see Corollary \ref{Condition2ab}).

 In case $\mathbb X=\mathrm iX$ is induced by a (Lawvere) metric space $X$, the sequence $s=(x_n)$ is Cauchy if, and only if, $\inf_{N\in\mathbb N}\sup_{n\geq m\geq N}X(x_m,x_n)=0,$ so that $s$ must be {\em forward Cauchy} in the sense of \cite{BBR98}; furthermore, $x$ is a normed colimit of $s$ if, and only if,
$X(x,y)=\inf_{N\in\mathbb N}\sup_{n\geq N}X(x_n,y)$
for all $y\in X$, which means that $x$ is a {\em forward limit} of $s$ in the language of \cite{BBR98}. (Note that here the ordinary colimit condition for $x$ comes for free since $\mathrm iX$ is a groupoid.) The notions of {\em backward Cauchy} sequence and {\em backward limit} in $X$ come about by dualization according to Remark \ref{dualization}, {\em i.e.}, by interchanging the arguments of $X$(-,-).

Of course, {\em if X is symmetric,} there is no need to distinguish between forward and backward, and {\em one obtains the standard notions of Cauchy sequence and sequential convergence}  in $X$.

(2) It is important to note that {\em the existence of a normed colimit for a sequence $s$ does not necessitate $s$ to be Cauchy}, even when $\V=\mathcal R_+$. For example, considering the ordered set $\mathbb N\cup\{\infty\}$ of natural numbers with an added maximum as a category, normed by $|m\to n|=n-m$ and $|n\to\infty|=0$ for all $m\leq n$ in $\mathbb N$, we obtain a normed category (satisfying (S), but not (S$^{\mathrm{op}}$)) such that $\infty$ is a normed colimit of the sequence $s=(n)_{n\in\mathbb{N}}$, although $s$ badly fails to be Cauchy; in fact, here $\inf_{N\in\mathbb N}\sup_{n\geq m\geq N}|s_{m,n}|=\infty$.  However, in the presence of (S$^{\op}$), see Remarks \ref{convergentisCauchy}(1).

(3) For $\V=\mathsf 1$ we have $ \mathsf{Cat}/\!/\mathcal V\cong\mathsf{Cat}$, and every sequence in a category $\mathbb X$ is Cauchy, and  $\mathbb X$ is Cauchy cocomplete if, and only if, $\mathbb X$ has colimits of sequences.

For $\mathcal V=\mathsf 2$, a
$\V$-normed category $\ncatX$ is an ordinary category equipped with a wide subcategory $\mathcal S$ (Examples \ref{exsofnormedcats}(1)). A sequence $s$ in $\mathbb X$ is Cauchy if, and only if, eventually all of its connecting maps lie in $\mathcal S$; and $\mathbb X$ is Cauchy cocomplete if, and only if, every Cauchy sequence $s$ has a colimit $x$ with a colimit cocone $(\gamma_n)_n$ lying eventually in $\mathcal S$, such that any morphism $f:x\to y$ belongs to $\mathcal S$ as soon as eventually all morphisms $f\cdot \gamma_n$ do so.
In simplified form, this equivalently means that the colimit of any sequence $s$ with all connecting morphisms in $\mathcal S$ exists in $\mathbb X$ and has a tail $s_{|N}=(s_{m,n})_{n\geq m\geq N}$ (with some $N\in \mathbb N$) which is actually a colimit in the subcategory $\ncatX_{\circ}=\mathcal S$.  For a wide array of such situations, especially in the dual situation, we refer to the literature, such as \cite{AnhWiegandt97}.
\end{facts}

With Corollary \ref{tailssuffice} and Proposition \ref{ncolimunderconditionS} one easily concludes from Facts \ref{thirdremarks}(1),(6) and \ref{firstCauchyexamples}(1),(3):

\begin{corollary}\label{CauchycocoforiX}
{\em (1)} For a classical metric space $X$, the normed category $\mathrm iX$ is Cauchy cocomplete if, and only if, $X$ is complete in the ordinary sense.

{\em (2)} If the ordinary category $\ncatX$ has colimits of sequences lying in a wide subcategory $\mathcal S$ that satisfies Condition {\em (S)}, then the $\mathsf 2$-normed category $(\ncatX,\mathcal S)$ is Cauchy cocomplete if, and only if,  the category $\mathcal S$ (as a category in its own right) has colimits of sequences.
\end{corollary}

\begin{remarks}
(1) Condition (S) is an essential hypothesis in (2) of the Corollary and may not be replaced by (S$^{\op}$). Even when a sequence in $\mathcal S$ has a colimit in $\ncatX$ with all colimit injections in $\mathcal S$, it may not be a colimit in $\mathcal S$. In fact, this may happen in unsuspected circumstances, for instance for the class $\mathcal S$ of subspace embeddings in the category of topological spaces that satisfies (S$^{\op}$): see \cite[Example 3.5]{AHRT23}.

(2) The statements of Facts \ref{firstCauchyexamples}(3) and the Corollary may be easily generalized to the case $\mathcal V=\mathcal PM$ with a commutative monoid $(M,+,0)$ as in Examples \ref{otherexamples}(1).
\end{remarks}

For further illustration of Cauchy cocompleteness, let us particularly look at (small) one-object $\V$-normed categories, {\em i.e.}, at a $\V$-normed monoid $(A,\cdot,1)$ as in Facts \ref{thirdremarks}(4),(5), for a general quantale $\V$.  A sequence $s$ in such a category is simply a sequence $(a_n)_n$ of elements in $A$, and a cocone $\alpha$ over $s$ is given by elements $\alpha_n\in A$ satisfying $\alpha_{n+1}a_n=\alpha_n$ for all $n\in\mathbb N$.
If $A$ is a group, the cocone is already determined by  $\alpha_0$,
since necessarily  $\alpha_n=\alpha_0(a_{n-1} ... a_1a_0)^{-1}$ (where the product of an empty string of elements is $1$). Since every morphism in the category $A$ is an isomorphism, any cocone $\gamma$ over $s$ presents the only object $\ast$  in  $A$ trivially as an ordinary colimit of $s$: the only factorizing morphism induced by any other cocone $\alpha$ is simply $\alpha_0\cdot\gamma_0^{-1}$:
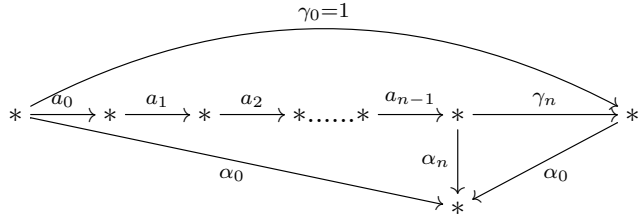
$$\xymatrix{&&&&&\ast\ar[dd]^{\alpha_0\cdot\gamma_0^{-1}} \\
\ast\ar[r]^{a_0}\ar@/_1.7pc/[rrrrrd]^{\alpha_0}\ar@/^1.7pc/[rrrrru]^{\gamma_0} & \ast\ar[r]^-{a_1}  &\ast ......\ast\ar[r]^-{a_{n-1}}&\ast\ar[rru]^{\gamma_n}\ar[drr]_{\alpha_n} && \\
&&&&&\ast\\
}$$
If the group $A$ is actually a $\V$-normed group, so that $A$ enjoys the symmetry condition (S) (as shown in Facts \ref{thirdremarks}(5)), for $\gamma$ to present $\ast$ as a normed colimit, by  Proposition \ref{ncolimunderconditionS} it is necessary and sufficient that $\gamma$ be a $\kk$-cocone, {\em i.e.},
\begin{equation}\label{C2aforgroups}
\kk\leq\bigvee_{N\in\mathbb N}\bigwedge_{n\geq N}|\gamma_0(a_{n-1}....a_0)^{-1}|.
\end{equation} 

\begin{example}\label{positiverationals}
With $\V=\mathcal R_+$ consider the multiplicative group $\mathbb Q_{>0}$ of the positive rationals, which becomes a normed group when one sets
$$ |r|:=\sum_p\max\{n_p,-n_p\} \;\text{whenever}\;r=\prod_p p^{n_p}\; \text{with}\; n_p\in\mathbb Z, $$
where $p$ runs through the set of prime numbers (and the products and sums are only nominally infinite). Note that $|r|=0$ only if $r=1$.
For a Cauchy sequence $s=(a_n)$ in $\mathbb Q_{>0}$ we have $\mathrm{inf}_N\mathrm{sup}_{n\geq m\geq N}|a_{n-1}...a_m|=0$.
Since the norms are always integer valued, this means that, beginning from some $N\in \mathbb N$, the sequence becomes constantly $1$. For the cocone $\gamma$ with $\gamma_0:=a_{N-1}...a_0$ we then have $\gamma_n=1$ for all $n\geq N$, so that $\mathrm{inf}_N\mathrm{sup}_{n\geq N}|\gamma_n|=0$, {\em i.e.}, (\ref{C2aforgroups}) holds and $\gamma$ exhibits $\ast$ as a normed colimit of $s$ in $\mathbb Q_{>0}$. This shows that $\mathbb Q_{>0}$ is Cauchy cocomplete.
\end{example}

The fact that all convergent sequences in $\mathbb Q_{>0} $ become constant (with value 1) is due to the discreteness of the integers in $[0,\infty]$ and, hence, quite atypical. In what follows we consider a (classical) normed vector space $V$ (so that $\nv a\nv<\infty$ for all $a\in V$). Then the vector norm makes its additive group normed, {\em i.e.} a one-object normed category, with vector addition as composition.

\begin{lemma}\label{serieslemma}
Let  $s=(a_n)_{n\in\mathbb N}$ be any sequence in $V$. Then:

{\em (1)} For $s$ to be Cauchy in the normed group $(V,+)$ means equivalently that the series $\sum_na_n$ satisfies the {\em Cauchy criterion}: $\forall\varepsilon>0\;\exists N\;\forall n\geq m\geq N:\nv\sum_{i=m}^{n-1}a_i\nv\leq\varepsilon$.

{\em (2)} For $s$ to be convergent in $(V,+)$ means equivalently that the series $\sum_na_n$ converges in $V$ in the standard sense (i.e., with respect to the metric induced by its norm).
\end{lemma}

\begin{proof}
(1) The Cauchy criterion translates to $\inf_n\sup_{n\geq m\geq N}\nv s_{m,n}\nv=0$, {\em i.e.}, to $s$ being Cauchy.

(2) The convergence of the series means  the existence of a vector $\gamma_0$ satisfying the condition $\inf_N\sup_{n\geq N}\nv \gamma_0-\sum_{i=0}^{n-1}a_n\nv=0$, and since any cocone over $s$ is determined by its first component, this equivalently means that a cocone $\gamma$ satisfying the first equality of (\ref{C2aforR}) exists.
\end{proof}

\begin{theorem}\label{BanachasCauchycoco}
The following are equivalent for a (classical) normed vector space $V$:
\begin{itemize}
\item[{\em (i)}] The additive normed group of $V$ is Cauchy cocomplete (as a one-object normed category).
\item[{\em (ii)}] Every series in $V$ satisfying the Cauchy criterion converges in $V$.
\item[{\em (iii)}] $V$ is a Banach space.
\end{itemize}
\end{theorem}

\begin{proof}
The equivalence (i)$\iff$(ii) follows from Lemma \ref{serieslemma}. For the equivalence (ii)$\iff$(iii) one uses the well-known fact that a normed vector space $V$ is complete if and only if every absolutely convergent series in $V$ converges in $V$ (\cite[Proposition 3.1.2]{Semadeni71}), along with the trivial fact that an absolutely convergent series satisfies the Cauchy criterion.
\end{proof} 

\section{Change of base for normed categories,  metric spaces}
Our next goal is to show that the category $\MET_{\infty}$ of all Lawvere metric spaces, with arbitrary maps between them as morphisms and normed as in  Example \ref{exsofnormedcats}(4), is Cauchy cocomplete. We will do so in three steps, by first giving  conditions on our quantale $(\V,\leq,\otimes,\kk)$ guaranteeing that the normed category $\V\text{-}\mathsf{Lip}$ of all small $\V$-categories with arbitrary maps as defined in Proposition \ref{Lipschitznorm} is Cauchy cocomplete. With the benefit of the methods explored in \cite{Fla92, Flagg97}, the proof extends standard ``epsilon'' arguments of analysis to a fairly general quantalic context. Then we will briefly discuss how Cauchy cocompleteness for $\V$-normed categories fares under changing the ``base''  $\V$, before applying our findings to the adjunction
\begin{equation}\label{elogadjunction}
e\dashv \log^{\circ}:\mathcal R_{\times}=([0,\infty],\geq,\cdot,1)\longrightarrow \mathcal R_+=([0,\infty],\geq,+,0)
\end{equation}
with the multiplicative version $\mathcal R_{\times}$ of the extended real half-line; see Examples \ref{realswithmult}(2). This then gives the Cauchy cocompleteness of $\MET_{\infty}$.

Recall that for $u,v\in\V$ one says that $u$ is {\em totally below} $v$, written as 
$u\lll v$, if $v\leq\bigvee W$ with $W\subseteq \V$ can hold only if $u\leq w$ for some $w\in W$. We say that $v$ is {\em approximated from totally below} if $v=\bigvee\!\triple v$, where $ \triple v=\{u\in\V\mid u\lll v\}$. Recall that the complete lattice $\V$ is {\em constructively completely distributive} \cite{Wood04, HST14} if every element in $\V$ is approximated from totally below. In the presence of Choice this property is equivalent to complete distributivity in the standard sense. 
\begin{theorem}\label{VLiptheorem}
Let the tensor-neutral element $\kk$ be approximated from totally below in $\V$, so that $\kk=\bigvee\!\triple\kk$. Then the $\V$-normed category $\V\text{-}\mathsf{Lip}$ is Cauchy cocomplete.
\end{theorem}

\begin{proof}
For a given Cauchy sequence $s=(\xymatrix{X_m\ar[r]^{s_{m,n}}&X_n\\})_{m\leq n}$ in $\V\text{-}\mathsf{Lip}$, we form its (ordinary) colimit $X$ in the category $\SET$ with colimit cocone  $\gamma=(\xymatrix{X_n\ar[r]^{\gamma_n}&X\\})_{ n\in\mathbb N}$ and now want to define a $\V$-category structure on $X$ by
\begin{equation}\label{defofcolimitmetric}
X(x,y):= \bigwedge_{N\in \mathbb N}\Phi_{x,y}(N),\;\mbox{with}\;\Phi_{x,y}(N):=\bigvee\{X_n(x',y')\mid n\geq N, x'\in\gamma_n^{-1}x, y'\in\gamma_n^{-1}y\}
\end{equation}
for all $x,y\in X$. Since $\Phi_{x,y}$ is monotonely decreasing in $N$, we actually have for every $K\in\mathbb N$
\begin{equation}\label{defofcolimmodified}
X(x,y)=\bigwedge_{N\geq K}\Phi_{x,y}(N).
\end{equation}
Trivially, $\kk\leq X(x,x)$. In order to establish the inequality
$X(x,y)\otimes X(y,z)\leq X(x,z)$ in $X$, we 
consider any $\varepsilon, \eta \lll\kk \in \V$. The Cauchyness of $s$ gives us some $K\in \mathbb N$ with $\varepsilon, \eta \leq |s_{m,n}|$ for all $n\geq m\geq K$. Then
\begin{align*}
\varepsilon\otimes\eta\otimes X(x,y)\otimes X(y,z) &= (\varepsilon\otimes\bigwedge_{N\geq K}\Phi_{x,y}(N))\otimes(\eta\otimes\bigwedge_{M\geq K}\Phi_{y,z}(M))\\
&\leq \bigwedge_{N\geq K}\;\bigvee_{\substack{n\geq N\\m\geq N}}\bigvee_{\substack{x'\in\gamma_n^{-1}x\\ y'\in\gamma_n^{-1}y}}\;\bigvee_{\substack{y''\in\gamma_m^{-1}y\\z'\in\gamma_m^{-1}z}}\varepsilon\otimes X_n(x',y')   \otimes\eta\otimes X_m(y'',z').
\end{align*}
Keeping the data occurring in the last row of joins fixed,  from $\gamma_ny'=\gamma_my''$ and the construction of $X$ as a directed colimit in $\SET$ one finds $\ell\geq n,m$ with $s_{n,\ell}y'=s_{m,\ell}y''$, which gives with (\ref{VLipnorm})
\begin{align*}
\varepsilon\otimes X_n(x',y')   \otimes\eta\otimes X_m(y'',z')&\leq |s_{n,\ell}|\otimes X_n(x',y')\otimes |s_{m,\ell}|\otimes X_m(y'',z') \\
&\leq X_{\ell}(s_{n,\ell}(x'),s_{n,\ell}(y'))\otimes X_{\ell}(s_{m,\ell}(y''),s_{m,\ell}(z'))\\
&\leq X_{\ell}(s_{n,\ell}(x'),s_{m,\ell}(z'))\\
&\leq \Phi_{x,z}(N).
\end{align*}
Combining the two displayed calculations we obtain first
$$\varepsilon\otimes\eta\otimes X(x,y)\otimes X(y,z) \leq \bigwedge_{N\geq K}\Phi_{x,z}(N)=X(x,z)$$
and then with $\kk=\bigvee\triple\kk$ the desired inequality
$$ X(x,y)\otimes X(y,z)=(\bigvee_{\varepsilon\lll\kk}\varepsilon \otimes X(x,y))\otimes(\bigvee_{\eta\lll\kk}\eta\otimes X(y,z))\leq X(x,z).$$ 

Having (C1) we must show that Conditions (C2a) and (C2b) of Corollary \ref{Condition2ab} are satisfied.

(C2a) For every $K\in\mathbb N$ and all $x,x'\in X_K$, since $|s_{K,n}|\leq [X_K(x,x'),X_n(s_{K,n}x,s_{K,n}x')]$ for all $n\geq K$,     one has 
\begin{align*}
X_K(x,x')\otimes\bigwedge_{n\geq K} |s_{K,n}|&\leq X_K(x,x')\otimes\bigwedge_{m\geq K}\bigvee_{n\geq m}|s_{K,n}|\\
&\leq \bigwedge_{m\geq K}\bigvee_{n\geq m}X_K(x,x')\otimes |s_{K,n}|    \\
&\leq\bigwedge_{m\geq K}\bigvee_{n\geq m} X_n(s_{K,n}x,s_{K,n}x') \\
&\leq\bigwedge_{m\geq K}\Phi_{\gamma_Kx,\gamma_Kx'}(m)=X(\gamma_Kx,\gamma_Kx'), 
\end{align*}
where in the last two steps we have used $s_{K,n}x\in\gamma_n^{-1}(\gamma_Kx)$ and $s_{K,n}x'\in\gamma_n^{-1}(\gamma_Kx')$ and 
(\ref{defofcolimmodified}).
We obtain
$\bigwedge_{n\geq K} |s_{K,n}|\leq\bigwedge_{x,x'\in X_K}[X_K(x,x'),X(\gamma_Kx,\gamma_kx')]=|\gamma_K|$
and, since $s$ is Cauchy, conclude
$$\kk\leq \bigvee_N\bigwedge_{K\geq N}\bigwedge_{n\geq K}|s_{K,n}|\leq \bigvee_N\bigwedge_{K\geq N}|\gamma_K|.$$

(C2b) For any mapping $f: X\to Y$ of $\V$-categories, to obtain $\bigvee_K\bigwedge_{n\geq K}|f\cdot\gamma_n|\;\leq\;|f|$ we need to show
$(\bigwedge_{n\geq K}|f\cdot\gamma_n|)\otimes X(x,y)\leq Y(fx,fy)$, for all $K\in\mathbb N$ and $x,y\in X$.  Indeed, we have:
\begin{align*}
(\bigwedge_{n\geq K}|f\cdot\gamma_n|)\otimes(\bigwedge_{N\geq K}\Phi_{x,y}(N))&\leq \bigwedge_{N\geq K} ( (\bigwedge_{n\geq K}|f\cdot\gamma_n|)\otimes (\bigvee_{m\geq N}\bigvee_{\substack{x'\in\gamma_m^{-1}x\\y'\in\gamma_m^{-1}y}}X_m(x',y'))            \\  
 & \leq  \bigwedge_{N\geq K}\bigvee_{m\geq N}\bigvee_{\substack{x'\in\gamma_m^{-1}x\\y'\in\gamma_m^{-1}y}} ((\bigwedge_{n\geq K}|f\cdot\gamma_n|)\otimes X_m(x',y')) \\
 & \leq  \bigwedge_{N\geq K}\bigvee_{m\geq N}\bigvee_{\substack{x'\in\gamma_m^{-1}x\\y'\in\gamma_m^{-1}y}} (\bigwedge_{n\geq K}|f\cdot\gamma_n|\otimes X_m(x',y')) \\ 
 & \leq  \bigwedge_{N\geq K}\bigvee_{m\geq N}\bigvee_{\substack{x'\in\gamma_m^{-1}x\\y'\in\gamma_m^{-1}y}} (|f\cdot\gamma_m|\otimes X_m(x',y')) \\
 & \leq  \bigwedge_{N\geq K}\bigvee_{m\geq N}\bigvee_{\substack{x'\in\gamma_m^{-1}x\\y'\in\gamma_m^{-1}y}}Y(f\gamma_m(x'),f\gamma_m(y'))\;=\;Y(fx,fy)\;. 
\end{align*}
\end{proof}

\begin{remarks}\label{convergentisCauchy}
(1) Similarly as in the above proof one shows that, {\em if the quantale $\V$ satisfies $\kk=\bigvee\triple\kk$, and if} 
(S$^{\op}$) {\em of Facts {\em \ref{thirdremarks}(2)} holds, a sequence $s$ with a normed colimit cocone $\gamma$ must be Cauchy.} Indeed, since (S$^{\op}$) implies $|\gamma_m|\otimes|\gamma_n|\leq|s_{m,n}|$ whenever $m\leq n$, all $\varepsilon, \eta\lll\kk$ give some $N$ with $\varepsilon \otimes\eta\leq |s_{m,n}|$ for all $n\geq m\geq N$, and the Cauchy condition follows.

(2) There is an alternative way of showing conditions (C2a) and (C2b) in the proof of Theorem \ref{VLiptheorem}, by relying on the Cauchy cocompleteness of $\nSetss{\V}$ as shown in Theorem \ref{d:thm:1}. Indeed, using the norm-preserving functor $D: \V\mbox{-}\mathsf{Lip}\longrightarrow \nSetss{\V},\; X\longmapsto X\times X$ of Proposition \ref{Lipschitznorm}, one proceeds as follows: given a Cauchy sequence $s$ in $\V\text{-}\mathsf{Lip}$, one takes the normed colimit of the Cauchy sequence $Ds$ in $\nSetss{\V}$; its normed colimit is of the form $D\gamma$, with the cocone $\gamma$ formed as in the first part of the above proof, and it satisfies (C2a) and (C2b); by the norm-preservation, $\gamma$ itself must satisfy (C2a) and (C2b).
\end{remarks} 

\begin{examples}\label{realswithmult}
(1) The quantales $\mathsf 2$ and $ \mathcal R_+$ satisfy (like any other quantale with a completely distributive lattice) the hypothesis of Theorem \ref{VLiptheorem}. Therefore, the  $\mathsf 2$-normed category $\mathsf 2\text{-}\mathsf{Lip}$ of preordered sets and arbitrary maps is Cauchy-cocomplete, and likewise for $\mathcal R_+\text{-}\mathsf{Lip}$. But note that, unlike in $\MET_\infty$ of Examples \ref{exsofnormedcats}(4), the norm in $\mathcal R_{+}\text{-}\mathsf{Lip}$ of any mapping $\varphi:X\to Y$ of Lawvere metric spaces is given by $|\varphi|=\sup_{x,x'\in X}(Y(\varphi x,\varphi x')-X(x,x'))$.

(2) The quantale $\mathcal R_{\times}=([0,\infty],\geq,\cdot,1)$ is defined so that
the exponential function $e:\mathcal R_+\to\mathcal R_{\times}$, extended by $e^{\infty}=\infty$, becomes a homomorphism of quantales, {\em i.e.}, a homomorphism of monoids which preserves infima (with respect to the natural order $\leq$ of $[0,\infty]$).
The monotonicity of the multiplication on $[0,\infty]$ in each variable necessitates $\alpha\cdot\infty=\infty$ for $\alpha>0$, and the preservation of infima then forces the equality
$ 0\cdot \infty=\infty$ .
Since it extends the usual fractions in case $\alpha, \beta\notin\{0,\infty\}$,
we denote the internal hom $[\beta,\alpha]$ in $\mathcal R_{\times}$ by $\frac{\alpha}{\beta}$ for all $\alpha, \beta\in[0,\infty]$. Hence, its value is given by adjunction, so that
$\frac{\alpha}{\beta}=\inf\{\gamma\in[0,\infty]\mid \alpha\leq\beta\cdot\gamma\}$;
in particular:
$ \frac{0}{0}=0,\quad\frac{\alpha}{0}=\infty\;(\alpha>0),\quad  \frac{\alpha}{\infty}=\frac{\infty}{\infty}=0.$
The quantale $\mathcal R_{\times}$ satisfies the hypothesis of Theorem \ref{VLiptheorem}. Hence, the category $\mathcal R_{\times}\text{-}\mathsf{Lip}$, in which the norm of an arbitrary map $\varphi:X\to Y$ is given by its Lipschitz value $L(\varphi)=\sup_{x,x'\in X}\frac{Y(\varphi x,\varphi x')}{X(x,x')}$ (as in (\ref{Lipschitzvalue}), arising from the general formula (\ref{VLipnorm})), is Cauchy cocomplete.
\end{examples}

Recall that a monotone map $\varphi: \V\to \mathcal W$ to a quantale $(\mathcal W,\leq,\boxtimes,\mathrm n)$ is a {\em lax homomorphism} of quantales if $\mathrm n\leq\varphi \kk$ and $\varphi v\boxtimes\varphi v'\leq\varphi(v\otimes v')$ for all $v,v'\in \V$. Such lax homomorphism induces the change-of-base functor
\begin{equation}\label{changeofbasefctr}
\mathcal B_{\varphi}:\mathsf{CAT}/\!/\V\longrightarrow\mathsf{CAT}/\!/\mathcal W,\quad (\ncatX,|\text{-}|)\longmapsto(\ncatX,|\text{-}|_{\varphi}),
\end{equation}
which regards a $\V$-normed category $\ncatX$ as a $\mathcal W$-normed category via $|f|_{\varphi}:=\varphi(|f|)$ for all morphisms $f$ in $\ncatX$, and which makes $\V$-normed functors become $\mathcal W$-normed.

Furthermore, if we also have a lax homomorphism $\psi:\mathcal W\to \V$ which, as a monotone map, is right adjoint to $\varphi$, then $\varphi$ is actually strict and we have the induced adjunction
$\mathcal B_{\varphi}\dashv\mathcal B_{\psi}$. Indeed, a $\V$-normed functor $F:\ncatX\to\mathcal B_{\psi}\ncatY$ may be considered equivalently as a $\mathcal W$-normed functor $F:\mathcal B_{\varphi}\ncatX\to\ncatY$, since the adjunction $\varphi\dashv\psi$ facilitates, for all morphism $f$ in $\ncatX$, the equivalence
$$|f|\leq|Ff|_{\psi}\iff|f|_{\varphi}\leq|Ff|\;.$$
\begin{proposition}\label{changeofbaseprop}
Let $\varphi\dashv\psi:\mathcal W\to\V$ be adjoint lax homomorphisms of quantales, with $\psi$ preserving joins of monotone sequences in $\mathcal W$. Then, if the $\mathcal W$-normed category $\ncatY$ is Cauchy cocomplete, so is the $\V$-normed category $\mathcal B_{\psi}\ncatY$.
\end{proposition}
\begin{proof}
Let $s=(s_{m,n})_{m\leq n}$ be a Cauchy sequence in $\mathcal B_{\psi}\ncatY.$ So, with $|\text{-}|$ denoting the norm in the given $\mathcal W$-normed category $\ncatY$, we have $\kk\leq\bigvee_N\bigwedge_{n\geq m\geq N}|s_{m,n}|_{\psi}$ and, since the left adjoint $\varphi$ preserves joins, obtain
$$\mathrm n=\varphi\kk=\bigvee_N\varphi(\bigwedge_{n\geq m\geq N}\psi(|s_{m,n}|))\leq \bigvee_N\bigwedge_{n\geq m\geq N}\varphi\psi(|s_{m,n}|)\leq\bigvee_N\bigwedge_{n\geq m\geq N}|s_{m,n}|\;. $$
Hence, $s$ is Cauchy in $\ncatY$ and, hence, has a normed colimit $x$ in $\ncatY$, with colimit cocone $\gamma$. We claim that $x$ is also a normed colimit of $s$ in $\mathcal B_{\psi}\ncatY$. Indeed, this is an immediate consequence of the assumed preservation of joins of monotone sequences in $\mathcal W$ and the preservation of all meets  by the right adjoint  $\psi$ since, for all morphisms $f:x\to y$ in $\ncatY$, from $|f|=\bigvee_N\bigwedge_{n\geq N}|f\cdot\gamma_n|$ one obtains $|f|_{\psi}=\bigvee_N\bigwedge_{n\geq N}|f\cdot\gamma_n|_{\psi}$.
\end{proof}

The inf-preserving map $e:[0,\infty]\to [0,\infty]$ of Examples \ref{realswithmult}(2) has an adjoint, $\log^{\circ}$ as in (\ref{elogadjunction}), whose values are given by the equivalence
($\log^{\circ}\alpha\leq\beta\iff\alpha\leq e^{\beta}$) for all $\alpha,\beta\in[0,\infty]$; in particular,
$\log^{\circ}0=0,\;\;\log^{\circ}\alpha=\max\{0,\log\alpha\}\;(0<\alpha<\infty),\;\;\log^{\circ}\infty=\infty.$
The mapping $\log^{\circ}:\mathcal R_{\times}\to \mathcal R_+$ is only a lax homomorphism of quantales, since the easily established inequality
$\log^{\circ}(\alpha\cdot\beta)\leq \log^{\circ}\alpha+\log^{\circ}\beta$
is generally strict.

Let us now specialize the adjunction $\varphi\dashv\psi$ of Proposition \ref{changeofbaseprop} to
$e\dashv \log^{\circ}:\mathcal R_{\times}\to \mathcal R_+$. Both functions are lax homomorphisms of quantales, with $\log^{\circ}:[0,\infty]\to[0,\infty]$ preserving infima. 
We consider the category $\mathbb Y$ whose objects are sets $X$ equipped with a mere function $X\times X\to[0,\infty]$, and whose morphisms are arbitrary maps $f:X\to Y$, $\mathcal R_{\times}$-normed by their Lipschitz value $L(f)=\sup_{x,x'}\frac{Y(fx,fx')}{X(x,x')}$. The second part of the proof of Theorem \ref{VLiptheorem} does not use the $\V$-category axioms of the objects of the category $\V\text{-}\mathsf{Lip}$ and hence, with $\V=\mathcal{R}_{\times}$, it shows that $\mathbb{Y}$ is Cauchy cocomplete. Therefore the $\mathcal R_+$-normed category $\mathcal B_{\log^{\circ}}(\mathbb Y)$ is also Cauchy cocomplete, and it contains $\mathsf{Met}_{\infty}$ as a full subcategory, with the same norm. 

It now suffices to show that $\mathsf{Met}_{\infty}$ is closed in $\mathcal B_{\log^{\circ}}(\mathbb Y)$ under taking normed colimits of Cauchy sequences, and this follows with an adaptation of the first part of the proof of Theorem \ref{VLiptheorem}. Indeed, for a Cauchy sequence $s=(\xymatrix{X_m\ar[r]^{s_{m,n}} & X_n})_{m\leq n}$ in $\mathsf{Met}_{\infty}$, one obtains that its normed colimit $X\in\mathbb Y$, structured by (7.ii), satisfies the triangle inequality, as follows. Given any $\varepsilon, \eta>0$, one chooses $K\in\mathbb N$ such that $\varepsilon, \eta\geq\log^{\circ}\mathrm L(s_{m,n})$ or, equivalently, $e^{\varepsilon},e^{\eta}\geq \mathrm L(s_{m,n})$, for all $n\geq m\geq K$. Then, following the same steps as in the proof of Theorem \ref{VLiptheorem}, one shows that
$e^{\varepsilon}X(x,y)+e^{\eta}X(y,z)\geq X(x,z)$
holds for all $x,y,z\in X$, which implies the triangle inequality.
Hence, we proved:

\begin{corollary}\label{MetCauchycocomplete}
The normed category $\MET_{\infty}$ is Cauchy cocomplete.
\end{corollary}

\section{Semi-normed and normed vector spaces}
When a norm function $\nv\text{-}\nv$ on a (real, say) vector space $X$ satisfies the standard axioms for a norm, except that non-zero vectors are allowed to have norm $0$, one usually calls $X$ semi-normed. Here, just like for the metric of a Lawvere metric space, we  initially abandon not only the separation condition, but also the finiteness condition for norms. This then necessitates the extension of the real multiplication to $\infty$, so that we can maintain the norm axiom for scalar multiples of vectors with norm $\infty$. This leads us naturally to considering the quantale $\mathcal R_{\times}$ of Examples \ref{realswithmult}(2) and the adjunction $e\dashv \log^{\circ}$ of (\ref{elogadjunction}), so that we can then establish the Cauchy cocompleteness of the normed category of semi-normed vector spaces using the corresponding result for metric spaces (Corollary \ref{MetCauchycocomplete}).

The ordinary categories of semi-normed and of normed vector spaces are defined as follows:

\begin{definition}
A {\em semi-norm} on a (real) vector space $X$ is a function $\nv\text{-}\nv:X\to [0,\infty]$ satisfying
\begin{itemize}
\item[(N0)] $\nv0\nv=0$,
\item[(N1)] $\nv a x\nv=|a|\nv x\nv$,
\item[(N2)] $\nv x+y\nv\leq\nv x\nv +\nv y\nv$,
\end{itemize}
for all $x,y\in X$ and $a\in\mathbb R, a\neq 0$.\footnote{We exclude $a=0$ since otherwise (N1) would contradict (N0), in light of $0\cdot\infty =\infty$ in the quantale $\mathcal R_{\times}$; see Examples \ref{realswithmult}(2).} The thus defined {\em semi-normed vector spaces} are the objects of the category
$\mathsf{SNVec}_{\infty}$ whose morphisms are arbitrary linear maps. It contains the full subcategory $\mathsf{NVec}_{\infty}$  of {\em normed vector spaces}; its objects $X$ satisfy, for all $x\in X$, also the separation condition
\begin{itemize}
\item[(N3)] $\nv x\nv=0\;\Longrightarrow\; x=0.$
\end{itemize}
They fall short of being normed vector spaces in the classical sense only insofar as here vectors are permitted to have infinite norms.
\end{definition}
\begin{facts}\label{SNVecfacts}
(1) The additive group $(X,+)$ underlying a semi-normed vector space is a normed group (in the sense of Facts \ref{thirdremarks}(5)) and, hence, becomes a Lawvere metric space  in the standard way, {\em i.e.,} $X(x,y)=|\!|x-y|\!|$ for all $x,y\in X$. This defines the ordinary forgetful functor
$U: \mathsf{SNVec}_{\infty}\longrightarrow\mathbb Y$, with $\mathbb Y$ the category used to justify Corollary \ref{MetCauchycocomplete}.

(2) The category $\mathbb Y$ is $\mathcal R_{\times}$-normed via the Lipschitz value $L(f)$ of an arbitrary map $f:X\to Y$(see Example \ref{realswithmult}(2)). If $f$ happens to be a linear map of semi-normed vector spaces (so that $f$ is an $U$-image), as a trivial instance of (\ref{initialnorm}) we can take $L(f)$ as its $\mathcal R_{\times}$-norm and thereby make
$\mathsf{SNVec}_{\infty}$ $\mathcal R_{\times}$-normed. The formula for $L(f)$ then simplifies to
\begin{equation}\label{multnormonSNVec}
L(f)=\sup_{x,x'\in X}\frac{|\!|fx- fx'|\!|}{|\!|x-x'|\!|}=\sup_{x\in X}\frac{|\!|fx|\!|}{|\!|x|\!|}\;,
\end{equation}
satisfying the inequalities (both generally proper)
$$L(\mathrm{id}_X)\leq 1 \quad\mbox{and}\quad L(g\cdot f)\leq L(f)L(g)\;.$$
Hence, with this structure we obtain the $\mathcal R_{\times}$-normed category $\mathsf{SNVec}_{\infty}^{\times}$ and the norm-preserving functor $U: \mathsf{SNVec}_{\infty}^{\times}\longrightarrow\mathbb Y$.
(In global categorical terms, (\ref{multnormonSNVec})  puts the Cartesian (or initial) structure as in (\ref{initialnorm}) on $\mathsf{SNVec}_{\infty}$ with respect to the forgetful functor $\mathsf{CAT}/\!/\mathcal R_{\times}\to\mathsf{CAT}$, where we have extrapolated this trivial  aspect of Theorem \ref{CatVproperties} from small to large categories).

 (3) If we subject the $\mathcal R_{\times}$-normed functor $U$ to an application of the change-of-base functor $\mathcal B_{\log^{\circ}}$ of Proposition \ref{changeofbaseprop}, where $\log^{\circ}:\mathcal R_{\times}\to\mathcal R_+$, then (without name change for $U$) we can regard $U$ as an $\mathcal R_+$-normed functor $\mathcal B_{\log^{\circ}}(\mathsf{SNVec}_{\infty}^{\times})\longrightarrow\mathcal B_{\log^{\circ}}(\mathbb Y)$ that actually takes values in $\MET_{\infty}$, {\em i.e.} 
 $$ U:\mathsf{SNVec}_{\infty}\longrightarrow\MET_{\infty} . $$
 This means that, henceforth, we will always consider $\mathsf{SNVec}_{\infty}$ as an $\mathcal R_+$-normed category, provided with its
{\em logarithmic norm} given (as in (\ref{NVecnorm})) by
\begin{equation}\label{logarithmicnorm}
|f:X\to Y|=\log^{\circ} L(f)=\sup_{x\in X}\,\log^{\circ}\frac{\nv fx\nv}{\nv x\nv}
\end{equation}
(with the second equality holding since $\log^{\circ}$ preserves suprema with respect to the natural order of $[0,\infty]$), and that $U$ stays to be norm-preserving.
\end{facts}

We list some easily seen properties of the logarithmic norm of $\mathsf{SNVec}_{\infty}$ before getting to its Cauchy cocompleteness.
\begin{lemma}\label{SNVeclemma}
Let $f:X\to Y$ be a linear map of semi-normed vector spaces. Then:

{\em (1)} If $X$ contains a vector $x_0$ with $\nv x_0\nv=0$ and $\nv fx_0\nv\neq 0$, then $|f|=\infty$.

{\em (2)} If $\nv x\nv=0$ always implies $\nv fx\nv=0$, then $|f|=\sup_{\nv x\nv=1}\,(\log^{\circ}\nv fx\nv)\;.$

{\em(3)} For all $x\in X$ one has $\nv fx\nv\leq e^{|f|}\, \nv x\nv$, and $|f|$ is minimal with that property.

{\em (4)} One has $|f|=0$ if, and only if, $\nv fx\nv\leq\nv x\nv$ holds for all $x\in X$.
\end{lemma}
\begin{proof}
(1) From $\log^{\circ}\frac{\nv fx_0\nv}{\nv x_0\nv}\leq|f|$ one obtains $\infty=\frac{\nv fx_0\nv}{\nv x_0\nv}\leq e^{|f|}$ and, hence, $|f|=\infty$.

(2) Trivially $t:=\sup_{\nv x\nv=1}\,(\log^{\circ}\nv fx\nv)\leq|f|$. For ``$\geq$'' consider any $x\in X$. If $\nv x\nv=0$, also $\nv fx\nv=0$ by hypothesis, and $\log^{\circ}\frac{\nv fx\nv}{\nv x\nv}=0\leq t$ follows; likewise if $\nv x\nv=\infty$. In all other cases one considers $x_1:=\frac{1}{\nv x\nv}x$ in a standard manner.

(3) $L(f)$ is minimal with $|\!|fx|\!|\leq L(f)|\!|x|\!|$ for all $x\in X$, so the first claim follows from $L(f)\leq e^{\log^{\circ}L(f)}=e^{|f|}$.
If, for some $\alpha$, $\frac{\nv fx\nv}{\nv x\nv}\leq e^{\alpha}$ for all $x$, then $\log^{\circ}\frac{\nv fx\nv}{\nv x\nv}\leq \alpha$ for all $x$, {\em i.e.} $|f|\leq\alpha$.

(4) $|f|=\log^{\circ}L(f)=0$ holds if, and only if, $L(f)\leq 1$ or, equivalently,  $f$ is 1-Lipschitz. \end{proof}

\begin{theorem}\label{seminormedtheorem}
With its logarithmic norm, $\mathsf{SNVec}_{\infty}$ is a Cauchy-cocomplete normed category whose $0$-morphisms are precisely the $1$-Lipschitz linear maps: $(\mathsf{SNVec}_{\infty})_{\circ}=\mathsf{SNVec}_1$.
\end{theorem}
\begin{proof}
From Facts \ref{SNVecfacts}(3) we know that
 $\mathsf{SNVec}_{\infty}$ is indeed a normed category, and its $0$-morphisms have been identified in Lemma \ref{SNVeclemma}(4). So the only remaining issue is its Cauchy cocompleteness.
 We consider a Cauchy sequence $s=(\xymatrix{X_m\ar[r]^{s_{m,n}}&X_n\\})_{m\leq n}$ in $\mathsf{SNVec}_{\infty}$. Since the functor $U:\mathsf{SNVec}_{\infty}\to \MET_{\infty}$ is norm-preserving (Facts \ref{SNVecfacts}(3)), $s$ is also a Cauchy sequence in $\MET_{\infty}$ and, by Corollary \ref{MetCauchycocomplete}, has a normed colimit, witnessed by a colimit cocone $(\gamma_n:X_n\to X)$. Since the forgetful functor $\MET_{\infty} \to \SET$ (trivially) preserves colimits and the forgetful functor $\mathsf{Vec}\to\SET$ of the algebraic category of vector spaces creates directed colimits, the metric space $X$ carries a vector space structure that makes $\gamma$ a colimit cocone in $\mathsf{Vec}$.

 By (\ref{defofcolimitmetric}) in the proof of Theorem \ref{VLiptheorem}, the metric structure of $X$ is given by
$$ X(x,y)=\sup_{N\in\mathbb N}\Phi_{x,y}(N)\mbox{ with }\Phi_{x,y}(N)=\inf\{X_n(x',y')\mid n\geq N,\;x'\in\gamma_n^{-1}x,\,y'\in\gamma_n^{-1}y\},$$
with $\Phi_{x,y}$ is monotonely increasing in $N$. We claim that, since the metric on each $X_n$ is induced by the norm on $X_n$ (Facts \ref{SNVecfacts}(1)) and therefore invariant under translation,  the same is true for
 the metric on $X$, {\em i.e.}, $X(x+z,y+z)=X(x,y)$ for all $z\in X$. Indeed, given $z$, since the colimit maps $\gamma_n$ are collectively epic, one has  $z=\gamma_K(z')$ for some $z'\in X_K$. Then, for all $n\geq N\geq K$ and  $x'\in\gamma_n^{-1}x,\,y'\in\gamma_n^{-1}y$ one obtains $X_n(x',y')=X_n(x'+s_{K,n}z',y'+s_{K,n}z')$ and thereby
 $$X(x,y)=\sup_{N\geq K}\Phi_{x,y}(N)=\sup_{N\geq K}\Phi_{x+z,y+z}(N)=X(x+z,y+z).$$
  By Facts \ref{thirdremarks}(4),(5), its translation-invariant metric makes the additive group of $X$ normed when we put $|\!|x|\!|=X(x,0)$. Hence, the conditions (N0) and (N2) for a semi-normed vector space hold. Regarding condition (N1),
   {\em i.e.}, $\nv ax\nv=|a|\nv x\nv$ for all real $a\neq 0$ and $x\in X$, we note that this equality is an immediate consequence of the equivalence $(z\in \gamma_n^{-1}(ax)\iff w\in \gamma_n^{-1}x)$ whenever $z=aw$, and of the fact that the multiplication in $[0,\infty]$ by the positive real  number $|a|$ preserves both, infima and suprema.

 It remains to be shown that $X$ is a normed colimit of $s$ in $\mathsf{SNVec}_{\infty}$. But (C1) of Definition \ref{defnormedcolimit} holds trivially by construction of $X$ as a colimit in $\mathsf{Vec}$, and also (C2) in the form (\ref{ConditionC2}) holds in $\mathsf{SNVec}_{\infty}$ since it holds in $\MET_{\infty}$, and since $U:\mathsf{SNVec}_{\infty} \to\MET_{\infty}$ is norm-preserving. 
\end{proof}

\begin{remarks}\label{seminormedremarks}
(1) The full normed subcategory $\mathsf{NVec}_{\infty}$ of  $\mathsf{SNVec}_{\infty}$ fails to be closed under the formation of normed colimits of Cauchy sequences. Even for a Cauchy sequence $s$ of (strictly contractive) linear maps $s_{m,n}: X_m\to X_n$ of normed vector spaces (with all norms finite), the normed colimit in  $\mathsf{SNVec}_{\infty}$ may fail to be a normed vector space. Indeed, consider the sequence (\ref{sequenceofreals}) of the Introduction; that is: $X_n:=\mathbb R$ normed by $\nv 1\nv_n=\frac{1}{n}$ and $s_{n,m}=\mathrm{id}_{\mathbb R}$ for all $m\leq n$. The normed colimit of $s$ in $\mathsf{SNVec}_{\infty}$
may again be formed by identity maps, $\gamma_n:X_n\to X=\mathbb R$, with the norm in $X$ given by
$\nv 1\nv=\sup_N\inf_{n\geq N}\nv 1\nv_n=0,$
{\em i.e.}, all norms in $X$ are $0$, so that the separation condition (N3) fails to the largest extent possible.

(2) Almost all of the morphisms $\gamma_n:X_n\to X$ of a normed colimit cocone $\gamma$ for a convergent sequence in $\mathsf{SNVec}_{\infty}$ have an important extra property beyond their linearity:  $\nv\gamma_nz\nv=0$ holds for all $z\in X_n$ with $\nv z\nv=0$. Indeed,  since $\inf_N\sup_{n\geq N}|\gamma_n|=0$ by (C2a), there is some $N\in \mathbb N$ with $|\gamma_n|\leq 1$ for all $n\geq N$. Consequently, with Lemma \ref{SNVeclemma}(3) one obtains $\nv\gamma_nz\nv\leq e^{|\gamma_n|}\nv z\nv\leq e\nv z\nv=0$.
\end{remarks}

\begin{definition}\label{zerotozerodef}
A linear map $f:X\to Y$ of semi-normed vector spaces is called a {\em zero-to-zero morphism} if $\nv fx\nv=0$ holds for all $x\in X$ with $\nv x\nv=0$. (We note that every bounded operator is a zero-to-zero morphism.) We denote by
$\mathsf{SNVec}_{00}$ the (non-full) subcategory of $\mathsf{SNVec}_{\infty}$ containing all semi-normed vector spaces and their zero-to-zero morphisms. 
The separation condition of its domain makes a morphism in $\mathsf{NVec}_{\infty}$, {\em i.e.}, an arbitrary linear map of normed vector spaces, automatically a zero-to-zero morphism. Therefore, $\mathsf{NVec}_{\infty}$ is a full subcategory of $\mathsf{SNVec}_{00}$.
\end{definition}


\begin{lemma}\label{normedproposition}
The normed category $\mathsf{NVec}_{\infty}$ is reflective in the normed category $\mathsf{SNVec}_{00}$, as $(\Setss{\mathcal R_+})$-enriched categories.
\end{lemma}

\begin{proof}
For $X\in\mathsf{SNVec}_{00}$ consider its subspace $X_0:=\{x\in X\mid \nv x\nv=0\}$ and let $p:X\to X/X_0$ be the projection.
Since
$\nv x\nv=\nv y+(x-y)\nv\leq\nv y \nv+\nv x-y\nv$
one has $(\nv x-y\nv=0 \Longrightarrow \nv x \nv=\nv y\nv)$ for all $x,y\in X$, so that $\nv px\nv:=\nv x\nv$ makes $X/X_0$ a well-defined object of $\mathsf{NVec}_{\infty}$ and $p$ a zero-to-zero morphism -- in fact, an isometry. Furthermore, for all $Y\in \mathsf{NVec}_{\infty}$ we have the natural bijection
$$-\cdot p:\mathsf{NVec}_{\infty}(X/X_0,Y)\to\mathsf{SNVec}_{00}(X,Y),$$
whose surjectivity is guaranteed by our restriction to zero-to-zero morphisms (as opposed to all linear maps of semi-normed vector spaces). In fact, this bijection is a $(\Setss{\V})$-isomorphism since, for every linear map $f:X/X_0\to Y$, one has
$$|f|=\sup_{z\in X/X_0}(\log^{\circ}\frac{\nv fz\nv}{\nv z\nv})=\sup_{x\in X}\;(\log^{\circ}\frac{\nv f(px)\nv}{\nv px\nv})=\sup_{x\in X}\;(\log^{\circ}\frac{\nv f(px)\nv}{\nv x\nv})=|f\cdot p|\;.  $$
\end{proof}

\begin{corollary}\label{NVecCauchycocomplete}
A sequence in $\mathsf{NVec}_{\infty}$ that has a normed colimit in $\mathsf{SNVec}_{00}$ has also a normed colimit in $\mathsf{NVec}_{\infty}$.
\end{corollary}
\begin{proof} Let $s$ be a sequence in $\mathsf{NVec}_{\infty}$ with normed colimit cocone in $\gamma:s\to \Delta X$ in
$\mathsf{SNVec}_{00}$.
We apply the reflector to $X$ to obtain the colimit cocone $(p\cdot\gamma_n: X_n\to X/X_{0})_n$ in the ordinary category $\mathsf{NVec}_{\infty}$. As one easily checks (or formally derives with Proposition \ref{d:prop:1} proved below), since the adjunction of Proposition \ref{normedproposition} is ($\Setss{\mathcal R_+}$)-enriched, this cocone presents $X/X_0$ in fact as a normed colimit of $s$ in $\mathsf{NVec}_{\infty}$.
\end{proof}

The restrictive condition in Corollary \ref{NVecCauchycocomplete} is essential, as witnessed by the sequence \eqref{sequenceofreals}:

\begin{proposition}\label{NVecnotCauchycocomplete}
The normed categories $\mathsf{SNVec}_{00}$ and $\mathsf{NVec}_{\infty}$ are not Cauchy cocomplete.
\end{proposition}
\begin{proof}
Suppose the Cauchy sequence of \eqref{sequenceofreals} (see also Remarks \ref{seminormedremarks}(1)) admits a normed colimit $(\gamma_n:\mathbb R_{\frac{1}{n}}\to X)_n$ in 
 $\mathsf{SNVec}_{00}$ or $\mathsf{NVec}_{\infty}$. As a necessarily one-dimensional vector space, $X$ may be taken to be $\mathbb R$, normed by $\nv x\nv=|x|\nv 1\nv$ for $x\neq 0$, \emph{i.e.} $X=\mathbb R_{c}$ with some $c=\nv 1\nv\in[0,\infty]$. Consider the cocone $(\varphi_n:\mathbb R_{\frac{1}{n}}\to\mathbb R_1=\mathbb R)_n$ comprised of identity maps and its induced morphism $f:\mathbb R_c\to \mathbb R$. As a linear map, $f$ must be the multiplication by some $\alpha\in \mathbb R$. Hence, on one hand we have $|f|=\log_{\circ}\frac{|\alpha|}{c}$; on the other hand, the normed colimit condition requires $|f|=\inf_N\sup_{n \geq N}|\varphi_n|=\inf_N\sup_{n\geq N} \log_{\circ }n=\infty$. This is possible only when $c=0$ which, however, prevents $X$ from being separated and $f$ from satisfying the zero-to-zero condition---a contradiction.
 \end{proof}

\begin{example}

A normed colimit in $\mathsf{NVec}_{\infty}$ of a Cauchy sequence of Banach spaces need not be Banach. Indeed,
consider the sequence
$$\xymatrix{\{0\}\ar[r] &\mathbb R\ar[r] & \mathbb R^{2}\ar[r] & \mathbb R^{3}\ar[r] & ...\ar[r] & \mathrm{colim}_n\;\mathbb R^{n}=\bigoplus_n\mathbb R=\mathbb R^{(\infty)}\\
}$$
of isometric embeddings of finite-dimensional spaces, normed by the $\max$-norm. As one checks easily, its normed colimit $\mathsf{NVec}_{\infty}$ exists and is given by the direct sum (in $\mathsf{Vec}$) of countably many copies of $\mathbb R$, with a norm that makes the colimit cocone consist of isometries. In the direct sum, we have the Cauchy sequence $(x_n)_n$, where the $i$-th component of $x_n$ is $\frac{1}{i+1}$ for $i\leq n$, and $0$ otherwise, but the sequence does not converge in
$\mathbb R^{(\infty)}$.
\end{example}

\section{Cauchy cocompleteness of normed presheaf categories}
We continue to work with a quantale $(\V,\leq,\otimes,\kk)$ and first consider an arbitrary sequence $s=\xymatrix{(A_m\ar[r]^-{s_{m,n}} & A_n)_{m\leq n}}$ in  $\nSetss{\V}$. So, while the sets $A_n$ are $\V$-normed, the maps $s_{m,n}$ may not be. Still, with the forgetful functor $U:\nSetss{\V}\to\SET$, we can form the colimit $A$ of $Us$ in $\SET$, with cocone $\xymatrix{(A_n\ar[r]^{\gamma_n} & A)_n}$. Trivially, {\em any} norm on $A$ makes the resulting $\V$-normed set a colimit of $s$ in $\nSetss{\V}$, since there is no constraint on the morphisms in that category. But there is one norm on $A$ that distinguishes itself by a special property, as follows. $$ |c|=\bigwedge_{N\in\mathbb N}\bigvee_{n\geq N}\bigvee_{a\in\gamma_n^{-1}c}|a|;$$
that is, we  employ the same formula as the one established for colimits of sequences in $\SET/\!/\V$ (see Proposition \ref{finalinSetV}), but now without any {\em a-priori} expectation that it would make the maps $\gamma_n$ $\V$-normed. We call the above norm on $A$ the {\em $\gamma$-induced Cauchy norm} since it has the important property (C2b) (see Corollary \ref{Condition2ab}):

\begin{lemma}\label{SetVCauchycoc}
 Let the set $A$ be provided with the $\gamma$-induced Cauchy norm as above. Then, for any mapping $f:A\to B$ to a $\V$-normed set $B$, one has
 $|f|\geq\bigvee_{N\in\mathbb N} \bigwedge_{n\geq N}|f\cdot\gamma_n|\,.$
\end{lemma}
\begin{proof}
Since, in its first (contravariant) variable, the internal hom $[\text{-,-}]$ of $\V$ transforms arbitrary joins into meets, we have:
\begin{align*}
|f|=\bigwedge_{c\in A}[|c|,|fc|] & =\bigwedge_{c\in A}[\bigwedge_{N}\bigvee_{n\geq N}\bigvee_{a\in\gamma_n^{-1}c}|a|,\,|fc|]   \\
& \geq \bigwedge_{c\in A}\bigvee_{N}[\bigvee_{n\geq N}\bigvee_{a\in\gamma_n^{-1}c}|a|,\,|fc|]\\
& \geq \bigvee_N\bigwedge_{c\in A}[\bigvee_{n\geq N}\bigvee_{a\in\gamma_n^{-1}c}|a|,\,|fc|]\\
& = \bigvee_N\bigwedge_{n\geq N}\bigwedge_{c\in A}\bigwedge_{a\in\gamma_n^{-1}c}[|a|,\,|fc|]\\
& = \bigvee_N\bigwedge_{n\geq N}\bigwedge_{a\in A_n}[|a|,\,|f(\gamma_na)|]\\
& = \bigvee_N\bigwedge_{n\geq N}|f\cdot\gamma_n|\;.
 \end{align*}
\end{proof}

In order to strengthen the assertion of Lemma \ref{SetVCauchycoc} and show that, in fact, $\nSetss{\V}$, and even all $\nSetss{\V}$-valued presheaf categories, are Cauchy cocomplete, we need a small additional hypothesis on the $\otimes$-neutral element $\kk$ of the quantale $\V$. Actually, we offer two alternative possibilities, (A) {\em or} (B), for suitably augmenting our general quantalic setting, as follows:

(A) \quad $\kk$ {\em is approximated from totally below} (see Theorem \ref{VLiptheorem}),  that is: $\kk=\bigvee\{\varepsilon\in\V \mid \varepsilon\lll\kk\}.$

(B) \quad $\kk$\; $\wedge$-{\em distributes over arbitrary joins}, that is: $\kk\wedge\bigvee_{i\in I}v_i=\bigvee_{i\in I}\kk\wedge v_i\;.$

\begin{remarks}\label{fifthremarks}
(1) Condition (A) certainly holds when the lattice $\V$ is {\em (constructively) completely distributive} in the sense of \cite{Wood04}. 

(2) Condition (B) trivially holds when the quantale $\V$ is {\em integral, i.e.,} when $\kk=\top$, and also when  the underlying lattice of the quantale $\V$ is a {\em frame} since, in the latter case {\em every} element in $\V$ $\wedge$-distributes over arbitrary  joins by definition, whilst in the former case the map $\kk\wedge(-)$ is just the identity map on $\V$.

(3) With the exception of those mentioned in Example \ref{otherexamples}(2), most quantales discussed in this paper satisfy {\em both} conditions, (A) and (B), and the others at least one. We examine the two conditions further in Section \ref{AvsB} and confirm their logical independence.
\end{remarks}

 We are now ready to prove the main general theorem of the paper.
\begin{theorem}\label{d:thm:1}
  When the quantale $\V$ satisfies condition {\em (A) or (B)}, then the $\V$-normed category $[\mathbb X,\nSetss{\V}]$ is Cauchy cocomplete, for every small $\V$-normed category $\mathbb X$.
 \end{theorem}

 \begin{proof}
 Considering a Cauchy sequence $ \sigma=\xymatrix{(P_m\ar[r]^-{\sigma_{m,n}}& P_n)_{m\leq n\in\mathbb N}}$ in the category
 $[\mathbb X,\nSetss{\V}]$ (given by all $\V$-normed functors $\mathbb X\to\nSetss{\V}$ and their natural transformations), 
  with the forgetful functor $U:\nSetss{\V}\to\SET$ we form the colimit $P$ of $U\sigma$ in the ordinary functor category $\SET^{\mathbb X}$, with cocone $\gamma=(\xymatrix{P_n\ar[r]^{\gamma_n} & P})_n$. Then, for every object $x$ in $\mathbb X$, the colimit $Px$ of the sequence $(U\sigma_{m,n}^x)_{m\leq n}$ in $\SET$ may be provided with the Cauchy norm induced by the cocone $(\gamma_n^x)_n$ (see Lemma \ref{SetVCauchycoc}), and in this way $P$ is then considered as a $\nSetss{\V}$-valued functor.

   In order to establish $P$ as a normed colimit of $\sigma$, by Corollary \ref{Condition2ab}, we must show:
 \begin{itemize}
 \item[(C1)] The functor $P:\mathbb X\to \nSetss{\V}$ is $\V$-normed (so that it serves as a colimit of $\sigma$ in the ordinary full subcategory $[\mathbb X,\nSetss{\V}]$ of $(\nSetss{\V})^{\mathbb X}$, formed by all $\V$-normed functors $\mathbb X\to\nSetss{\V}$);
 \item[(C2a)] $\gamma$ is a $\kk$-cocone, {\em i.e.}, $\kk\leq \bigvee_N\bigwedge_{n\geq N}|\gamma_n|$;
 \item[(C2b)] $|\alpha|\geq \bigvee_N\bigwedge_{n\geq N}|\alpha\cdot\gamma_n|$, for every natural transformation $\alpha:P\to Q$.
 \end{itemize}
 For showing (C1) we use (C2a) (the proof of which is presented further below, independently of (C1)) and, since every $P_n$ is $\V$-normed and every $\gamma_n=(\gamma_n^x)_{x\in\mathbb X}:P_n\to P$ is natural, obtain for all morphisms $f:x\to y$ in $\mathbb X$
 \begin{align*}
 |f|=\kk\otimes|f|&\leq(\bigvee_N\bigwedge_{n\geq N}|\gamma_n|)\otimes|f|\\
 &\leq\bigvee_N(\bigwedge_{n\geq N}|\gamma_n^y|\otimes|f|)\\
&\leq \bigvee_N\bigwedge_{n\geq N}|\gamma_n^y|\otimes |P_nf|\\
&\leq \bigvee_N\bigwedge_{n\geq N}|\gamma_n^y\cdot P_nf|\\
&= \bigvee_N\bigwedge_{n\geq N}|Pf\cdot\gamma_n^x|\leq|Pf|\,,
 \end{align*}
 with the last inequality following from the fact that $Px$ carries the $(\gamma_n^x)_n$-induced Cauchy norm, so that Lemma \ref{SetVCauchycoc} applies.

 Using this last argument again for every $x\in\mathbb X$, and before turning to the more cumbersome proof of (C2a), we can immediately show that condition (C2b) holds, as follows:
 $$|\alpha|=\bigwedge_{x\in X}|\alpha_x|\geq\bigwedge_{x\in\mathbb X}\bigvee_N\bigwedge_{n\geq N}|\alpha_x\cdot\gamma_n^x|\geq\bigvee_N\bigwedge_{n\geq N}\bigwedge_{x\in\mathbb X}|\alpha_x\cdot\gamma_n^x|=\bigvee_N\bigwedge_{n\geq N}|\alpha\cdot\gamma_n|\,.    $$

 For the proof of (C2a), we first calculate
 \begin{align*}
 \bigvee_N\bigwedge_{n\geq N}|\gamma_n| &= \bigvee_N\bigwedge_{n\geq N}\bigwedge_{x\in\mathbb X}|\gamma_n^x| \\
 & =  \bigvee_N\bigwedge_{n\geq N}\bigwedge_{x\in\mathbb X}\bigwedge_{a\in P_nx}[|a|,|\gamma_n^xa|] \\
 &=  \bigvee_N\bigwedge_{x\in\mathbb X}\bigwedge_{c\in Px}\bigwedge_{n\geq N}\bigwedge_{a\in(\gamma_n^x)^{-1}c}[|a|,|c|] \\
  &=  \bigvee_N\bigwedge_{x\in\mathbb X}\bigwedge_{c\in Px}[\;\bigvee_{n\geq N}\bigvee_{a\in(\gamma_n^x)^{-1}c}|a|,\; \bigwedge_M \bigvee_{m\geq M}\bigvee_{b\in(\gamma_m^x)^{-1}c}|b|\; ] \\
&=  \bigvee_N\bigwedge_M\bigwedge_{x\in\mathbb X}\bigwedge_{c\in Px}[\,|\!|c|\!|_N, |\!|c|\!|_M\,]\,,\qquad\qquad\qquad\qquad\qquad\qquad\qquad\qquad(*)
   \end{align*}
   where, for the last equality, we have used the abbreviation $|\!|c|\!|_N:=\bigvee_{n\geq N}\bigvee_{a\in(\gamma_n^x)^{-1}c}|a|$ for all $N\in \mathbb N, x\in\mathbb X$, and $c\in Px$.

   We now consider the alternative hypotheses (A) and (B) and finish the proof under each of them separately, as follows.

   (A) Since the Cauchy sequence $\sigma$ satisfies $\kk\leq\bigvee_N\bigwedge_{n\geq m\geq N}|\sigma_{m,n}|$, for every $\varepsilon\lll \kk$ in $\V$ we find an $N\in\mathbb N$ with $\varepsilon\leq\bigwedge_{n\geq m\geq N}|\sigma_{m,n}|$, {\em i.e.},
   $\varepsilon\leq |\sigma_{m,n}^x|$
   for all $n\geq m\geq N$ and $x\in \mathbb X$. Now, given any $c\in Px$ and  $M\in \mathbb N$, in the case $M\leq N$ we trivially have $|\!|c|\!|_N\leq|\!|c|\!|_M$ and obtain $\varepsilon\ll\kk\leq [ |\!|c|\!|_N,|\!|c|\!|_M]$, so certainly $\varepsilon\leq [ |\!|c|\!|_N,|\!|c|\!|_M]$   . If $M\geq N$, with $\ell:=M-N$ we have
   \begin{align*}
 |\!|c|\!|_M =\bigvee_{m\geq M}\bigvee_{b\in(\gamma_m^x)^{-1}c}|b| & \geq \bigvee_{n\geq N}\bigvee_{a\in(\gamma_n^x)^{-1}c}|\sigma_{n,n+\ell}^x\,a|  \\
& \geq \bigvee_{n\geq N}\bigvee_{a\in(\gamma_n^x)^{-1}c}|a|\otimes|\sigma_{n,n+\ell}^x| \\
& \geq (\bigvee_{n\geq N}\bigvee_{a\in(\gamma_n^x)^{-1}c}|a|\,)\otimes \varepsilon=|\!|c|\!|_N\otimes \varepsilon\,,
  \end{align*}
  which again implies $\varepsilon\leq [ |\!|c|\!|_N,|\!|c|\!|_M]$. Consequently, since $\kk=\bigvee\{\varepsilon\mid \varepsilon\lll\kk\}$, with $(*)$ we obtain $\kk\leq \bigvee_N\bigwedge_{n\geq N}|\gamma_n|$, as desired.

  (B) Analyzing further the equality $(*)$, we have
  $$\bigvee_N\bigwedge_{n\geq N}|\gamma_n|= \bigvee_N\bigwedge_{x\in\mathbb X}\bigwedge_{c\in Px}[\,|\!|c|\!|_N, \bigwedge_M|\!|c|\!|_M\,]\,,\;\mbox{ with }  $$
  \begin{align*}
  [\,|\!|c|\!|_N, \bigwedge_M|\!|c|\!|_M\,] &=[\,|\!|c|\!|_N,\bigwedge_{M\leq N} |\!|c|\!|_M\wedge \bigwedge_{M\geq N}|\!|c|\!|_M\,]\\
  &=[\,|\!|c|\!|_N,\bigwedge_{M\leq N} |\!|c|\!|_M\,]\;\wedge\; [\,|\!|c|\!|_N,\bigwedge_{M\geq N}|\!|c|\!|_M\,]  \\
  & \geq \kk \wedge [\,|\!|c|\!|_N,\bigwedge_{M\geq N}  |\!|c|\!|_M\,] \\
 & =\bigwedge_{M\geq N}(\kk\wedge[\,|\!|c|\!|_N,  |\!|c|\!|_M\,])\,.
    \end{align*}
 Here, for $M\geq N$, as in part (A), setting $\ell =M-N$ one has
  $$|\!|c|\!|_M\geq  \bigvee_{n\geq N}\bigvee_{a\in(\gamma_n^x)^{-1}c}|a|\otimes|\sigma_{n,n+\ell}^x| = |\!|c|\!|_N\otimes \bigvee_{n\geq N}|\sigma_{n,n+\ell}^x|\geq |\!|c|\!|_N\otimes |\sigma_{N,M}^x|$$
  and, hence, $ [\,|\!|c|\!|_N, |\!|c|\!|_M\,]\geq |\sigma_{N,M}^x| $.  Consequently, with hypothesis (B) and the Cauchyness of $\sigma$ we obtain
  \begin{align*}
 \bigvee_N\bigwedge_{n\geq N}|\gamma_n| & \geq \bigvee_N\bigwedge_{x\in\mathbb X}\bigwedge_{M\geq N}(\kk\wedge|\sigma_{N,M}^x|) \\
  &=\bigvee_N(\kk\wedge\bigwedge_{M\geq N}\bigwedge_{x\in\mathbb X}|\sigma_{N,M}^x|)   \\
  &=\kk\wedge\bigvee_N\bigwedge_{M\geq N}|\sigma_{N,M}| \\
   & \geq\kk\wedge\kk=\kk\,,
  \end{align*}
  which concludes the proof.
   \end{proof}
In conjunction with Remarks \ref{fifthremarks} we conclude:

   \begin{corollary}
     For every small $\V$-normed category $\mathbb X$, the $\V$-normed category $[\mathbb X,\nSetss{\V}]$ is Cauchy cocomplete under any of the following hypotheses:
     \begin{itemize}
     \item the quantale $\V$ is integral;
     \item the lattice $\V$ is a frame;
     \item the lattice $\V$ is (constructively) completely distributive.
     \end{itemize}
   \end{corollary}

We don't have an answer to the following question:
  \begin{problem}
  Is there a quantale $\V$ (and a small $\V$-normed category $\ncatX$) for which the $\V$-normed category (of presheaves of $\ncatX$ with values in) $\nSetss{\V}$ fails to be Cauchy cocomplete?
   \end{problem}

 \begin{remark}\label{ABremark}
By Theorem \ref{d:thm:1}, any quantale $\V$ for which $\nSetss{\V}$ is not Cauchy cocomplete must fail Conditions (A) and (B). As observed in \cite{GH24}, after providing any (Abelian) group $G$ with the discrete order, one may consider its MacNeille completion $G_{\bot}^{\top}$ as a (commutative) quantale (see Example \ref{otherexamples}(2)). Moreover, when $G$ has at least order 3, one easily shows that (A) and (B) both fail in $G_\bot^\top$ (see also Section 15). 
 \end{remark}

  \section{Normed colimits as weighted colimits}

In this and the next section {\em we assume that the quantale $\V$ satisfies condition} (A) \emph{or} (B) so that we can apply Theorem~\ref{d:thm:1}. Under this condition, we show that normed colimits of Cauchy sequences can be equivalently described as {\em weighted} (formerly {\em indexed}) colimits in the sense of \cite{Kel82}, for an appropriate class of weights. By Definition \ref{defnormedcolimit}, a normed colimit of a sequence \(s\) in a \(\V\)-normed category \(\ncatX\) is equivalently given by an object \(x\) of \(\ncatX\) together with bijections
\begin{displaymath}
  \Nat(s,\Delta z)\cong\ncatX(x,z),
\end{displaymath}
naturally in \(z\), so that the induced cocone
 $ (\kappa_N:\Nat(s_{|N},\Delta z)\longrightarrow\ncatX(x,z))_{N\in\NN}$
is a colimit in \(\Setss{\V}\); with $z=x$ they determine the normed colimit cocone $\gamma=\kappa_0^{-1}(1_x):s\to\Delta x$.

We start with the following observation.

\begin{proposition}\label{d:prop:1}
  Every \(\V\)-normed left adjoint functor \(F \colon\ncatX \to\ncatY\)  (in the $\Setss{\V}$ enriched sense)  preserves normed colimits of sequences.
\end{proposition}
\begin{proof}
  Let \(G \colon \ncatY\to\ncatX\) be right adjoint of \(F\) in \(\Catss{\V}\), so that we have isomorphisms
  \begin{displaymath}
    \ncatX(x,Gy)\longrightarrow \ncatY(Fx,y),
  \end{displaymath}
  in \(\Setss{\V}\), naturally in \(x\) and \(y\). Therefore, for every sequence $s$,
  every \(N\in\NN\) and every object \(y\) in \(\ncatY\), we also have an isomorphism
 $ \Nat(s_{|N},\Delta Gy)\longrightarrow \Nat(Fs_{|N},\Delta y)$
  in \(\Setss{\V}\), naturally in \(s\) and \(y\). Let \(x\) with cocone $\gamma$ be a normed colimit of \(s\) in \(\ncatX\). For every object \(y\) in \(\ncatY\), the diagram
  \begin{displaymath}
    \begin{tikzcd}
      \Nat(s_{|N},\Delta Gy) %
      \ar{r}{\sim} %
      \ar{d}[swap]{\kappa_N } %
      & \Nat(Fs_{|N},\Delta y) %
      \ar{d}{ } \\
      \ncatX(x,Gy) %
      \ar{r}{\sim} %
      & \ncatY(Fx,y) %
    \end{tikzcd}
  \end{displaymath}
  commutes, with the inverse of the right vertical isomorphism for $y=Fx$ sending $1_{Fx}$ to $F\gamma$. Since the cocone \((\Nat(s_{|N},\Delta Gy)\to\ncatX(x,Gy))_{N\in\NN}\) is a colimit in \(\Setss{\V}\), so is the cocone \((\Nat(Fs_{|N},\Delta y)\to\ncatY(Fx,y))_{N\in\NN}\). This proves that \(Fx\)  with cocone $F\gamma$ is a normed colimit of \(Fs\) in \(\ncatY\).
\end{proof}

Recall from \cite{Kel82} that, for \(\V\)-normed functors \(F \colon\ncatA\to\ncatX\) and \(\phi \colon\ncatA^{\op}\to\nSetss{\V}\), a \emph{\(\phi\)-weighted} colimit of \(F\) is given by an object \(x\) in \(\ncatX\) together with  isomorphisms in $\Setss{\V}$
\begin{equation}\label{d:eq:2}
  \ncatX(x,y)\cong\Nat(\phi,\ncatX(F-,y)),
\end{equation}
naturally in \(y\). In this context it is convenient to use the language of \((\nSetss{\V})\)-valued distributors \(\ncatX\lmodto\ncatY\) which, just like the $\V$-valued distributors in Section 4,  are defined as \(\V\)-normed functors \(\ncatX^{\mathrm{op}}\otimes\ncatY\to\nSetss{\V}\). Every \(\V\)-normed functor \(F \colon\ncatX\to\ncatY\) induces a pair of distributors
\begin{align*}
  F_{*}& \colon\ncatX\lmodto\ncatY,\quad F_{*}(x,y)=\ncatY(Fx,y),\\
  F^{*}& \colon\ncatY\lmodto\ncatX,\quad F^{*}(y,x)=\ncatY(y,Fx).
\end{align*}
In particular, interpreting an object \(x\) in \(\ncatX\) as a \(\V\)-normed functor \(x \colon\ncatOne\to\ncatX\) from the monoidal unit \(\ncatOne=\ncatOne_{\kk}\) to \(\ncatX\), one obtains
\begin{displaymath}
  x_{*}\colon\ncatOne\lmodto\ncatX,\quad x_{*}=\ncatX(x,-)
  \quad\text{and}\quad
  x^{*}\colon\ncatX\lmodto\ncatOne,\quad x^{*}=\ncatX(-,x).
\end{displaymath}
For distributors \(\phi \colon\ncatX\lmodto\ncatA\) and \(\psi \colon\ncatA\lmodto\ncatX\) and objects \(x\) in \(\ncatX\) and \(y\) in \(\ncatY\), one considers
\begin{displaymath}
  (\psi\cdot\phi)(x,y)\cong\int^{a\in\ncatA}\psi(a,y)\otimes\phi(x,a)
\end{displaymath}
whenever this coend exists (see \cite{MacLane98, Lor21}). This is certainly the case when \(\ncatA\) is small (since {\color{red} $\Setss{\V}$} is small-cocomplete), and then the formula above defines the composite distributor \(\psi\cdot\phi \colon\ncatX\lmodto\ncatY\). Another important case is  \(\phi=G_{*}\) for a \(\V\)-normed functor \(G \colon\ncatX\to\ncatA\), since then one simply has
\begin{displaymath}
  (\psi\cdot G_{*})(x,y)\cong\psi(Gx,y)
\end{displaymath}
for all objects \(x\) in \(\ncatX\) and \(y\) in \(\ncatY\). Hence, the presheaf \(\ncatX(F-,y)\)  in \eqref{d:eq:2} with \(F \colon\ncatA\to\ncatX\)  can be written as the composite \(\ncatX(F-,y)=y^{*}\cdot F_{*}\colon\ncatA\lmodto\ncatOne\). Moreover, for \(\ncatA\) small, any presheaf \(\phi \colon\ncatA\lmodto\ncatOne\) may be composed with \(F^{*}\colon\ncatX\lmodto\ncatA\) to yield \(\phi\cdot F^{*}\colon\ncatX\lmodto\ncatOne\), and \(-\cdot F^{*}\dashv -\cdot F_{*}\) is an adjunction between \([\ncatA^{\op},\nSetss{\V}]\) and the higher-universe \(\V\)-normed category \([\ncatX^{\op},\nSetss{\V}]\). Therefore we have natural isomorphisms
\begin{displaymath}
  \Nat(\phi\cdot F^{*},y^{*})\cong\Nat(\phi,y^{*}\cdot F_{*}).
\end{displaymath}

  From the discussion above we obtain:
\begin{lemma}
  An object \(x\) of \(\ncatX\) is a \(\phi\)-weighted colimit of \(F\) if, and only if, \(x\) is a \((\phi\cdot F^{*})\)-weighted colimit of the identity functor \(\ncatX\to\ncatX\), and in that case we simply speak of a \((\phi\cdot F^{*})\)-weighted colimit in \(\ncatX\).
\end{lemma}

For a $\V$-normed category $\ncatX$ we consider the Yoneda embedding
\begin{displaymath}
  \yoneda_{\ncatX}\colon\ncatX \to
  [\ncatX^{\op},\nSetss{\V}],\quad x\mapsto x^{*}=\ncatX(\text{-},x),
\end{displaymath}
whose codomain (irrespective of potentially having to be formed in a higher universe) is $\V$-normed again. Actually, $\yoneda_{\ncatX}$ preserves norms since, for every $f:x\to y$ in $\ncatX$, one has $$|\yoneda_{\ncatX}f|=|\ncatX(\text{-},f)|=\bigwedge_{z\in\ncatX}|\ncatX(z,f)|=\bigwedge_{h:z\to x}[|h|,|f\cdot h|]=|f|.$$

  \begin{definition}
    For every \(\V\)-normed category \(\ncatX\), let \(\Psh\ncatX\) denote the full $\V$-normed subcategory of \([\ncatX^{\op},\nSetss{\V}]\) defined by all \emph{accessible} presheaves (see \cite{KS05}).
  \end{definition}

By definition, the accessible presheaves are the small-weighted colimits of representables. Viewing a presheaf \(\phi \colon\ncatX^{\op}\to\nSetss{\V}\) as a distributor \(\phi \colon\ncatX\lmodto\ncatOne\), this means that \(\phi\) belongs to \(\Psh\ncatX\) if, and only if, there is a fully faithful $\V$-normed functor \(F \colon\ncatA\to\ncatX\) with \(\ncatA\) small and a distributor \(\psi \colon\ncatA\lmodto\ncatOne\) with \(\phi=\psi\cdot F^{*}\). Of course, for \(\ncatX\) small, one has
  \(\Psh\ncatX= [\ncatX^{\op},\nSetss{\V}]\).

\begin{proposition}
  For every \(\V\)-normed category \(\ncatX\), the $\V$-normed category  \(\Psh\ncatX\) is Cauchy cocomplete. Moreover, for every \(\V\)-normed functor \(F \colon\ncatX\to\ncatY\), the \(\V\)-normed functor
  \begin{displaymath}
    \Psh F \colon \Psh\ncatX \longrightarrow\Psh\ncatY,\;
    \phi \longmapsto \phi\cdot F^{*},
  \end{displaymath}
  preserves normed colimits of sequences.
\end{proposition}
\begin{proof}
  Let \(\sigma\colon\NN\to \Psh\ncatX\) be a Cauchy sequence in \(\Psh\ncatX\). Since \(\NN\) is a (countable) set, there is a small full $\V$-normed subcategory \(\ncatA\) of \(\ncatX\) (with inclusion functor \(I \colon\ncatA\to\ncatX\)) such that \(\sigma\) factors as
 \begin{displaymath}
   \begin{tikzcd}[column sep=large]
     \NN %
     \ar{dr}{\sigma} %
     \ar{d}[swap]{\sigma_{0}}\\
     \Psh\ncatA %
     \ar{r}[swap]{\Psh I=-\cdot I^{*}} %
     & \Psh\ncatX, %
   \end{tikzcd}
 \end{displaymath}
 and \(\sigma_{0}\colon\NN\to\Psh\ncatA\) is Cauchy in \(\Psh\ncatA\). By Theorem~\ref{d:thm:1}, \(\Psh\ncatA\) is Cauchy cocomplete; we let \(Q\)  with cocone $\Gamma$ be a normed colimit of \(\sigma_{0}\) in \(\Psh\ncatA\). By Proposition~\ref{d:prop:1}, \(\Psh I(Q)\)  with $\Psh I\cdot\Gamma$  is a normed colimit of \(\sigma\) in \(\Psh\ncatX\). Finally, since \(\Psh F\cdot\Psh I\) is left adjoint, \(\Psh F(\Psh I(Q))\) with $\Psh F\cdot\Psh I\cdot\Gamma$ is a normed colimit of \(\Psh F\cdot\sigma\).
\end{proof}

For a Cauchy sequence \(s\) in \(\ncatX\), we let \(\phi_{s}\) denote the normed colimit of \(\yoneda_{\ncatX}\cdot s\) in \(\Psh\ncatX\). Then, for every object \(y\) in \(\ncatX\), the cocone
\begin{displaymath}
  \Nat(s|_{N},\Delta y)\cong
  \Nat(\yoneda_{\ncatX}\cdot s|_{N},\Delta\yoneda_{\ncatX}y)
  \longrightarrow \Nat(\phi_s,\yoneda_{\ncatX}y),\;{N\in\NN},
\end{displaymath}
is a colimit in \(\Setss{\V}\).

\begin{proposition}\label{normedvsweightedprop}
  Let \(s\) be a Cauchy sequence in \(\ncatX\). Then \(\ncatX\) has a normed colimit of \(s\) if, and only if, \(\ncatX\) has a \(\phi_{s}\)-weighted colimit.
\end{proposition}
\begin{proof}
  Assume first that \(x\) with cocone $\gamma$ is a normed colimit of \(s\). Then, for every object $y\in\ncatX$,  the cocone
 $ \kappa_N:  \Nat(s|_{N},\Delta y)\to \ncatX(x,y)\;(N\in\NN)$
  is a colimit in \(\Setss{\V}\). Therefore we obtain an isomorphism
  \begin{equation}\label{d:eq:1}
    \ncatX(x,y)\longrightarrow \Nat(\phi_{s},\yoneda_{\ncatX}y),
  \end{equation}
  naturally in \(y\)  which exhibit \(x\) as a \(\phi_{s}\)-weighted colimit.

  Conversely, given the natural isomorphisms \eqref{d:eq:1}, then one obtains a colimit cocone
  \begin{displaymath}
  \kappa_N:  \Nat(s|_{N},\Delta y)\longrightarrow \ncatX(x,y)\;(N\in\NN)
  \end{displaymath}
   in \(\Setss{\V}\). Therefore \(x\)  with the induced cocone $\gamma$ becomes a normed colimit of \(s\) in \(\ncatX\).
\end{proof}

We conclude that normed colimits of Cauchy sequences are equivalently described as certain weighted colimits. Below we explain that it suffices to consider countable diagram shapes; here we call a ($\V$-normed) category \(\ncatX\) \emph{countable} whenever the class of arrows of \(\ncatX\) is a countable set.

\begin{facts}
  Let \(\sigma \colon\NN\to\ncatX\) be a Cauchy sequence in a $\V$-normed category \(\ncatX\), and let \(\phi \colon\ncatX\lmodto\ncatOne\) be a colimit of \(\yoneda_{\ncatX}\cdot\sigma\) in \(\Psh\ncatX\).
  \begin{enumerate}
  \item Consider the $\V$-normed subcategory \(\ncatA\) of \(\ncatX\) generated by the image of \(\sigma\); that is, the objects of \(\ncatA\) are given by the objects \(\sigma(n)\) (\(n\in\NN\)), and the arrows are finite composites of arrows of the form \(\sigma(n\leq m)\), with inclusion functor \(I \colon\ncatA\to\ncatX\). By construction, \(\ncatA\) is countable. Moreover, with \(\sigma_{0}\colon\NN\to\ncatA\) denoting the sequence in \(\ncatA\) with \(I\cdot\sigma_o=\sigma\), also \(\sigma_0\) is a Cauchy sequence in \(\ncatA\). Letting \(\phi_{0}\colon\ncatA\lmodto\ncatOne\) be the normed colimit of the Cauchy sequence \(\yoneda_{\ncatA}\cdot \sigma_0\) in  \(\Psh\ncatA\), we have \(\phi=\Psh I(\phi_0)=\phi_0\cdot I^{*}\). Therefore, \(\ncatX\) has a \(\phi\)-weighted colimit if, and only if, \(\ncatX\) has a \(\phi_{0}\)-weighted colimit of \(I \colon\ncatA\to\ncatX\).
  \item Consider \(\NN\) just as an ordinary category (given by its order). By Proposition \ref{CatVproperties}, $\NN$ may be equipped with the initial normed structure 
  with respect to the ordinary functor \(\sigma \colon\NN\to\ncatX\) and the given norm of $\ncatX$. Then, since $\sigma$ is Cauchy in $\ncatX$, the sequence
    $  0\leq 1\leq 2\dots $
    becomes Cauchy in \(\NN\) as well, and we can form the normed colimit \(\phi_0\) of \(\yoneda_{\NN}\) in \(\Psh\NN\). With $\phi$ defined as above, \(\ncatX\) has a \(\phi\)-weighted colimit if, and only if, \(\ncatX\) has a \(\phi_{0}\)-weighted colimit of \(I \colon\NN\to\ncatX\).
  \end{enumerate}
\end{facts}

All told, we have the following characterization of Cauchy cocompleteness.

\begin{corollary}\label{d:cor:1}
  Let \(\ncatX\) be a \(\V\)-normed category. Then the following assertions are equivalent.
  \begin{tfae}
  \item[{\em(i)}] \(\ncatX\) is Cauchy cocomplete.
  \item[{\em (ii)}]\label{d:item:1} \(\ncatX\) has all weighted colimits of diagrams \(F\colon\ncatA\to\ncatX\), \(\phi \colon\ncatA\lmodto\ncatOne\), where $\ncatA$ is small and \(\phi\) is a normed colimit of a Cauchy sequence of representables in \(\Psh\ncatA\).
  \item[{\em(iii)}]  \(\ncatX\) has all weighted colimits of diagrams \(F\colon\ncatA\to\ncatX\), \(\phi \colon\ncatA\lmodto\ncatOne\), where \(\ncatA\) is countable and \(\phi\) is a normed colimit of a Cauchy sequence of representables in \(\Psh\ncatA\).
  \item[{\em(iv)}]  \(\ncatX\) has all weighted colimits of diagrams \(F\colon\NN\to\ncatX\), \(\phi \colon\NN\lmodto\ncatOne\), where the underlying category of \(\NN\) is given by its order and \(\phi\) is a normed colimit of a Cauchy sequence of representables in \(\Psh\NN\).
  \end{tfae}
\end{corollary}

\section{Cauchy cocompletion of $\V$-normed categories}

In the previous section we have shown that normed colimits of Cauchy sequences can be equivalently described as weighted colimits, for a certain choice of weights -- under the assumption that the quantale $\V$ satisfies condition (A) \emph{or} (B), which we also assume in this section. Following the nomenclature of \cite{KS05,AK88}, for every small \(\V\)-normed category \(\ncatA\) we consider the class \(\Phi[\ncatA]\) of presheaves \(\phi\in \Psh\ncatA\) that are normed colimits of Cauchy sequences of representables  as in Corollary~\ref{d:cor:1}~\ref{d:item:1}, and put
\begin{displaymath}
  \Phi=\sum_{\ncatA\text{ small}}\Phi[\ncatA].
\end{displaymath}
A \(\V\)-normed category \(\ncatX\) is called \emph{\(\Phi\)-cocomplete} whenever, for all \(\V\)-normed functors \(F \colon\ncatA\to\ncatX\) and \(\phi \colon\ncatA^{\op}\to\nSetss{\V}\) with \(\ncatA\) small and \(\phi\in\Phi[\ncatA]\), the \(\phi\)-weighted colimit of \(F\) in \(\ncatX\) exists. Moreover, a \(\V\)-normed functor is called \emph{\(\Phi\)-cocontinuous} whenever it preserves all weighted colimits with weight in \(\Phi\).

By Corollary~\ref{d:cor:1}, a normed category \(\ncatX\) is Cauchy cocomplete if, and only if, \(\ncatX\) is \(\Phi\)-cocomplete. Furthermore, since the diagram
\begin{displaymath}
  \begin{tikzcd} 
    \ncatX %
    \ar{r}{F} %
    \ar{d}[swap]{\yoneda_{\ncatX}} %
    & \ncatY %
    \ar{d}{\yoneda_{\ncatY}} \\
    \Psh\ncatX %
    \ar{r}[swap]{\Psh F} %
    & \Psh\ncatY %
  \end{tikzcd}
\end{displaymath}
commutes (up to isomorphism) for every \(\V\)-normed functor \(F \colon\ncatX\to\ncatY\), using the notation of Proposition \ref{normedvsweightedprop} and writing ncolim instead of just colim to stress the normedness of a colimit, for every  Cauchy sequence \(s \colon\NN\to\ncatX \) we have
\begin{displaymath}
  PF(\phi_{s})
  = PF(\ncolim(\yoneda_{\ncatX}\cdot s))
  \cong\ncolim(\yoneda_{\ncatY}\cdot F\cdot s)
  =\phi_{F\cdot s}.
\end{displaymath}
Therefore, \(F\) preserves normed colimits of Cauchy sequences if, and only if, \(F\) is \(\Phi\)-cocontinuous. 

We  denote by
$  \Cocts{\Phi}$
the 2-category of \(\Phi\)-cocomplete small \(\V\)-normed categories, \(\Phi\)-cocontinuous \(\V\)-normed functors, and their natural transformations, and write
 $ \COCTS{\Phi}$
for its higher universe counterpart. For every \(\V\)-normed category \(\ncatX\), we let \(\Phi(\ncatX)\) denote the smallest replete full \(\V\)-normed subcategory of \(\Psh\ncatX\) containing \(\ncatX\) and closed under \(\Phi\)-colimits. Then the Yoneda functor of \(\ncatX\) restricts to
$  \yoneda_{\ncatX}\colon\ncatX\to\Phi(\ncatX),$
and we have that  \(\Phi(\ncatX)\) is \(\Phi\)-cocomplete and the inclusion functor  \(\Phi(\ncatX)\to\Psh\ncatX\) is \(\Phi\)-cocontinuous.
 We now show that \(\Phi(\ncatX)\) serves as a correct-size Cauchy cocompletion of the $\V$-normed category $\ncatX$, both for small and large $\ncatX$.
\begin{lemma}
  For each small normed category \(\ncatX\), the presheaf category \(\Psh\ncatX\) is small.
\end{lemma}
\begin{proof}
  Consider
  \begin{displaymath}
    \Phi_{0}=\sum_{\ncatA\text{ countable}}\Phi[\ncatA],
  \end{displaymath}
hence \(\Phi_0\) is small. By Corollary~\ref{d:cor:1}, \(\Phi(\ncatX)=\Phi_{0}(\ncatX)\) for every normed category \(\ncatX\). By \cite[Section~7]{AK88}, \(\Phi_0(\ncatX)\) is small.
\end{proof}

\begin{theorem}[Proposition~3.6 in \cite{KS05}]
  For every \(\V\)-normed category \(\ncatX\) and every Cauchy cocomplete \(\V\)-normed category \(\ncatY\), the composition with \(\yoneda_{\ncatX}\colon\ncatX\to\Phi(\ncatX)\) defines an equivalence
  \begin{displaymath}
    \COCTS{\Phi}(\Phi(\ncatX),\ncatY)\to\CATss{\V}(\ncatX,\ncatY);
  \end{displaymath}
  that is, \(\Phi(-)\) provides a left biadjoint to the inclusion 2-functor \(\Cocts{\Phi}\to\CATss{\V}\). This equivalence restricts to
    \begin{displaymath}
    \Cocts{\Phi}(\Phi(\ncatX),\ncatY)\to\Catss{\V}(\ncatX,\ncatY),
  \end{displaymath}
  when \(\ncatX\) and \(\ncatY\) are small.
\end{theorem}

  \section{The Banach Fixed Point Theorem for normed categories}

At first, letting the quantale $\V$ remain general, but then specializing it to $\V=\mathcal R_+$, we consider a $\V$-normed category $\mathbb X$ and a $\V$-normed endofunctor $F$ of $\mathbb X$ and give sufficient conditions guaranteeing the existence of an object $x$ with $x\cong Fx$, in such a way that they reproduce Banach's Fixed Point Theorem when $\mathbb X=\mathrm iX$ for a metric space $X$. The following terminology makes precise what $x\cong Fx$ may mean in the $\V$-normed context.

\begin{definition}
For a $\V$-normed functor $F:\mathbb X\to\mathbb X$, we say that an object $x$ in $\mathbb X$ is
\begin{itemize}
\item a {\em forward fixed point of $F$} if there is an isomorphism $f:x\to Fx$ of (the ordinary category) $\mathbb X$ with $\kk\leq |f|$;
\item a {\em backward fixed point of $F$} if there is an isomorphism $f:Fx\to x$ of (the ordinary category) $\mathbb X$ with $\kk\leq |f|$;
\item a {\em fixed point of $F$} if there is an isomorphism $f:x\to Fx$ in $\mathbb X_{\circ}$.
\end{itemize}
\end{definition}

\begin{facts}
(1) Trivially, a fixed point of $F$ is both, a forward fixed point and a backward fixed point of $F$. By Facts \ref{thirdremarks}(2), if $\mathbb X$ is forward (backward) symmetric, every forward (backward, respectively) fixed point of $F$ is already a fixed point of $F$.

(2) Here is a normed functor $F$ of a normed category $\mathbb X$ in which every object is a forward fixed point of $F$, but which has no backward fixed point of $F$: consider $\mathbb X=\mathrm  iX$ for $X=\{0,1,2, ...\}$ with $X(m,n)=0$ for $m\leq n$ and $X(m,n)=1$ otherwise, and let $F$ be given by $Fn=n+1$ for all $n$. (Note that $F$ is even norm preserving.)

(3) For $X$ as in (2), considering $\mathbb X=\mathrm{i}(X\otimes X^{\mathrm{op}})$ and its normed endofunctor $F\otimes F^{\mathrm{op}}$, we have simultaneously (many) forward and (many) backward fixed points of $F\otimes F^{\mathrm{op}}$, but no fixed point.
\end{facts}

For a $\V$-normed endofunctor $F:\mathbb X\to\mathbb X$, let us first consider any morphism $f:x\to Fx$ and, following the standard procedure of ordinary categorical fixed point theory, form the {\em iteration sequence} $s_f$ of $f$:
$$\xymatrix{x\ar[r]^f & Fx\ar[r]^{Ff} & F^2x\ar[r]^{F^2f} & F^3x\ar[r]^{F^3f}& ... \\}.$$
Assuming that, at the ordinary category level, there is a colimit cocone $\gamma_f:s_f\to \Delta y$ in $\mathbb X$, we obtain a comparison morphism $\overline{f}:y\to Fy$ with $\Delta\overline{f}\cdot\gamma_f=F\gamma_f$, which is an isomorphism precisely when the (ordinary) functor $F$ preserves the colimit. Assuming further that $\gamma_f$ actually exhibits $y$ as a {\em normed} colimit of $s_f$, since $F$ is $\V$-normed, in the terminology of Lemma \ref{kcocone} not only $\gamma_f$ must be $\kk$-cocone, but also $F\gamma_f$, so that with property (C2b) of Corollary \ref{Condition2ab} one concludes that $\overline{f}$ must be $\kk$-morphism:
$$|\overline{f}|\geq\bigvee_N\bigwedge_{n\geq N}|\overline{f}\cdot(\gamma_f)_n|=\bigvee_N\bigwedge_{n\geq N}|F(\gamma_f)_n|\geq\bigvee_N\bigwedge_{n\geq N}|(\gamma_f)_n|\geq\kk\;.$$
 Furthermore, if  $F$ preserves $y$ (with $\gamma_f$) as a normed colimit of $s_f$, then trivially also $\overline{f}^{-1}$ must be a $\kk$-morphism. This normed preservation of the colimit is particularly guaranteed if $\mathbb X$ is forward or backward symmetric, since  with Facts \ref{thirdremarks}(2) and the proof of Proposition \ref{ncolimunderconditionS} we again obtain that  $\overline{f}^{-1}$ is a $\kk$-morphism.

In summary, we proved:

\begin{proposition}\label{Banachprop}
 Let $F:\mathbb X\to\mathbb X$  be a $\V$-normed functor preserving ordinary colimits of sequences, and let $f:x\to Fx$ be a morphism for which the iterated sequence $s_f$ has a normed colimit $y$ in $\mathbb X$. Then $y$ is a forward fixed point of $F$ in $\mathbb X$, and it is even a fixed point if $F$ preserves the colimit $y$ as a normed colimit, in particular if  $\mathbb X$ is forward or backward symmetric.
\end{proposition}

We now consider $\V=\mathcal R_+$ and provide a sufficient condition {\em \`{a} la} Banach for the existence of a normed colimit of the iterated sequence of a morphism $x\to Fx$, for a {\em contractive} functor $F:\mathbb X\to \mathbb X$, so that there is a (non-negative) Lipschitz factor $L< 1$, {\em i.e.}, $|Fh|\leq L |h|$ for all morphisms $h$ in $\mathbb X$.

\begin{theorem}\label{Banachthm}
Let $\mathbb X$ be a Cauchy cocomplete normed category, and let $F:\mathbb X\to \mathbb X$ be a contractive functor which preserves (ordinary) colimits of sequences. Then, if $\mathbb X$ contains any morphism $f:x\to Fx$ with $|f|<\infty$, then $\mathbb X$ contains a forward fixed point of $F$, and even a fixed point of $F$ if $F$ preserves normed colimits of Cauchy sequences; in particular, if $\mathbb X$ is forward or backward symmetric.
\end{theorem}

\begin{proof}
In light of the Proposition, it suffices to show that the iterated sequence $s_f$ of the given morphism $f$ with $|f|<\infty$ is Cauchy. This, however, follows just like in the classical case of a contraction of a metric space from the Cauchyness of the geometric series given by $L$: indeed, for all $m\leq n$ one has
$$|(s_f)_{m,n}:F^mx\to F^nx|= |F^{n-1}f\cdot ...\cdot F^mf|\leq(L^{n-1}+\dots+L^m)|f|\;.                  $$
\end{proof}

\begin{remarks}
(1) The classical Banach Fixed Point Theorem for the contraction $\varphi$ of a (non-empty) complete (classical) metric space $X$ follows when we consider $\mathbb X=\mathrm i X$ and $F=\mathrm i\varphi$.

(2) One cannot expect the uniqueness statement for fixed points in the classical metric case to extend {\em verbatim} to normed categories, not even for Lawvere metric spaces: just consider the coproduct in $\mathsf{Met}_1$ of two copies of the Euclidean line. However, the classical uniqueness is an obvious consequence of the following general statement: Suppose we are given a forward fixed point $x$ and a backward fixed point $y$ of the contraction $F:\mathbb X\to\mathbb X$, with the property that the minimum of $\{|h|\mid h:x\to y \text{ in }\mathbb X\}$ exists and is positive;
 then such minimal morphism $h_0$ must be a $0$-isomorphism. Indeed, since we have isomorphisms $f:x\to Fx$ and $g:Fy\to y$ with $|f|=0=|g|$, the minimality of $|h_0|$ forbids $|h_0|>0$, as this would imply
$$|h_0|\leq|g\cdot Fh_0\cdot f|\leq |g|+|Fh_0|+|f|=|Fh_0|<|h_0|\;.$$
(3) Theorem \ref{Banachthm} improves Kubi\'{s}'s Corollary 4.2 in \cite{Kubis17}, since the normed sequential colimits considered there are not necessarily unique up to $0$-isomorphism: see Facts \ref{thirdremarks}(3). Actually, we have not been able to establish a valid proof of Kubi$\check{\mathrm s}$'s version of the Banach Fixed Point Theorem since, in the absence of Condition (C2b), one cannot argue as in our proof of Proposition \ref{Banachprop}.

(4) Under fairly general conditions on a quantale $\V$, the paper \cite{BenkhadraStubbe22} provides an in-depth study of fixed points of a $\V$-endofunctor $F$ in the specialized situation $\ncatX=\mathrm iX$ for a $\V$-category $X$, with the contractivity of $F$ expressed by an accompanying ``control function'' $\varphi:\V\to \V$. Earlier works in this context include \cite{KosWas11}, preceded by \cite{Was10}.

 We must leave the question of how the approach of  \cite{BenkhadraStubbe22} may be generalized to arbitrary $\V$-normed categories and functors for future work.  
\end{remarks}

\section{Appendix: Local presentability of $\Catss{\V}$}
We first want to show that the category $\Setss{\V}$ is locally presentable (\cite{GU71, AR94}) and revisit its strong generator $\{\mathrm E_v\mid v\in\V\}$ (see Facts \ref{firstremarks}(4)).

\begin{lemma}\label{locpreslemma}
For the least regular cardinal $\lambda$ larger than the size of the set \(\V\), and for every element \(v\in\V\), the $\V$-normed set $\mathrm E_v$ whose only element has norm $v$ is \(\lambda\)-presentable in \(\Setss{\V}\), that is: the functor $P_v:\Setss{\V}\to\SET$  represented by $E_v$ preserves $\lambda$-directed colimits.
\end{lemma}

\begin{proof}
 For a $\lambda$-directed ordered set $I$ (so that any subset of size $<\lambda$ has an upper bound in $I$), we consider an $I$-indexed diagram $(f_{i,j}:A_i\to A_j)_{i\leq j}$ with colimit cocone  $(g_i:A_i\to B)_{i\in I}$ in $\Setss{\V}$. We must show that every morphism $b:\mathrm E_v\to B$ in $\Setss{\V}$  has an essentially unique factorization through some $g_i$. But such morphism is described by an element $b\in B$ with $v\leq |b|$, and since $B$ carries the final structure with respect to the colimit cocone, we have $|b|=\bigvee C_b$ with
  \begin{displaymath}
    C_b=\{|{a}|\mid \exists i\in I: a\in g_i^{-1}b\}\subseteq\V.
  \end{displaymath}
  For every \(u\in C_b\) we choose an index \(i_u\in I\) and an element $a_u\in g_{i_u}^{-1}b$ with  \(|{a_u}|=u\). Since the terminal object $\{*\}$ is \(\lambda\)-presentable in \(\SET\) and \(I\) is \(\lambda\)-directed, there are \(j\in I\) and \(a\in A_j\) so that, for all \(u\in  C_b\), one has \(i_u\le j\) and \(f_{i_u,j}(a_u)=a\). Therefore \(g_{j}(a)=b\) and, for all \(u\in C_b\), \(u\le|a|\). Consequently \(v\le|a|\), that is, the map \( \mathrm E_v\to A_j\) with value $a$ is actually a morphism in \(\Setss{\V}\), and we conclude that \(b \colon\mathrm E_v\to B\) factors through \(g_j:A_j\to B\). The essential uniqueness of this factorization follows from the fact that a singleton set is \(\lambda\)-presentable in \(\SET\).
\end{proof}

\begin{proposition}\label{SetVlocpres}
 The category \(\Setss{\V}\) is locally presentable.
\end{proposition}
\begin{proof}
The $\SET$-based topological category  \(\Setss{\V}\) is cocomplete, and every object of its strong generator $\{\mathrm E_v\mid v\in\V\}$ is locally $\lambda$-presentable, with $\lambda$ as in Lemma \ref{locpreslemma}.
\end{proof}

Since $\Setss{\V}$ is locally presentable, by the general result shown in \cite{KL01} this important property gets inherited by $(\Setss{\V})\text{-}\mathsf{Cat}$:
\begin{corollary}
 The category \(\Catss{\V}\) is locally presentable.
\end{corollary}

\section{Appendix: Comparison with idempotent completeness}

For a general $\V$-normed category $\mathbb X$ we briefly examine the question of whether constant sequences in $\mathbb X$ are Cauchy and have a normed colimit in $\mathbb X$.
Here a sequence $s:\mathbb N\to \mathbb X$ is understood to be {\em constant} if $s_{m,n}=e:x\to x$ for all $m<n$ in $\mathbb N$. Such morphism $e$ must necessarily be idempotent in the ordinary category $\mathbb X$, and every idempotent morphism defines a constant sequence. Recall that the idempotent $e$ {\em splits} in $\mathbb X$ if $e= t\cdot r$ for some morphisms $r, t$ with $r\cdot t=1$ (which already exist when $\mathbb X$ has epi-mono factorizations, or equalizers, or coequalizers). Such factorization of $e$ is unique, up to a uniquely determined isomorphism.

\begin{lemma}
The constant sequence given by an idempotent $e:x\to x$ has a colimit in the ordinary category $\ncatX$ if, and only if, $e$ splits.
\end{lemma}

\begin{proof}
Given a colimit cocone $\rho_n:x\to y\,(n\in\mathbb N)$ of the constant sequence defined by $e$, one has
$$\rho_0=\rho_1\cdot e=\rho_2\cdot e\cdot e=\rho_2\cdot e =\rho_1$$
and, inductively, $\rho_1=\rho_2=\dots=:r$. The colimit cocone makes this morphism epic. Furthermore, the cocone $\eta$ with $\eta_n=e$ for all $n$ corresponds to a morphism $t:y\to x$ with $t\cdot r=e$, which also satisfies $r\cdot t=1_y$ since $r\cdot t\cdot r=r\cdot e=r$.

Conversely, given the splitting $r,t$ of $e$, the cocone $\rho$ with $\rho_n:=r:x\to y$ for all $n$ which we call {\em related to the splitting}, exhibits $y$ as a colimit of the constant sequence $e$:
since any cocone $\alpha:s\to\Delta z$ is easily seen to satisfy $\alpha_0\cdot t=\alpha_n\cdot t$, we obtain $(\alpha_0\cdot t)\cdot \rho_n=\alpha_n$ for all $n\in \mathbb N$, and furthermore, any morphism $f:y\to z$ with $\Delta f\cdot \rho =\alpha$ necessarily satisfies $f\cdot r=\alpha_0$, so that $f=\alpha_0\cdot t$.
\end{proof}

Recall that
 $\mathbb X$ (as an ordinary category) is said to be {\em idempotent complete}\footnote{We adopt here the terminology used in the recent paper \cite{HemelaerRogers24}. Other terms used in the literature are Karoubi complete, Lawvere complete or, most frequently, Cauchy complete. We avoid the latter term, not to risk confusion with the dualization of our term of Cauchy cocompleteness for normed categories which is far more directly modelled after Cauchy's original ideas than idempotent completeness (of any flavour) is. Besides, as  a concept that gained its recognition through various important contributions in different contexts, it may indeed be difficult to attach just one person's name to idempotent completeness.} if all idempotents split in $\mathbb X$. Idempotent completeness of the category $\mathbb X_{\circ}$ suffices to provide an affirmative answer to the question raised at the beginning of this section. More precisely:

\begin{proposition}
The constant sequence $s$ in a $\V$-normed category $\mathbb X$ given by an idempotent morphism $e$ is Cauchy precisely when $e$ is a $\kk$-morphism. In this case, the constant cocone related to a given splitting $e=r\cdot t$ of $e$ in $\mathbb X$ gives a normed colimit of $s$ in $\mathbb X$ if, and only if, the morphisms $r$ and $t$ are  both $\kk$-morphisms. \end{proposition}

\begin{proof}
The first claim is obvious. Also, trivially the constant cocone $\rho$ related to the splitting $r, t$ of $e$ is a $\kk$-cocone if, and only if, $r$ is a $\kk$-morphism. Since $\rho$ is an ordinary colimit cocone, for the proof of the second claim, assuming $\kk\leq |r|$, we just need to show that $\kk\leq |t|$ holds if, and only if,  $$\bigvee_{N\in\mathbb N}\bigwedge_{n\geq N}|f\cdot \rho_n|=|f\cdot r|\leq|f|$$ for all $f:y\to z$ in $\mathbb X$. Indeed, from $\kk\leq|t|$ one obtains
$|f\cdot r|\leq|f\cdot r|\otimes|t|\leq|f\cdot r\cdot t|=|f|\;,$
and conversely, exploiting this inequality for $f=t$, from $\kk\leq|e|$ we obtain $\kk\leq|t\cdot r|\leq |t|$.
\end{proof}

\begin{corollary}\label{ordcatCauchycomplete}
For a $\V$-normed category $\mathbb X$, the category $\mathbb X_{\circ} $ is idempotent complete if, and only if, every constant Cauchy sequence in $\mathbb X$ has a normed colimit.
\end{corollary}

\section{Appendix: Condition A vs. Condition B}\label{AvsB}
For a (unital and commutative) quantale $\V=(\V,\leq,\otimes,\kk)$, we show the logical independence of conditions (A) and (B) of Section \ref{SetVCauchycoc}, {\em i.e.}, of the conditions
$$\text{(A)}\quad\kk=\bigvee\{u\in\V \mid u\lll\kk\}\qquad\text{and}\qquad\text{(B)}\quad\kk\wedge(-):\V\to\V\quad\text{preserves arbitrary joins}.$$

$\text{(B)}\nRightarrow\text{(A)}$:

It suffices to find an integral quantale which does not satisfy (A) (see Remarks \ref{fifthremarks}(2)). This is not hard; for example, for any infinite set $X$, consider the cofinite topology $\mathcal{O}(X)$ on $X$ (so that a non-empty subset $U\subseteq X$ is open precisely when $X\setminus U$ is finite) as a quantale ($\mathcal{O}(X),\subseteq,\cap, X)$. Then any open set $U$ with $U\lll X$ must be empty since, otherwise, the finiteness of $X\setminus U$ makes the infinite set $X$ satisfy
$X=\bigcup_{x\in U}X\setminus\{x\},$
whereas no $x\in U$ allows $U\subseteq X\setminus\{x\}$. Consequently, (A) is violated in $\mathcal O(X)$.

$\text{(A)}\nRightarrow\text{(B)}$ (see \cite{GH24}):

Consider the $3$-element cyclic group $Z_3=(\{0,1,2\},+)$ as a discretely ordered set. Its MacNeille completion adds the top and bottom elements $\top$ and $\bot$ to it, giving us the $5$-element diamond lattice $M_3$, with atoms $0,1,2$. This lattice becomes a quantale as in Examples \ref{otherexamples}(2). 
The quantalic unit $0$ makes (B) fail in $M_3$, but (A) fails as well. To overcome the failure of (A), one extends the lattice $M_3$ by two new elements, $\kk$ and
$\overline{\top}$, and obtains the desired $7$-element quantale $\overline{M}_3$ satisfying (A) but not (B). Indeed, its quantalic operation $\otimes$ extends the tensor product of $M_3$ and makes $\kk$ a new tensor-neutral element in $\overline{M}_3$, above only $0$ and $\bot$, while $\overline{\top}$ becomes a new top element in $\overline{M}_3$; tensoring by $\overline{\top}$ is defined by $\overline{\top}\otimes\overline{\top}=\overline{\top}$ and $ \overline{\top}\otimes\alpha=\top$ for $\alpha\in\{\top,0,1,2\}$. The element $\kk$ is trivially join prime in $\overline{M_3}$ ({\em i.e.}, $\kk\lll\kk$), but $\kk\wedge-$ does not distribute over the join $1\vee 2$.

{\small
\hspace*{20mm}
$
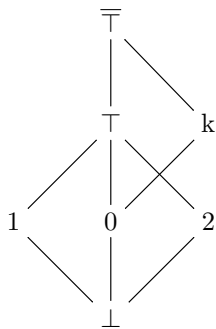
\xymatrix{ & \overline{\top}\ar@{-}[d]\ar@{-}[rd] & \\
& \top\ar@{-}[ld]\ar@{-}[d]\ar@{-}[rd] & \kk\ar@{-}[ld]\\
1\ar@{-}[rd] & 0\ar@{-}[d] & 2\ar@{-}[ld] \\
& \bot}$

\vspace*{-50mm}\hspace*{80mm}
\begin{tabular}[t]{|c||c|c|c|c|c|c|c|}
\hline
$\otimes$&$\overline{\top}$&$\kk$&$\top$&0&1&2&$\bot$\\
\hline\hline
$\overline{\top}$&$\overline{\top}$&$\overline{\top}$&$\top$&$\top$&$\top$&$\top$&$\bot$\\
\hline
$\kk$&$\overline{\top}$&$\kk$&$\top$&0&1&2&$\bot$\\
\hline
$\top$&$\top$&$\top$&$\top$&$\top$&$\top$&$\top$&$\bot$\\
\hline
0&$\top$&0&$\top$&0&1&2&$\bot$\\
\hline
1&$\top$&1&$\top$&1&2&0&$\bot$\\
\hline
2&$\top$&2&$\top$&2&0&1&$\bot$\\
\hline
$\bot$&$\bot$&$\bot$&$\bot$&$\bot$&$\bot$&$\bot$&$\bot$\\
\hline
\end{tabular}

\vspace*{10mm}}

\begin{remarks}
(1) Following a question posed at the occasion of the second-named author's presentation of some of the results of this paper at the Portuguese Category Seminar in October 2023, a first witness for $\text{(A)}\nRightarrow\text{(B)}$ was communicated  to the authors shortly afterwards by Javier Guti\'{e}rrez-Garc\'{i}a. His subsequent paper \cite{GH24} with Ulrich H\"{o}hle comprehensively analyzes and characterizes many types of quantales with $\text{(A)}\nRightarrow\text{(B)}$, even in the context of not necessarily commutative or unital quantales. The above example of a quantale witnessing $\text{(A)}\nRightarrow\text{(B)}$ is smallest with that property, but there are other such $7$-element quantales. Remarkably, as mentioned in Remark \ref{ABremark} and shown in \cite{GH24}, the procedure of finding counter-examples as sketched here starting with the group $Z_3$ may be followed with any group $G$ of at least 3 elements instead.

(2) There are many infinite topological spaces $X$ (other than those carrying the cofinite topology) such that $\mathcal O(X)$ witnesses $\text{(B)}\nRightarrow\text{(A)}$, but none of them can be Alexandroff (so that $\mathcal O(X)$ would be closed under arbitrary intersection). Indeed, one easily verifies that the quantale $\mathcal O(X)$ is completely distributive for every Alexandroff space $X$.
Other types of integral quantales satisfying (B) but violating (A) include those complete MV-algebras whose underlying lattice fails to be completely distributive. Indeed, by Proposition 3.13 of \cite{GHK22}, such MV-algebras must fail condition (A).
\end{remarks}

\end{document}